\documentclass[10pt,reqno,a4]{amsart}
\usepackage{graphicx}
\usepackage{setspace} 
\usepackage{amssymb, pifont, color, graphics, amsmath, amssymb, mathrsfs, amsfonts, multicol, amsmath, tikz, pgfplots, graphicx, lmodern,footnote} %, showlabels} 
\usepackage{tikz}
\usepackage[ruled,vlined]{algorithm2e}
\usepackage{hyperref,stackengine}
\hypersetup{
    colorlinks=true,
    linkcolor=blue,
    filecolor=red,      
    urlcolor=red,
}

\usepackage{mathtools}
\usepackage{esint}
\usepackage{lineno}

\makeatletter
\newcommand{\leqnomode}{\tagsleft@true\let\veqno\@@leqno}
\newcommand{\reqnomode}{\tagsleft@false\let\veqno\@@eqno}
\makeatother

\DeclareMathOperator*{\du}{d\!}
\DeclareMathOperator*{\Div}{div}

\DeclareMathOperator{\bu}{\boldsymbol{u}}
\DeclareMathOperator{\blf}{\boldsymbol{f}}
\DeclareMathOperator{\bw}{\boldsymbol{w}}
\DeclareMathOperator{\bv}{\boldsymbol{v}}
\DeclareMathOperator{\bV}{\mathbf{V}}
\DeclareMathOperator{\bH}{\mathbf{H}}
\DeclareMathOperator{\bW}{\mathbf{W}}

\DeclareMathOperator{\bTh}{\mathbf{\Theta}}
\DeclareMathOperator{\bth}{\boldsymbol{\theta}}
\DeclareMathOperator{\bn}{\boldsymbol{n}}
\DeclareMathOperator{\bN}{\boldsymbol{N}}
\DeclareMathOperator{\bg}{\boldsymbol{g}}

\DeclareMathOperator{\bphi}{\boldsymbol{\varphi}}
\DeclareMathOperator{\bpsi}{\boldsymbol{\psi}}
\DeclareMathOperator{\btau}{\boldsymbol{\tau}}
\DeclareMathOperator{\dive}{\mathrm{div}}

\DeclareMathOperator{\Oad}{\mathcal{O}_{ad}}
\DeclareMathOperator{\ws}{\mathrel{\ensurestackMath{\stackon[1pt]{\rightharpoonup}{\scriptstyle\ast}}}}

\theoremstyle{plain}
\newtheorem{thrm}{Theorem}[section]
\newtheorem{lmm}[thrm]{Lemma}

\newtheorem{prpstn}[thrm]{Proposition}

\theoremstyle{definition}
\newtheorem{dfntn}[thrm]{Definition}

\newtheorem{rmrk}[thrm]{Remark}

\makeatletter

\newenvironment{pethau}%
{\begin{list}{}%
    {%
      \setlength{\itemindent}{-10pt}%
      \setlength{\leftmargin}{20pt}%
      \setlength{\labelwidth}{.3\normalparindent}%
      \addtolength{\topsep}{-0.5\parskip}%
      \listparindent \normalparindent
      \setlength{\parsep}{\parskip}}}%
  {\end{list}}
\makeatother

\makesavenoteenv{tabular}
\makesavenoteenv{table}

\numberwithin{equation}{section}

\begin{document}
%%-----------------------------
%%      the top matter
%%-----------------------------
%\title{A Volume Preserving and Perimeter Regularized Shape Optimization Problem for the Vorticity Maximization of the Navier--Stokes Equations}\thanks{This work is supported by JSPS KAKENHI Grant Numbers JP18H01135, JP20H01823, JP20KK0058 and JP21H04431, and JST CREST Grant Number JPMJCR2014 for {\bf HN}; and by the Japanese Government (MEXT) Scholarship for {\bf JSS}.}% At most 5 thanks

\title[Maximizing Vortex for the Navier--Stokes Flow]{Maximizing Vortex for the Navier--Stokes Flow with a Convective Boundary Condition: A Shape Design Problem}\thanks{This work is supported by JSPS KAKENHI Grant Numbers JP18H01135, JP20H01823, JP20KK0058 and JP21H04431, and JST CREST Grant Number JPMJCR2014 for {\bf HN}; and by the Japanese Government (MEXT) Scholarship for {\bf JSS}.}% At most 5 thanks
\author{John Sebastian H. Simon}\address{Division of Mathematical and Physical Sciences, 
              Graduate School of Natural Science and Technology, 
              Kanazawa University, Kanazawa 920-1192, Japan; \texttt{john.simon@stu.kanazawa-u.ac.jp} }
\author{Hirofumi Notsu}\address{Faculty of Mathematics and Physics, Kanazawa University, Kanazawa 920-1192, Japan; \texttt{notsu@se.kanazawa-u.ac.jp}}
%
%\date{\today}
%
\begin{abstract}
In this study, a shape optimization problem for the two-dimensional stationary Navier--Stokes equations with an artificial boundary condition is considered. The fluid is assumed to be flowing through a rectangular channel, and the artificial boundary condition is formulated so as to take into account the possibility of ill-posedness caused by the usual do-nothing boundary condition.  The goal of the optimization problem is to maximize the vorticity of the said fluid by determining the shape of an obstacle inside the channel. Meanwhile, the shape variation is limited by a perimeter functional and a volume constraint.  The perimeter functional was considered to act as a Tikhonov regularizer and the volume constraint is added to exempt us from topological changes in the domain. The shape derivative of the objective functional was formulated using the rearrangement method,  and this derivative was later on used for gradient descent methods. Additionally, an augmented Lagrangian method and a class of solenoidal deformation fields were considered to take into account the goal of volume preservation.  Lastly, numerical examples based on the gradient descent and the volume preservation methods are presented.
\end{abstract}
%
%\begin{resume} ... \end{resume}
%
\subjclass[2020]{35Q30,49Q10, 49J20, 49K20}
\keywords{Navier--Stokes equations, convective boundary condition, shape optimization,  rearrangment method}
\maketitle
%%-----------------------------
%%      your text
%%-----------------------------
\section*{Introduction}

Partial differential equation (PDE) constrained optimization is among the most active fields in mathematics, mainly due to its applicability in engineering and physics.  In particular, the system of the Navier--Stokes equations is the theme of a lot of papers because of its ability to closely mimic physical fluid flow. From this reason, several authors including that of \cite{colin2010},  \cite{mihir1994}, and \cite{fursikov2000} have studied optimal control problems where the controls are designed to steer the fluid dynamics according to a given objective. 

{In this paper we are interested in an optimization problem where the control is the shape of the domain instead of functionals. Several authors have considered such problems due to its applicability for example in aeronautics\cite{mohammadi2004}.  In \cite{moubachir2006}, Moubachir, M. and Zolesio, J.-P.  wrote an extensive introductory literature for shape optimization problems which are governed by fluid flows and are mostly described through the Navier--Stokes equations, and the objectives are mostly formulated as tracking functionals.  For a more applied approach to fluid shape design problems, we refer the reader to \cite{pironneau2010}. Meanwhile, Kunisch, K. and Kasumba, H. \cite{kasumba2012} analyzed a vorticity minimization problem that steers the fluid flowing through a channel to exhibit more laminar flow. In the said literature, the authors considered a bounded domain, which compels the imposition of an input function on one end of the channel, and a do-nothing boundary condition on the other to capture the outflow. However,  due to the nonlinear nature of the Navier--Stokes equations, the outflow condition the authors considered can cause non-existence of weak solutions. For this reason, we propose an artificial boundary condition that can capture the outflow while making sure of the well-posedness of the state equations. }

We also mention that although most shape design problems involving fluid flow are formulated as to minimize vorticity or drag, there are literatures where vorticity maximization is the goal.  This type of maximization problems are mostly considered applicable, for example, in optimal mixing problems (c.f. \cite{eggl2020,mathew2007}).  Recently, Goto, K. et. al. \cite{goto2020} also showed that emergence of vortices has an interesting result in the field of information theory.

With all these,  we consider the following shape optimization problem
\begin{equation}
\min\{\mathcal{G}(\bu,\Omega): |\Omega| = m, \Omega\subset D \}, 
\end{equation}
where $D$ is a hold-all domain which is assumed to be a fixed bounded connected domain in $\mathbb{R}^2$, $\Omega\subset D$ is an open bounded domain, $\mathcal{G}(\bu,\Omega) : = J(\bu,\Omega) + \alpha P(\Omega),$ $J$ and $P$ are vortex and perimeter functionals respectively given by 
\[J(\bu,\Omega) = -\frac{\gamma}{2}\int_{\Omega}|\nabla\times \bu|^2\du x,\text{ and }P(\Omega) = \int_{\Gamma_{\rm f}}\du s,\]
where $m>0$ is a given constant, and $|\Omega|:= \int_\Omega \du x$.  Here, $\bu$ is the velocity field governed by a fluid flowing through a channel with an obstacle. The flow of the fluid is reinforced by an input function $\bg$ on the left end of the channel, which is denoted by $\Gamma_{\rm{in}}$, whilst an outflow boundary condition is imposed to the fluid on the right end of the channel denoted by $\Gamma_{\rm{out}}$. The boundary $\Gamma_{\rm f}$  is the boundary of the submerged obstacle in the fluid. The remaining boundaries of the channel are denoted by $\Gamma_{\rm{w}}$, upon which -- together with the obstacle boundary $\Gamma_{\rm f}$ -- a no-slip boundary condition is imposed on the fluid. { For the condition on $\Gamma_{\rm out}$, let us point out that the usual candidate for capturing outflow is the so-called {\it do-nothing} condition, however such condition is not sufficient to ensure the existence of solutions to the Navier--Stokes equations. Aside from the do-nothing condition, one formulation caught our attention. An artificial boundary condition called the {\it directional do-nothing} shows to be a good candidate for imposing outflow while ensuring well-posedness of the Navier--Stokes system. However, this condition generates complications for optimization problems, especially when an adjoint approach is considered.  In this paper we decided to consider a boundary condition, which we call a {\it convective boundary condition}. The said conditions will be discussed even further in the coming section.}
 
From here, when we talk about a domain $\Omega$, we always consider it having the boundary $\partial\Omega = \overline\Gamma_{\rm in}\cup\overline\Gamma_{\rm w}\cup\overline\Gamma_{\rm out}\cup\overline\Gamma_{\rm f}$, and that $\mathrm{dist}(\partial\Omega\backslash\Gamma_{\rm f},\Gamma_{\rm f}) > 0$.
\begin{figure}[h!]
 \centering
  \includegraphics[width=.5\textwidth]{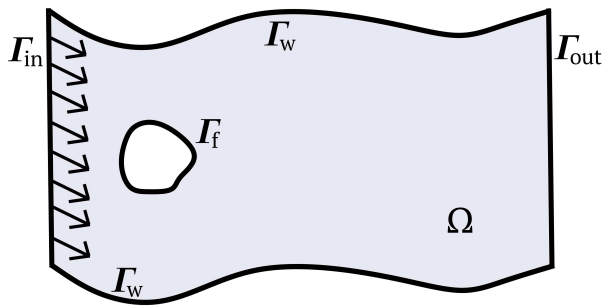}\vspace{-.1in}
 \caption{Set up of the fluid flow.}
\label{fig1}
 \end{figure}

%\begin{figure}[h]
%\centering
%\beginpgfgraphicnamed{flow}
%\begin{tikzpicture}
%	  \coordinate (c) at (1.75,2);
%      \draw[black, fill = blue, fill opacity = 0.5, semithick, even odd rule]
%            (0,0) rectangle (8,4) (c) \irregularcircle{.65cm}{.4mm};
%       \draw (-0.25,2) node{$\Gamma_{\rm{in}}$}  (2.3,2.65) node {$\Gamma_{\rm{f}}$} (7.5,2) node{$\Gamma_{\rm{out}}$} (3,4.2) node {$\Gamma_{\rm{w}}$} (3,-.25) node{$\Gamma_{\rm{w}}$} (6,3.5) node{${\Omega}$};
%       \foreach \y in {.25cm, .5cm, .75, 1cm, 1.25cm, 1.5cm, 1.75 ,2cm, 2.25cm, 2.5cm, 2.75 , 3cm, 3.25cm, 3.5cm, 3.75 }
%		\draw[->] (0pt,\y) -- (20pt,\y);
%\end{tikzpicture}
%\endpgfgraphicnamed
%\vspace{-.1in}
%\caption{Set up of the fluid flow.}
%\label{fig1}
%\end{figure}

{Due to the non-convexity of the vorticity functional $J$, possible non-existence of an optimal solution is a dilemma. Nevertheless, we shall show that $J$ is sequentially continuous with respect to the state variables. Combine this with the continuity of the map $\Omega\mapsto{\bu}$ with respect to some topology in the set of admissible domains, we circumvent the issue of existence of optimal shapes. We also propose a Tikhonov regularization in the form of the perimeter functional and a volume constraint to handle possible topological changes\footnote{Since our numerical treatment will be based on finite triangular elements and the shape deformations are realized using node/mesh movement, topological changes may cause some boundary elements to intersect, which is a problem that our current method cannot handle.}.}

To numerically solve the optimization problem, we shall utilize a gradient descent method based on the first order Eulerian derivative of the shape functional $\mathcal{G}$. Although there are several methods to determine the derivative with respect to shape variations - see for example \cite{gao2008} where the authors utilized minimax formulation and chain rule based on the Piola transform (see also \cite{schmidt2010} for computation of shape derivatives to general functionals) - the formulation of the shape derivative will be aided by the so-called rearrangement method which was formalized by Ito, K., Kunisch, K. and Peichl, G. \cite{ito2008}.  However, this computation will not take into account the volume constraint. This compels us to utilize two methods: first is the augmented Lagrangian method based on \cite[Framework17.3]{nocedal2006} which was applied to shape optimization problem in \cite{dapogny2018}; and the other one is by using solenoidal deformation fields.  Comparing the solutions of the two methods, we recognize a common profile - although the shapes are virtually different - which is a bean-shaped obstacle.

This paper is organized as follows: in the next section, we introduce the governing state equations, and its corresponding variational formulation.  In Section \ref{sect3} we establish the existence of optimal shapes, where we shall use the $L^\infty$-topology on the set of characteristic functions. The sensitivity of the objective functional will be analyzed in Section \ref{sect4}. Section \ref{sect5} is dedicated to the discussion of the numerical treatment by which we shall make use of the derivative formulated in Section \ref{sect4}. In this section, we shall also show the convergence of numerical solutions to a manufactured {\it exact} solution with respect to the Hausdorff measure, and the $H^1$ and $L^2$ norms.  Concluding remarks and possible future works will be discussed in Section \ref{sect6}.

%%%%%%%%%%%%%%%%%%%%%%%%
%                       Preliminaries							 %
%%%%%%%%%%%%%%%%%%%%%%%%
\section{Preliminaries}\label{sect2}
\subsection{Governing equations and necessary function spaces}
Let $\Omega\subset D$ be an open and bounded domain, the motion of the fluid is described by the velocity $\bu$ and the pressure $p$ which satisfy the stationary Navier--Stokes equations given by:
\begin{align}
	\left\{
	\begin{aligned}
	\, - \nu\Delta\bu +(\bu\cdot\nabla) \bu +  \nabla p
	&= \blf && \text{ in } \Omega,\\	
	\, \Div\bu &= 0 && \text{ in } \Omega,\\ 		
	\, \bu & = \bg && \text{ on }\Gamma_{\rm in}\\
	\, \bu &= 0 && \text{ on } \Gamma_{\rm{f}}\cup\Gamma_{\rm w}.\\
	\end{aligned} \right.
	\label{state}\tag{SE}
\end{align}
Here, $\bg$ is a divergence-free input function acting on the boundary $\Gamma_{\rm in}$.
The condition $\bu= 0$ on $\Gamma_{\rm{f}}\cup \Gamma_{\rm w}$ corresponds to no-slip boundary condition. For boundary $\Gamma_{\rm out}$, we choose an appropriate condition that will correspond to an outflow condition and that will give us a good energy estimate that is crucial in showing the existence of the velocity $\bu$ and for showing continuity with respect to domain variations.

Usually, the condition imposed to capture fluid outflow is what we call the {\it do-nothing} boundary condition. Mathematically, this condition is written by letting the product of the stress tensor and the vector normal to the boundary $\Gamma_{\rm out}$ equal to zero, i.e.,
\begin{equation} 
-p\bn + \nu \partial_{\bn\!}\bu = 0,
\label{donothing}
\tag{BC0}
\end{equation}
where $\partial_{\bn\!}\bu := (\nabla\bu)\bn$, and $\bn$ is the outward unit vector normal to $\Gamma_{\rm out}$. This condition has been considered as the outflow character in a lot of inquiries including that of Gresho, P.M. in \cite{gresho1991}, among others.

This profile, however, does not ensure the possibility of obtaining a good energy estimate, let alone the existence of a weak solution. As such, several alternative weak formulations has been formulated to address this issue (e.g. \cite{bruneau1996,heywood1992,zhou2016}). To illustrate the severity of the issue, we note that to solve the existence of a weak solution for the Navier--Stokes equations coupled with the do-nothing boundary condition the computations will give us the following expression
\begin{align*}
	\int_{\Gamma_{\rm out}} (\bu\cdot\bn)|\bv|^2\du s.
\end{align*}
Now, if $\bu$ and $\bv$ are identical, which will be a case when trying to establish the existence, specifically coercivity of the trilinear form with respect to the weak formulation, it is imperative to have a knowledge on the sign of the above quantity. This will be challenging due to its cubic form. This problem is addressed and circumvented by Bruneau, C.-H., and Fabrie, P. in \cite{bruneau1996} by introducing several versions of artificial boundary conditions for the Neumann condition. In this paper, we shall utilize one of the conditions introduced in the said literature, as well as establish its well-posedness.

We consider the Sobolev spaces $W^{k,p}(D)$ for any domain $D\subset\mathbb{R}^2$, $k\ge 0$ and $p\ge 1$. Note that if $p=2$, then $H^k(D)=W^{k,2}(D)$. 
For any domain $\Omega\subset\mathbb{R}^2$, letting $\Gamma_0 := \partial \Omega\backslash \Gamma_{\rm out} $we define 
\[\mathcal{W}(\Omega) = \{\boldsymbol{\varphi} \in C^\infty(\Omega)^2 : \boldsymbol\varphi=0\text{ on a neighborhood of }\Gamma_0 \}.\]
We denote by ${\bH}_{\Gamma_{0}}^r(\Omega)$ the closure of $\mathcal{W}(\Omega)$ with respect to the norm in $H^r(\Omega)^2$. By utilizing Meyers-Serrin type arguments it can be easily shown that 
\[ {\bH}_{\Gamma_{0}}^r(\Omega) = \{{\bphi}\in H^r(\Omega)^2: {\bphi}=0 \text{ on }\Gamma_0 \}.\]
To deal with the incompressibility condition, we shall take into account the following solenoidal spaces
\begin{align*}
	\bW(\Omega) &= \{ {\bphi}\in \mathcal{W}(\Omega): \Div{\bphi}=0\text{ in }\Omega \},\\
	\bV(\Omega) & = \{{\bphi}\in{\bH}_{\Gamma_{0}}^1(\Omega) : \Div{\bphi}=0 \text{ in }\Omega  \},\\
	\bH(\Omega) & = \{{\bphi}\in L^2(\Omega)^2 : \Div{\bphi}=0 \text{ in }L^2(\Omega), {\bphi}\cdot{\bn}=0\text{ on }\Gamma_0 \}.
\end{align*}
The space $\bW(\Omega)$ is dense in $\bV(\Omega)$, and the spaces $\bV(\Omega)$ and $\bH(\Omega)$ satisfy the Galfand triple property, i.e., $\bV(\Omega)\hookrightarrow \bH(\Omega) \hookrightarrow \bV(\Omega)^*$, where the first imbedding is dense, continuous and compact. Lastly, for simplicity of notations, we shall denote by $(\cdot,\cdot)_{\mathcal{D}}$ the $L^2(\mathcal{D})$, $L^2(\mathcal{D})^2$, or $L^2(\mathcal{D})^{2\times2}$ inner products, where $\mathcal{D}$ is any measurable set.

%%%%%%%%%%%%%%%%%%%%%%%%
%        Existence of State Solutions				 %
%%%%%%%%%%%%%%%%%%%%%%%%
\subsection{Weak formulation and existence of solutions}
Before we begin, let us first settle the issue on the regularity of the domain.  Although the assumption that $\Omega$ is of class $\mathcal{C}^{0,1}$ is sufficient to ensure the existence of weak solution to the Navier--Stokes equations, which we shall show later, this resulting solution will be established to be at most first differentiable. With this reason,  if we want a higher regularity for the solution then additional regularity on the domain is needed to be imposed. So from here on out, we shall say that the domain $\Omega$ satisfies (H$_{\Omega}$) if either of the following assumptions is satisfied:
\begin{align}
	\left\{
	\begin{aligned}
	& \bullet \text{ $\Omega$ is of class $\mathcal{C}^{1,1}$; or	}\\
	& \bullet \text{ the outer surface is a boundary of a convex polygon and the inner surface is of class $\mathcal{C}^{1,1}$.}
	\end{aligned}
	\right.
	\label{domainassumption}
	\tag{H$_{\Omega}$}
\end{align}

The reason for mentioning the above regularity assumptions on $\Omega$ is twofold: one -- as referred to before -- is for the regularity of the weak solution; and the other reason is to ensure that the extension of the input function ${\bg}$ that satisfies the properties below exists
\begin{align}
	\left\{
	\begin{aligned}
	&\Div \bg = 0\text{ in }\Omega;\quad \bg = 0\text{ on }\Gamma_{\rm f}\cup\Gamma_{\rm wall};\text{ and}\\
	&\int_{\Gamma_{\rm out}}\bg\cdot\bn \du s = -\int_{\Gamma_{\rm in}}\bg\cdot\bn \du s \ge 0.
	\end{aligned}
		\right.\label{gprop}
\end{align}
Of course, if we assume the minimum regularity on the domain, i.e.,  $\Omega$ is of class $\mathcal{C}^{0,1}$, then we can easily show, by the virtue of  \cite[Lemma 2.2]{girault1986}, that the extension of the input function ${\bg}\in H^{1/2}(\Gamma_{\rm in})^2$ that satisfies \eqref{gprop} exists and is in $H^1(\Omega)^2$.  However, this is not enough if we would wish to have higher regularity on the extension. To be precise, if we wish to extend ${\bg}\in H^{3/2}(\Gamma_{\rm in})^2$ in $\Omega$ that satisfies \eqref{gprop} and that the extension is in $H^2(\Omega)^2$, then the minimum regularity is for the domain to satisfy \eqref{domainassumption}. To simplify things later, we shall use the same notation for the input function and its extension. To be precise, when we refer to the assumption that ${\bg}\in H^m(\Omega)^2$ and satisfies \eqref{gprop}, we are implicitly saying that the input function is in $H^{m - 1/2}(\Gamma_{\rm in})^2$.

By letting $\tilde\bu = \bu - \bg$, we consider the following variational problem: For a given domain $\Omega$, find $\tilde\bu\in \bV(\Omega)$ such that
\begin{align}
	\mathbb{B}(\tilde{\bu},{\bphi}) := \mathbb{A}(\tilde{\bu};\tilde{\bu},{\bphi}) - \frac{1}{2}\mathbb{C}(\tilde{\bu};\tilde{\bu},{\bphi}) = \langle \Phi, {\bphi}\rangle_{\bV(\Omega)^*\times\bV(\Omega)},\quad \forall{\bphi}\in{\bV}(\Omega),
	\label{weak}
\end{align}
where $\mathbb{A}:{\bV}(\Omega)\times{\bV}(\Omega)\times{\bV}(\Omega)\to\mathbb{R}$ and $\mathbb{C}:{\bV}(\Omega)\times{\bV}(\Omega)\times{\bV}(\Omega)\to\mathbb{R}$ are trilinear forms respectively defined as follows:
\begin{align*}
\mathbb{A}({\bw};{\bu},{\bv}) &= \nu a({\bu},{\bv})_{\Omega} + b({\bw};{\bu},{\bv})_\Omega + b({\bg};{\bu},{\bv})_\Omega + b({\bw};{\bg},{\bv})_\Omega,\\ 
\mathbb{C}({\bw};{\bu},{\bv}) &=  ({\bw}\cdot{\bn},{\bu}\cdot{\bv})_{\Gamma_{\rm out}} +  ({\bg}\cdot{\bn},{\bu}\cdot{\bv})_{\Gamma_{\rm out}} +  ({\bw}\cdot{\bn},{\bg}\cdot{\bv})_{\Gamma_{\rm out}},
\end{align*} 
where the members of $\mathbb{A}$ are defined below:
\begin{enumerate}
\item[(i)] $a(\cdot,\cdot)_{\Omega}:\bV(\Omega)\times\bV(\Omega)\to\mathbb{R}$ is a bilinear form defined by
\[a(\bu,\bv)_{\Omega}=\int_{\Omega}\nabla\bu:\nabla\bv\du x,\] 
\item[(ii)] $b(\cdot;\cdot,\cdot)_\Omega: \bV(\Omega)\times\bV(\Omega)\times\bV(\Omega)\to\mathbb{R}$ is a trilinear form given as
\[b(\bw;\bu,\bv)_\Omega = \int_\Omega [(\bw\cdot\nabla)\bu]\cdot\bv \du x,\]
\end{enumerate}
and the action of $\Phi\in \bV(\Omega)^*$ is defined as
\begin{align*}
	\langle \Phi, {\bphi}\rangle_{\bV(\Omega)^*\times\bV(\Omega)} = ({\blf}-({\bg}\cdot\nabla){\bg},{\bphi})_{\Omega} - \nu(\nabla{\bg},\nabla{\bphi})_{\Omega} + \frac{1}{2}({\bg}\cdot{\bn},{\bg}\cdot{\bphi})_{\Gamma_{\rm out}},\quad \forall{\bphi}\in{\bV}(\Omega).
\end{align*}
Note that the variational equation \eqref{weak} is achieved from assuming that ${\bu}$ satisfies the boundary condition
\begin{align}
-p{\bn} + \nu\partial_{\bn}{\bu} - \frac{1}{2}({\bu}\cdot{\bn}){\bu} = 0\quad\text{on }\Gamma_{\rm out}. 
\label{conboundcond}
\tag{BC1}
\end{align}
We further mention that this boundary equation is among the artificial conditions proposed in \cite{bruneau1996} for $\Theta(a) = a$. As far as we are aware, this condition has never been analyzed, let alone be considered in an optimization problem.  The problem arises in -- as we have mentioned before -- showing the coercivity of the operators of the left hand side of the variational equation and due to the non-homogeneous Dirichlet data on $\Gamma_{\rm in}$. More obvious formulations have been considered by several authors (see \cite{braack2014}, \cite{zhou2016}) by considering the nonlinear part to be $({\bu}\cdot{\bn})_{-}{\bu}$, i.e.,
\begin{align}
-p{\bn} + \nu\partial_{\bn}{\bu} - \frac{1}{2}({\bu}\cdot{\bn})_{-}{\bu} = 0\quad\text{on }\Gamma_{\rm out},
\label{dirdonothing}
\tag{BC2}
\end{align}
where the notation $(\cdot)_{-}$ is defined as 
\begin{align*}
	(h)_{-} = \left\{\begin{aligned} 0 &&\text{if }h\ge0,\\ h &&\text{if } h<0.\end{aligned}\right.
\end{align*} 
This will {\it automatically} give a coercive left hand side, and hence the well-posedness of the state equations. However, this formulation will not be easy to handle when a variational approach for the optimization problem is utilized, and in the numerical treatment of the adjoint problem, to be specific.

Fortunately, we shall prove an analogous  result to that of \cite[Lemma IV.2.3]{girault1986} but takes into account the boundary integral on $\Gamma_{\rm out}$ that will help to render the variational problem \eqref{weak} well-posed. The mentioned analogy is shown in Lemma \ref{midg} and is proven in the Appendix.

We start the analysis by showing - under appropriate conditions - that $\Phi\in{\bV}(\Omega)^*$.
\begin{prpstn}
Let $\Omega\subset\mathbb{R}^2$ be of class $\mathcal{C}^{0,1}$, ${\blf}\in L^2(\Omega)^2$ and ${\bg}\in H^1(\Omega)^2$ satisfy \eqref{gprop}. Then, $\Phi\in{\bV}(\Omega)^*$ and that 
\begin{align}
	\|\Phi\|_{\bV(\Omega)^*} \le c(\|{\blf}\|_{L^2(\Omega)^2} + (1+\|{\bg}\|_{H^1(\Omega)^2})\|{\bg}\|_{H^1(\Omega)^2}),
\end{align}
for some constant ${c}>0.$
\label{Phicont}
\end{prpstn}
\begin{proof}
	The proof will utilize the compact embeddings $H^1(\Omega)^2 \hookrightarrow L^q(\partial\Omega)^2$ and $H^1(\Omega)^2 \hookrightarrow L^q(\Omega)^2$ for $q\ge2$ (see for example \cite[Part I, Theorem 6.3]{adams2003}).  Indeed, if ${\bphi}\in{\bV}(\Omega)$ then
	\begin{align*}
		|\langle \Phi, {\bphi}\rangle_{\bV(\Omega)^*\times\bV(\Omega)}| =&\, |({\blf}-{\bg}\cdot\nabla{\bg},{\bphi})_{\Omega} + \nu(\nabla{\bg},\nabla{\bphi})_{\Omega} + \frac{1}{2}({\bg}\cdot{\bn},{\bg}\cdot{\bphi})_{\Gamma_{\rm out}}|\\
		\le &\, |({\blf},{\bphi})_{\Omega}| + |({\bg}\cdot\nabla{\bg},{\bphi})_{\Omega}| + \nu|(\nabla{\bg},\nabla{\bphi})_{\Omega}| + \frac{1}{2}|({\bg}\cdot{\bn},{\bg}\cdot{\bphi})_{\Gamma_{\rm out}}|\\
		\le &\, \|{\blf}\|_{L^2(\Omega)^2}\|{\bphi}\|_{L^2(\Omega)^2} + \|{\bg}\|_{L^4(\Omega)^2}\|{\bg}\|_{H^1(\Omega)^2}\|{\bphi}\|_{L^4(\Omega)^2}\\
		& + \nu\|{\bg}\|_{H^1(\Omega)^2}\|{\bphi}\|_{H^1(\Omega)^2} + \frac{\tilde{c}_1}{2}\|{\bg}\|_{L^2(\Gamma_{\rm out})^2}^2\|{\bphi}\|_{L^2(\Gamma_{\rm out})^2}\\
		\le &\, c(\|{\blf}\|_{L^2(\Omega)^2} + (1+\|{\bg}\|_{H^1(\Omega)^2})\|{\bg}\|_{H^1(\Omega)^2})\|{\bphi}\|_{\bV(\Omega)},
	\end{align*}
where ${c} = \max\{1,\nu, \tilde{c}_1/2\}$ and $\tilde{c}_1$ is dependent on the $L^\infty$-norm of the outward normal vector ${\bn}$ which is bounded due to the regularity assumptions on the domain. Lastly, the linearity follows from the linearity of the action itself.
\end{proof}

Any function $\tilde\bu\in\bV(\Omega)$ that solves the variational equation \eqref{weak} is said to be a \textit{weak solution} to the Navier--Stokes equations \eqref{state} with the boundary condition \eqref{conboundcond}. The existence of the solution $\tilde\bu$ is summarized below.
\begin{thrm}
Let $\Omega$ be of class $C^{0,1},$ $\blf\in L^2(\Omega)^2$, and $\bg\in H^1(\Omega)^2$ satisfy \eqref{gprop}. The solution $\tilde\bu\in \bV(\Omega)$ to the variational problem \eqref{weak} exists such that 
\begin{equation}
\|\tilde{\bu}\|_{\bV} \le c(\|{\blf}\|_{L^2(\Omega)^2} + (1+\|{\bg}\|_{H^1(\Omega)^2})\|{\bg}\|_{H^1(\Omega)^2}),
\label{energy}
\end{equation}
for some constant \(c>0\). 
\label{th:wp}
\end{thrm}

The proof of the theorem will make use of the following lemmata whose proofs can be easily redone and can be based for example on \cite[Lemma II.1.8]{temam1977} and \cite[Lemma IV.2.3]{girault1986}.

\begin{lmm}
Let $\bu,\bv,\bw\in \bV(\Omega)$, the trilinear form $b$ satisfies the following properties.
\begin{itemize}
	\item[1.] $|b(\bu;\bv,\bw)_{\Omega}| \le c\|\bu\|_{\bH(\Omega)}^{1/2}\|\bu\|_{\bV(\Omega)}^{1/2}\|\bv\|_{\bH(\Omega)}^{1/2}\|\bv\|_{\bV(\Omega)}^{1/2}\|\bw\|_{\bV(\Omega)}$;
	\item[2.] $b(\bu;\bv,\bw)_{\Omega} + b(\bu;\bw,\bv)_{\Omega} = \displaystyle\int_{\Gamma_{\rm out}} (\bu\cdot\bn)(\bv\cdot\bw)\du s$;
	\item[3.] $b(\bu;\bv,\bv)_{\Omega} = \displaystyle\frac{1}{2}\int_{\Gamma_{\rm out}} (\bu\cdot\bn)|\bv|^2\du s$.
\end{itemize}
\label{propb}
\end{lmm}

\begin{lmm}
For any $\gamma>0$ there exists ${\bw}_0 = {\bw}_0(\gamma)\in H^1(\Omega)^2$ such that 
\begin{align}
	\Div {\bw}_0 = 0,\quad &{\bw}_0{\large|}_{\Gamma_0} = {\bg},\\
	|b({\bv};{\bw}_0,{\bv})_{\Omega} - \frac{1}{2}({\bv}\cdot{\bn}, {\bv}\cdot{\bw}_0)_{\Gamma_{\rm out}}&|\le \gamma\|{\bv}\|^2_{\bV(\Omega)}\quad \forall\bv\in\bV(\Omega).
\end{align} 
\label{midg}
\end{lmm}

\begin{rmrk}
We note that the assumption that the input function ${\bg}$ is in $H^1(\Omega)^2$ is valid also due to Lemma \ref{midg}. Meaning to say, ${\bw}_0 = {\bg}$ not only on $\Gamma_0$ but also inside $\Omega$. In particular, we also get the estimate
\begin{align*}
	|b({\bv};{\bg},{\bv})_{\Omega} - \frac{1}{2}({\bv}\cdot{\bn}, {\bv}\cdot{\bg})_{\Gamma_{\rm out}}&|\le \gamma\|{\bv}\|^2_{\bV(\Omega)}\quad \forall\bv\in\bV(\Omega).
\end{align*} 
\end{rmrk}
Furthermore, let us also point out that even though the input function is a fixed Dirichlet profile on the boundary $\Gamma_{\rm in}$, its extension is dependent on the domain $\Omega$. This implies that the extension -- which we chose to denote similarly as ${\bg}$ -- in $\Omega$ will also be sensitive with domain deformations.

Aside from the lemmata, it is also noteworthy to point out the well-definedness of the trilinear form $b$, i.e., $[(\bu\cdot\nabla)\bv]\cdot\bw\in L^1(\Omega)$ whenever $\bu,\bv,\bw\in \bV(\Omega)$. \\
\\
\noindent{\it Proof of Theorem \ref{th:wp}}. The proof utilizes Galerkin method by considering the finite dimensional subspaces $\bV_n\subset\bV(\Omega)$. We consider the projected problem on $\bV_n$, i.e., we find solutions $\tilde\bu_n\in\bV_n$ that satisfies for all $\bphi\in \bV_n$ the equation
\begin{equation}
 \mathbb{A}(\tilde\bu_n;\tilde\bu_n,\bphi)_\Omega - \frac{1}{2}\mathbb{C}(\tilde\bu_n;\tilde\bu_n,\bphi) = \langle\Phi,{\bphi} \rangle_{{\bV}(\Omega)^*\times{\bV}(\Omega)}.
\label{weakdiscrete}
\end{equation}
Note that the left hand side is coercive, i.e., for any $\bu\in\bV(\Omega)$ there exists $c>0$ such that 
\[ \mathbb{B}({\bu},{\bu}) =  \mathbb{A}(\bu;\bu,\bu)_\Omega -\frac{1}{2}\mathbb{C}(\bu;\bu,\bu) \ge c\|\bu\|^2_{\bV(\Omega)}. \]
Indeed, by choosing $\gamma=\frac{\nu}{2}$ in Lemma \ref{midg}, 
\[ |b(\bu;\bg,\bu)_{\Omega} -\frac{1}{2}({\bu}\cdot{\bn}, {\bu}\cdot{\bg})_{\Gamma_{\rm out}}| \le \frac{\nu}{2}\|\bu\|^2_{\bV(\Omega)}. \]
Furthermore, 
\begin{align*} 
b(\bu;\bu,\bu)_{\Omega} + b(\bg;\bu,\bu)_{\Omega} -\frac{1}{2}({\bu}\cdot{\bn}, {\bu}\cdot{\bu})_{\Gamma_{\rm out}}-\frac{1}{2}({\bg}\cdot{\bn}, {\bu}\cdot{\bu})_{\Gamma_{\rm out}}& = 0.
\end{align*}
These identities give us the coercivity
\begin{align*}
	\frac{\nu}{2}\|\bu\|_{\bV(\Omega)}^2 \le \mathbb{B}({\bu},{\bu}).
\end{align*}

Using the coercivity above, and Lemma \ref{Phicont}, for any $n\ge1$ we have
\begin{align*}
	\frac{\nu}{2}\|\tilde\bu_n\|_{\bV(\Omega)}^2 \le c(\|{\blf}\|_{L^2(\Omega)^2} + (1+\|{\bg}\|_{H^1(\Omega)^2})\|{\bg}\|_{H^1(\Omega)^2})\|\tilde\bu_n\|_{\bV(\Omega)}.
\end{align*}
This gives us the uniform estimate, 
\begin{align*}
	\|\tilde\bu_n\|_{\bV(\Omega)} \le c(\|{\blf}\|_{L^2(\Omega)^2} + (1+\|{\bg}\|_{H^1(\Omega)^2})\|{\bg}\|_{H^1(\Omega)^2}),
\end{align*}
for some $c:=c(\nu,\Omega)>0$. Hence, there exists $\tilde\bu\in\bV(\Omega)$ such that 
\begin{align*}
\left\{
	\begin{aligned}
		\tilde\bu_n\rightharpoonup\tilde\bu &&\text{in }\bV(\Omega),\\
		\tilde\bu_n\to\tilde\bu &&\text{in }\bH(\Omega).\\
	\end{aligned}
\right.
\end{align*}
Using usual arguments (see for example \cite{gunzburger1998,kasumba2012,temam1977,girault1986}), we can easily show that for any $\bphi\in \bV(\Omega)$, $\mathcal{A}(\tilde\bu_n;\tilde\bu_n,\bphi)_{\Omega}-\frac{1}{2}\mathbb{C}(\tilde\bu_n;\tilde\bu_n,\bphi) \to \mathbb{A}(\tilde\bu;\tilde\bu,\bphi)_{\Omega}-\frac{1}{2}\mathbb{C}(\tilde\bu;\tilde\bu,\bphi) $. 

Indeed,  by the virtue of Lemma \ref{propb}(3), \eqref{weakdiscrete} can be written as
\begin{align*}
	\nu a(\tilde{\bu}_n,{\bphi}) - b(\tilde{\bu}_n+{\bg};{\bphi},\tilde{\bu}_n) -b(\tilde{\bu}_n,{\bphi},{\bg}) &+\frac{1}{2} \mathbb{C}(\tilde{\bu}_n;\tilde{\bu}_n,\bphi)= \langle\Phi,\bphi\rangle_{\bV(\Omega)^*\times\bV(\Omega)}
\end{align*}
By following the proofs on the said references, it can be shown that the next convergences hold
\begin{align*}
	\left\{
		\begin{aligned}
			&a(\tilde{\bu}_n,\bv)_{\Omega} \to a(\tilde{\bu},\bv)_{\Omega},\\
			&b(\tilde{\bu}_n;\bphi,\tilde{\bu}_n)_{\Omega} \to b(\tilde{\bu};\bphi,\tilde{\bu})_{\Omega},\\
			&b(\tilde{\bu}_n;\bphi,\bg)_{\Omega} \to b(\tilde{\bu};\bphi,\bg)_{\Omega},\\
			&b(\bg;\bphi,\tilde{\bu}_n)_{\Omega} \to  b(\bg;\bphi,\tilde{\bu})_{\Omega}.
		\end{aligned}
	\right.
\end{align*}
So, what remains for us to show is that
\begin{align*}
	\left\{
		\begin{aligned}
			&(\tilde\bu_n\cdot\bn,\tilde\bu_n\cdot\bphi)_{\Gamma_{\rm out}}\to (\tilde\bu\cdot\bn,\tilde\bu\cdot\bphi)_{\Gamma_{\rm out}},\\
			&(\tilde\bu_n\cdot\bn,\bg\cdot\bphi)_{\Gamma_{\rm out}}\to(\tilde\bu\cdot\bn,\bg\cdot\bphi)_{\Gamma_{\rm out}},\\
			&(\bg\cdot\bn,\tilde\bu_n\cdot\bphi)_{\Gamma_{\rm out}}\to(\bg\cdot\bn,\tilde\bu\cdot\bphi)_{\Gamma_{\rm out}} .
		\end{aligned}
	\right.
\end{align*}
We focus on the first convergence, since the others can be done similarly. Before we begin, we note that $H^1(\Omega)$ is compactly embedded to $L^q(\partial\Omega)$, for $q\ge2$ (see for example \cite[Part I, Theorem 6.3]{adams2003}). This implies that $\tilde\bu_n\to\tilde\bu$ and $\tilde\bu_n\cdot\bn\to\tilde\bu\cdot\bn$ in $L^q(\Gamma_{\rm out})^2$ and $L^q(\Gamma_{\rm out})$, respectively. Hence, we have the following computation:
\begin{align*}
	\big|(\tilde\bu_n\cdot\bn,\tilde\bu_n\cdot\bphi)_{\Gamma_{\rm out}}-  (\tilde\bu\cdot\bn,&\tilde\bu\cdot\bphi)_{\Gamma_{\rm out}}\big|\\
	 \le& \left| ([\tilde\bu_n-\tilde\bu]\cdot\bn,\tilde\bu_n\cdot\bphi)_{\Gamma_{\rm out}} \right| + \left| (\bu\cdot\bn,[\tilde\bu_n-\tilde\bu]\cdot\bphi)_{\Gamma_{\rm out}} \right|\\
	\le&\ c\|(\tilde\bu_n-\tilde\bu)\cdot\bn \|_{L^2(\Gamma_{\rm out})}\|\tilde\bu_n\|_{L^4(\Gamma_{\rm out})^2}\|\bphi\|_{L^4(\Gamma_{\rm out})^2}\\
	& + c\|\tilde\bu\cdot\bn \|_{L^4(\Gamma_{\rm out})}\|\tilde\bu_n-\tilde\bu\|_{L^2(\Gamma_{\rm out})^2}\|\bphi\|_{L^4(\Gamma_{\rm out})^2}.
\end{align*}
The right side approaches zero from the mentioned convergences in $L^q(\Gamma_{\rm out})^2$ and $L^q(\Gamma_{\rm out})$, which proves the claim. 
Therefore, $\tilde\bu$ solves \eqref{weak} with the energy estimate
\begin{align*}
	\|\tilde\bu\|_{\bV(\Omega)} \le c(\|{\blf}\|_{L^2(\Omega)^2} + (1+\|{\bg}\|_{H^1(\Omega)^2})\|{\bg}\|_{H^1(\Omega)^2}).
\end{align*}\qed

\begin{rmrk}

(i) Even though we just assumed ${\blf}\in L^2(\Omega)^2$ and ${\bg}\in H^1(\Omega)^2$, as we shall see later, we can have these functions extended to the hold all domain. Furthermore, since $\Omega\subset D$, the energy estimate can be extended to the hold-all domain $D$, i.e.,
\[\|\tilde\bu\|_{\bV(\Omega)} \le c(\|{\blf}\|_{L^2(D)^2} + (1+\|{\bg}\|_{H^1(D)^2})\|{\bg}\|_{H^1(D)^2}), \]
and $c>0$ is dependent on $D$, thanks to Faber--Krahn inequality.
\noindent(ii) As previously mentioned, it can be shown that a more regular domain yields a more regular solution. In particular, if $\Omega$ satisfies \eqref{domainassumption}, then the weak solution to \eqref{weak} satisfies $\tilde{\bu}\in \bV(\Omega)\cap\,H^2(\Omega)^2$, and as a consequence of the Rellich-Kondrachov embedding theorem \cite[Chapter~5, Theorem~6]{evans1998} the solution is in $C(\overline{\Omega})^2$.
\end{rmrk}

The existence of the distribution $p$ is a direct consequence of De Rham's theorem \cite[Proposition~I.1.1]{temam1977} in such a way that $p\in L^2(\Omega)$. Furthermore, if we assume that 
\begin{align}
4\nu^2 > \|\Phi\|_{\bV'}\sup_{\bu,\bv,\bw}\frac{|b(\bu;\bv,\bw)_{\Omega}-\frac{1}{2}({\bu}\cdot{\bn},{\bv}\cdot{\bw})_{\Gamma_{\rm out}} |}{\|\bu\|_{\bV(\Omega)}\|\bv\|_{\bV(\Omega)}\|\bw\|_{\bV(\Omega)}} ,
\label{est:uniqueness}
\end{align}
then one can show that the solution is unique.

In strong form, if we assume appropriate regularity on $\Omega$,  and if $\tilde\bu\in\bV(\Omega)\cap H^2(\Omega)^2$ solves the variational equation \eqref{weak} then $\bu = \tilde{\bu}+\bg \in H^2(\Omega)^2$ can be regarded as the solution to the equations,
\begin{align}
	\left\{
	\begin{aligned}
	\, - \nu\Delta{\bu} +({\bu}\cdot\nabla){\bu} +  \nabla p
	&= {\blf} && \text{ in } \Omega,\\	
	\, \Div{\bu} &= 0 && \text{ in } \Omega,\\ 		
	\, {\bu} & = {\bg} && \text{ on }\Gamma_{\rm in}\\
	\, {\bu} &= 0 && \text{ on } \Gamma_{\rm{f}}\cup\Gamma_{\rm w},\\
	\, -p\bn + \nu\partial_{\bn}{\bu} & = \frac{1}{2}({\bu}\cdot{\bn}){\bu } &&\text{ on } \Gamma_{\rm out}.
	\end{aligned} \right.
	\label{strong_nontranslated}
\end{align}

Before closing this section we look at some operators which will be handy when we go to the investigation of sensitivities. First, let us introduce the Stokes operator $A_\Omega:D(A_\Omega)\subset\bH(\Omega)\to\bH(\Omega) $ defined by $A_\Omega\bu = -P_\Omega\Delta\bu$ for $\bu\in D(A_\Omega)$. Here $P_\Omega:L^2(\Omega)^2\to \bH(\Omega)$ is the Leray projection that is associated with the decomposition $L^2(\Omega)^2 = \bH(\Omega)\oplus\nabla L^2(\Omega)$.  As a direct consequence, if the domain $\Omega$ satisfies \eqref{domainassumption} then $D(A_\Omega) = \bV(\Omega)\cap H^2(\Omega)^2 $, hence item (ii) of the remark above.  We mention, furthermore, that the Stokes operator is a self-adjoint positive operator with dense domain and compact inverse, this in turn, gives us the orthonormal basis of $\bH(\Omega)$. This orthonormal basis makes the Galerkin method in the proof of Theorem \ref{th:wp} possible to utilize.  With this operator, we can write the first member of $\mathbb{A}$ as follows,
\[a(\bu,\bv)_\Omega = \langle A_\Omega\bu, \bv \rangle_{\bV(\Omega)^*\times \bV(\Omega)}. \]
We also introduce the operator $B_\Omega:\bV(\Omega)\times\bV(\Omega)\to \bV(\Omega)^*$ defined by 
\[\langle B_\Omega(\bu,\bv),\bw\rangle_{\bV(\Omega)^* \times \bV(\Omega)} = b(\bu;\bv,\bw)_\Omega \quad\forall \bu,\bv,\bw\in\bV(\Omega).\]
For any $\bu\in\bV(\Omega)$, we shall also use the notation $B_\Omega{\bu} := B_\Omega({\bu},{\bu})$.  Lastly, we consider the boundary operator $C:\bV(\Omega)\times\bV(\Omega)\to H^{-\frac{1}{2}}(\Gamma_{\rm out})^2$ defined by
 \[ \langle C(\bu,\bv),\bw\rangle_{H^{-\frac{1}{2}}(\Gamma_{\rm out})^2\times H^{\frac{1}{2}}(\Gamma_{\rm out})^2} = \int_{\Gamma_{\rm out}} (\bu\cdot\bn)(\bv\cdot\bw)\du s. \]
The well-definedness, and continuity of the operator $C$ can be established by utilizing the Rellich--Kondrachov embedding  \cite[Part I, Theorem 6.3]{adams2003}.

%%%%%%%%%%%%%%%%%%%%%%%%cr
%             		Existence of Shapes					 %
%%%%%%%%%%%%%%%%%%%%%%%%
\section{Existence of Optimal Shapes}\label{sect3}
The purpose of this section is to establish the existence of a solution to the shape optimization problem. To be able to do this, a set of admissible domains $\Oad$ will be considered to ensure the existence of state solutions as well as ensure that -- assuming that it exists --  the solution is embedded to the hold-all domain $D$. In this set of admissible domains, a topology will be endowed upon which $\Oad$ itself is compact. Utilizing this compactness property, standard sequential arguments will then be followed to establish the promised existence of the optimal shape.

\subsection{Cone property and the set of admissible domains}

Before we begin, we assume that $\Gamma_{\rm out}\subset \partial D$, which is not going to be an issue on the generality of the problem since $\Gamma_{\rm out}$ is not a part of the free-boundary. Now, to define the topology on the set of admissible domains, we shall resort to the collection of domains that satisfy the cone property. Note that, one way to define the topology on the set of admissible domains is by parametrizing the free-boundary. Examples of such parametrizations are found in \cite{rabago2019,rabago2020} and the references therein. 
However defining the domains in such way,  will be problematic for the current problem since the free-boundary parametrization might lead to generating domains with varying volumes.

To pass through the challenge of preserving the volume, a classical way to define the topology on the set of admissible domains is by considering the collection of characteristic functions, which is a fact highlighted by Henrot and Privat in \cite{privat2010} and is rigorously discussed in \cite{zolesio2011} and \cite{henrot2018}. 
But before we delve into this topology, let us first look at the manner by which the domains are defined. 
In particular, we shall consider domains which satisfy the \textit{cone property} as defined in \cite{chenais1975}. 

\begin{dfntn} Let $h>0$ and $\theta\in[0,2\pi]$, and \(\xi\in\mathbb{R}^2\) such that $\|\xi\| = 1$. 

(i.) The cone of angle $\theta$, height $h$ and axis $\xi$ is defined as 
\[ C(\xi,\theta,h) = \{x\in\mathbb{R}^2: (x,\xi)>\|x\|\cos\theta, \|x\|<h \},\]
where \((\cdot,\cdot)\) and $\|\cdot\|$ denote the inner product and Euclidean norm in $\mathbb{R}^2$, respectively.

(ii.) A set $\Omega\subset\mathbb{R}^2$ is said to satisfy the \text{cone property} if and only if for all $x\in\partial\Omega$, there exists $C_x=C(\xi_x,\theta,h)$, such that for all $y\in B(x,r)\cap\Omega$ we have $y+C_x \subset\Omega$.
\end{dfntn}

From this definition, we now define the set of admissible domains  as 
\[ \Oad := \{\Omega\subset D: \Omega\text{ satisfies the }cone\ property\text{ and }|\Omega|=m \}. \]
A sequence $\{\Omega_n\}\subset\Oad$ is said to converge to $\Omega\in\Oad$ if 
\[\chi_{\Omega_n}\ws\chi_{\Omega} \text{ in }L^\infty(D),\]
where the function $\chi_A$ for a set $A\subset\mathbb{R}^2$ refers to the characteristic function defined by
\[\chi_A(x) =\left\{\begin{matrix} 1&\text{if } x\in A,\\ 0&\text{if }x\not\in A. \end{matrix}\right.   \] 

\begin{rmrk}
(i.) This set of admissible domains has been established to be non-empty (see the proof of Proposition 4.1.1 in \cite{henrot2018}), which exempts us from the futility of the analyses we will be going through in the succeeding sections.\\
(ii.) Note the this convergence may seem too weak in the sense that the weak$^*$ convergence only assures us that the limit $\chi_\Omega$ only satisfies $\chi_\Omega(x) \in [0,1]$, however as pointed out by Henrot, et.al., \cite{henrot2018} in Proposition 2.2.1, this convergence holds in the space $L^p_{loc}(\mathbb{R}^2)$ and thus $\chi_\Omega$ is an almost everywhere characteristic function. We refer the reader to \cite[Chapter 5]{zolesio2011} and \cite[Chapter 2~Section 3]{henrot2018} for a more detailed discussion on the topology of characteristic functions of finite measurable domains.
\end{rmrk}

Before we mention the compactness of the set $\Oad$, let us first look at an important implication of the cone property, i.e., the existence of a uniform extension operator, and is given by the lemma below.
\begin{lmm}[cf \cite{chenais1975}]
There exists $K>0$ such that for all $\Omega\in\Oad$, there exists \[\mathcal{E}_{\Omega}^d:H^{m}(\Omega)^d\to H^{m}(D)^d\] which is linear and continuous such that  $\max\{\|\mathcal{E}^d_\Omega\|\} \le K$.
\label{lemmal:ubp}
\end{lmm}

This lemma will be utilized in several occasions, for example when we show that the \text{domain-to-state} map is continuous.
Meanwhile, the compactness of $\Oad$ follows from the fact that it is closed and relatively compact -- as defined by D. Chenais \cite{chenais1975} -- with respect to the $weak^*$-$L^\infty$ topology on $\mathcal{U}_{ad}:= \{\chi_\Omega:\Omega\in \Oad \}$. One can also read upon the proof in \cite[Proposition 2.4.10]{henrot2018}. We shall not discuss the proof of such properties, nevertheless they are summarized on the lemma below.
\begin{lmm}
The set $\Oad$ is compact with respect to the topology on $\mathcal{U}_{ad}$.
\label{le:comome}
\end{lmm}

Aside from these properties, it is noteworthy to mention the set of admissible domains can be identified as Lipschitzian domains as well. Since we only consider the hold-all domain $D$ to be possessing a bounded boundary, a proof of this property can be found in Henrot, A. and Pierre, M. \cite[Theorem 2.4.7]{henrot2018}.

\subsection{Well-posedness of the optimization problem}

Now that we have defined a good topology on the set of admissible domains, this subsection is dedicated to establishing the existence of the solution to the shape optimization problem. 

Recall that the functional $J$ is written as a function of the state solution $\bu$ and of the domain $\Omega$. Fortunately, we point out that for each $\Omega\in\Oad$ there exists a solution $\bu\in\bV(\Omega)$ to the weak formulation \eqref{weak}, hence the map $\Omega\mapsto\bu$.
This implies that the well-posedness of the optimization will depend on the continuity of the domain-to-state map, which we shall briefly establish shortly. For now, let us consider the velocity-pressure formulation of \eqref{weak} given by finding $(\tilde{\bu},p)\in H_{\Gamma_0}^1(\Omega)^2\times L^2(\Omega)$ such that
\begin{align}
	\left\{
		\begin{aligned}
			\mathbb{A}(\tilde{\bu};\tilde{\bu},{\bphi}) - \frac{1}{2}\mathbb{C}(\tilde{\bu};\tilde{\bu},{\bphi}) + d({\bphi},p)_\Omega =& \langle \Phi, {\bphi}\rangle_{[H^1_{\Gamma_0}(\Omega)^2]^*\times H^1_{\Gamma_0}(\Omega)^2} &&\forall{\bphi}\in H_{\Gamma_0}^1(\Omega)^2,\\
			d(\tilde{\bu},q)_\Omega =& 0&& \forall q\in L^2(\Omega),
		\end{aligned}
	\right.\label{weak:vp}
\end{align}
where $d(\cdot,\cdot)_\Omega:H_{\Gamma_0}^1(\Omega)^2\times L^2(\Omega)\to \mathbb{R}$ is defined as $d(\tilde{\bu},p)_\Omega = -(p,\dive\tilde{\bu})_{\Omega}$.
The existence of the pair $(\tilde{\bu},p)\in H_{\Gamma_0}^2(\Omega)^2\times L^2(\Omega)$ that satisfies \eqref{weak:vp} follows from the existence of solution to \eqref{weak} and since $d(\cdot,\cdot)_\Omega$ satisifies the inf-sup condition \cite{bertoluzza2017}. Furthermore, the following energy estimate holds
\begin{align}
	\|\tilde{\bu} \|_{H_{\Gamma_0}^1(\Omega)^2} + \|p\|_{L^2(\Omega)} \le c(\|{\blf}\|_{L^2(D)^2} + (1+\|{\bg}\|_{H^1(D)^2})\|{\bg}\|_{H^1(D)^2}).
	\label{energy:vp}
\end{align}

The impetus for introducing the variational equation \eqref{weak:vp} is to make sure that the divergence-free property of the states will be preserved on any domain in $\Oad$ and so that the said property will still be reflected when the uniform extension property in Lemma \ref{lemmal:ubp} is utilized.

\begin{prpstn}
Let $\{\Omega_n\}\subset\Oad$ be a sequence that converges to $\Omega\in\Oad$. Suppose that for each $\Omega_n$, $(\tilde{\bu}_n,p_n)\in H_{\Gamma_0}^1(\Omega_n)^2\times L^2(\Omega_n)$ is a solution of the variational equation \eqref{weak:vp} on the respective domain; then the extensions $(\overline{\bu}_n,\overline{p}_n):=(\mathcal{E}_{\Omega_n}^2\tilde\bu_n, \mathcal{E}_{\Omega_n}^1p_n )\in H^1(D)^2\times L^2(D)$ coverges to a state $(\overline{\bu},\overline{p} )\in H^1(D)^2\times L^2(D)$, such that $(\tilde\bu,p )=(\overline{\bu},\overline{p} )\big|_{\Omega}$ is a solution to \eqref{weak:vp} in $\Omega$.
\label{contu}
\end{prpstn}
\begin{proof}
From the uniform extension property, there exists $K>0$ such that 
\[\|\overline{\bu}_n\|_{H^1(D)^2} + \|\overline{p}_n \|_{L^2(D)} \le K(\|\tilde\bu_n\|_{H_{\Gamma_0}^1(\Omega_n)^2} + \|p_n \|_{L^2(\Omega_n)})\text{ for all }\Omega_n.\] Furthermore, from the energy estimate \eqref{energy:vp}
\[\|\tilde{\bu}_n \|_{H_{\Gamma_0}^1(\Omega_n)^2} + \|p_n\|_{L^2(\Omega_n)} \le c(\|{\blf}\|_{L^2(D)^2} + (1+\|{\bg}\|_{H^1(D)^2})\|{\bg}\|_{H^1(D)^2}).\] 
From the uniform boundedness of $\{(\overline\bu_n,\overline{p}_n )\}$ in $H^1(D)^2\times L^2(D)$ and by the virtue of the Rellich-Kondrachov and Banach-Alaoglu theorems, there exists a subsequence of $\{(\overline\bu_n,\overline{p}_n )\}$, which we denote in the same manner, and an element $(\overline{\bu},\overline{p} )\in H^1(D)^2\times L^2(D)$ such that 
\begin{align}
\left\{
\begin{aligned}
	&\overline{\bu}_n \rightharpoonup \overline{\bu} &&\text{in }H^1(D)^2,\\
	&\overline{\bu}_n \to \overline{\bu}	&&\text{in }L^2(D)^2,\\
	&\overline{p}_n \rightharpoonup \overline{p}			&&\text{in }L^2(D).
\end{aligned}\right.\label{solcon}
\end{align} 

\noindent\textit{Passing through the limit.}
The next step is to show that $(\tilde\bu,p )$ solves \eqref{weak:vp} in $\Omega$. 

By the assumed domain convergence, we have $\chi_n:=\chi_{\Omega_n}\ws \chi:=\chi_\Omega$ in $L^\infty(D)$. Furthermore, since $(\tilde{\bu}_n,p_n)$ solves \eqref{weak:vp} in $\Omega_n$, then for any $(\bpsi,\phi )\in H^1(D)^2\times L^2(D).$
\begin{equation}
	\left\{
	\begin{aligned}
	\mathcal{A}_{\chi_n}(\overline{\bu}_n;\overline{\bu}_n,\bpsi)_D - \frac{1}{2}\mathbb{C}(\overline{\bu}_n;\overline{\bu}_n,\bpsi) + d({\bpsi},\chi_n \overline{p}_n)_D  & = \langle \Phi_{\chi_n},{\bpsi}\rangle_{[H^1(D)^2]^*\times H^1(D)^2},\\
	d(\overline{\bu}_n,\chi_n \phi)_D & = 0.
	\end{aligned}
	\right.
\label{weakDn}
\end{equation}
Here, for any function $\vartheta: D\to \mathbb{R}$, the trilinear form $\mathcal{A}_{\vartheta}(\cdot,\cdot,\cdot)_D$ can be dissected into several components, namely 
\[ \mathcal{A}_{\vartheta}(\bw;\bu,\bv)_D = a_\vartheta(\bu,\bv)_D+ b(\vartheta\bw;\bu,\bv)_D+b(\vartheta\bg;\bu,\bv)_D+b(\vartheta\bw;\bg,\bv)_D, \]
where $a_\vartheta(\bu,\bv)_D = \int_{D}\vartheta\nabla\bu:\nabla\bv\du x $, and $\Phi_\vartheta \in [H^1(D)^2]^*$ is defined as 
\begin{align*}
	\langle \Phi_\vartheta, {\bpsi}\rangle_{[H^1(D)^2]^*\times H^1(D)^2} = (\vartheta[{\blf}-({\bg}\cdot\nabla{\bg})],{\bpsi})_{D} + \nu(\vartheta\nabla{\bg},\nabla{\bpsi})_{D} + \frac{1}{2}({\bg}\cdot{\bn},{\bg}\cdot{\bpsi})_{\Gamma_{\rm out}}.
\end{align*}

Our goal is to show that $(\tilde\bu,p )=(\overline{\bu},\overline{p} )\big|_{\Omega}$ is a solution to \eqref{weak:vp} by establishing that the following system holds
\begin{equation}
	\left\{
	\begin{aligned}
	\mathcal{A}_{\chi}(\overline{\bu};\overline{\bu},\bpsi)_D -  \frac{1}{2}\mathbb{C}(\overline{\bu};\overline{\bu},\bpsi)  + d({\bpsi},\chi \overline{p})_D & = \langle \Phi_{\chi},{\bpsi}\rangle_{[H^1(D)^2]^*\times H^1(D)^2},\\
	d(\overline{\bu},\chi \phi)_D & = 0.
	\end{aligned}
	\right.
\label{weakD}
\end{equation}
by utilizing \eqref{solcon} and the weak-$^*$ limit of the characteristic functions on \eqref{weakDn}. 

Using the same arguments as in Theorem \ref{th:wp} it can be easily shown that 
\[ 
\mathcal{A}_{\chi_n}(\overline{\bu}_n;\overline{\bu}_n,\bpsi)_D - \frac{1}{2}\mathbb{C}(\overline{\bu}_n;\overline{\bu}_n,\bpsi) \to \mathcal{A}_{\chi}(\overline{\bu};\overline{\bu},\bpsi)_D - \frac{1}{2}\mathbb{C}(\overline{\bu};\overline{\bu},\bpsi).
\]
Furthermore, from the assumed convergence of the characteristic functions we infer that 
\[\langle \Phi_{\chi_n},{\bpsi}\rangle_{[H^1(D)^2]^*\times H^1(D)^2} \to \langle \Phi_{\chi},{\bpsi}\rangle_{[H^1(D)^2]^*\times H^1(D)^2}.\]
Lastly, since $\overline{\bu}_n \rightharpoonup \overline{\bu}$ in $H^1(D)^2$ and  $\overline{p}_n \rightharpoonup \overline{p}$ in $L^2(D)$, then
\begin{align*}
	\left\{
		\begin{aligned}
			d({\bpsi},\chi_n \overline{p}_n)_D \to d({\bpsi},\chi \overline{p})_D,\\
			d(\overline{\bu}_n,\chi_n \phi)_D \to d(\overline{\bu},\chi \phi)_D.
		\end{aligned}
	\right.
\end{align*}

By letting $({\bphi},q)\in H_{\Gamma_0}^1(\Omega_n)^2\times L^2(\Omega_n)$, and defining $({\bpsi},\phi)\in H^1(D)^2\times L^2(D)$ by $({\bpsi},\phi) = ({\bphi},p)$ in $\Omega$ and $({\bpsi},\phi)=(0,0)$ in $D\backslash\overline{\Omega}$, we get the variational equation \eqref{weak:vp} in $\Omega$.
\end{proof}

Just to reiterate, Proposition \ref{contu} proves that the map $\Omega \mapsto (\tilde{\bu},p)$ is continuous. This property will be instrumental to prove that the optimal shape exists. In fact, we shall use the fact that we can write the objective functional as a function that depends solely on the elements of $\Oad$, and show that it is continuous with respect to this collection as well. To be precise, we prove the existence of a solution to the shape optimization problem on the theorem below. Before we begin, let us introduce the following notations which were made possible from the well-definedness of the map $\Omega\mapsto{\bu}$:
\begin{align*}
	\mathcal{G}(\Omega):= \mathcal{G}(\mathcal{\bu},\Omega),\quad J(\Omega):=J({\bu},\Omega).
\end{align*}

\begin{thrm}
Suppose that the assumptions for Theorem \ref{th:wp} together with the uniqueness assumption \eqref{est:uniqueness} hold; then, there exists $\Omega^*\in\Oad$ such that 
\[ \mathcal{G}(\Omega^*) = \min_{\Omega\in\Oad}\mathcal{G}(\Omega).\]
\end{thrm}
\begin{proof}
First, we show that the objective functional is lower semicontinuous. This is done by showing that the component $J$ is continuous (hence upper-semicontinuous) with respect to the state variable ${\bu} = \tilde{\bu} + {\bg}$, and by using the lower-semicontinuity of the perimeter functional \cite[Proposition 2.3.7]{henrot2018}.

We shall then show that $\mathcal{G}$ is bounded from below, which will imply the existence of a minimizing sequence of domains whose evaluations will converge to an infimum value of the objective functional. By using the compactness of $\Oad$, we shall then show that this sequence converges to a domain such that its evaluation coincides with the infimum.

\noindent {\it Step 1:} Lower Semicontinuity of $\mathcal{G}$.

Note that we can estimate $J$ by the $H^1$ norm of the state $\mathcal{\bu}$ by the virtue of the following computation, which implies the continuity of $J$ with respect to the state variable ${\bu}$:
\begin{align}
J(\Omega) & = \frac{\gamma}{2}\|\nabla\times\bu\|_{L^2(\Omega)}^2 = \frac{\gamma}{2}\int_{\Omega} |\nabla\times\bu|^2\du x \le  \frac{\gamma}{2}\int_{\Omega} |\nabla\bu|^2\du x \le  \frac{\gamma}{2}\|\bu\|^2_{H^1(\Omega)^2}.
\label{est:j}
\end{align}
Furthermore,  since Proposition \ref{contu} implies that the map $\Omega \mapsto \tilde{\bu}+\bg$ is continuous, the map $\Omega\mapsto J(\Omega)$ is also continuous, and hence upper-semicontinuous. Thus, the map $\Omega \mapsto \mathcal{G}(\Omega)$ is lower-semicontinuous, i.e., for any sequence $\{\Omega_n\}_n\subset\Oad$ that converges to an element $\Omega^*\in\Oad$, then
\begin{align*}
	\mathcal{G}(\Omega^*) \le \liminf_{n\to\infty}\mathcal{G}(\Omega_n).
\end{align*}

\noindent {\it Step 2:} Existence of a minimizing sequence.

Firstly, by using estimate \eqref{est:j}, we can show that $J$ is uniformly bounded. Indeed, from \eqref{est:j} and the energy estimate \eqref{energy}, we get 
\begin{align}
	\begin{aligned}
	J(\Omega) 	& \le \frac{\gamma}{2}\|\bu\|^2_{H^1(\Omega)^2} = \frac{\gamma}{2}\|\tilde{\bu} + {\bg}\|^2_{H^1(\Omega)^2}\\
						& \le \frac{\gamma}{2}(\|\tilde{\bu}\|^2_{H^1(\Omega)^2} +\|{\bg}\|^2_{H^1(\Omega)^2} )\\
						& \le c(\|{\blf}\|_{L^2(D)^2} + (2+\|{\bg}\|_{H^1(D)^2})\|{\bg}\|_{H^1(D)^2}).
	\end{aligned}
\end{align}
Furthermore, since any domain $\Omega\in\Oad$ is bounded, then the inner boundary $\Gamma_{\rm f}$ is also bounded and that its perimeter can be bounded below uniformly, say $\alpha P(\Omega) \ge P^*$ for any $\Omega\in\Oad$. Therefore, $\mathcal{G}$ is bounded from below, i.e., for any $\Omega\in\Oad$
\begin{align*}
	\mathcal{G}(\Omega) = \alpha P(\Omega) - J(\Omega) \ge P^* - c(\|{\blf}\|_{L^2(D)^2} + (2+\|{\bg}\|_{H^1(D)^2})\|{\bg}\|_{H^1(D)^2}).
\end{align*}
Hence, there exists a sequence $\{\Omega_n\}_n\subset\Oad$ such that 
\begin{align*}
	\mathcal{G}^* := \inf_{\Omega\in\Oad} \mathcal{G}(\Omega_n) = \lim_{n\to\infty}\mathcal{G}(\Omega_n).
\end{align*}

\noindent {\it Step 3:} Existence of a minimizer for $\mathcal{G}$.

Since $\Oad$ is compact, then the sequence $\{\Omega_n\}_n$ from {\it Step 2} has a subsequence -- which we shall denote similarly -- that converges to an element $\Omega^*\in\Oad$. Hence, from the lower-semicontinuity of $\mathcal{G}$, we establish that $\Omega^*$ is the minimizer of $\mathcal{G}$:
\begin{align*}
	\mathcal{G}(\Omega^*) \le \liminf_{n\to\infty}\mathcal{G}(\Omega_n) = \mathcal{G}^*.
\end{align*}

\end{proof}

%%%%%%%%%%%%%%%%%%%%%%%%
%                     Shape Sensitivity						 %
%%%%%%%%%%%%%%%%%%%%%%%%
\section{Shape sensitivity analysis}\label{sect4} 
This section is dedicated to investigate the sensitivity of the objective functional $\mathcal{G}$ with respect to domain variations. We start this section by introducing the identity perturbation method, where we consider domain variations generated by a given autonomous velocity.
\subsection{Identity perturbation}
We consider a family of autonomous velocity fields $\bth$ belonging to $\bTh := \{\bth\in C^{1,1}(\overline{D};\mathbb{R}^2): \bth = 0 \text{ on }\partial D\cup\Gamma_{\rm in}\cup\Gamma_{\rm w}\}$. From an element $\bth\in\bTh$ we can define an identity perturbation operator $T_t:\overline{D}\to\mathbb{R}^2$ defined by $T_t(x) = x + t\bth(x).$ We note that this operator can be shown as a direct consequence of the velocity method as discussed in \cite{sokolowski1992,henrot2018}.

A given domain $\Omega\subset D$ is perturbed by means of the identity perturbation operator so that for some $t_0:=t_0(\bth)>0$ we get the family of perturbed domains $\{\Omega_t: 0<t<t_0 \}$ with $\Omega_t: = T_t(\Omega)$.  The parameter $t_0>0$ is chosen so that for any $t\in(0,t_0)$,  $\mathrm{det}\nabla T_t >0$ and $J_tM_t(M_t^\top)$ is coercive, i.e., for some $0<\alpha_1 <\alpha_2$
  \[
  	\alpha_1|\xi|^2 \le [J_tM_t(M_t^\top)]\xi\cdot\xi\le \alpha_2|\xi|^2\quad \forall\xi\in \mathbb{R}^{2},
  \]
  where $J_t = \mathrm{det}\nabla T_t$,  $M_t(x) = (\nabla T_t(x))^{-1}$,  and $\nabla T_t$ denotes the Jacobian matrix of the operator $T_t$.

By the definition of $\bth$, $\bth\equiv 0$ on $\Gamma_{\rm out}$, $\Gamma_{\rm in}$, and $\Gamma_{\rm w}$, this implies that these boundaries are part of the perturbed domains $\Omega_t$. To be precise, we have
\[ \partial\Omega_t = \Gamma_{\rm out}\cup\Gamma_{\rm in}\cup\Gamma_{\rm w}\cup T_t(\Gamma_{\rm f}). \]
Additionally, a domain that has at most $C^{1,1}$ regularity preserves its said regularity with this transformation, this means that for $0\le t\le t_0$, $\Omega_t$ has $C^{1,1}$ regularity given that the initial domain $\Omega$ is a $C^{1,1}$ domain. 

Before we move further in this exposition, let us look at some vital properties of $T_t$. 

\begin{lmm}[cf \cite{zolesio2011,sokolowski1992}]
Let $\bth\in\bTh$, then for sufficiently small $t_0>0$, the identity perturbation operator $T_t$ satisfies the following properties:
%\begin{table}[h!]
%\centering
%\begin{tabular}{cl@{\hskip 0.1in}cl}
%$\bullet$ & $[t\mapsto T_t]\in C^1([0,\tau];C^{2,1}(\overline{D},\mathbb{R}^2))$; &$\bullet$& $[t\mapsto T_t^{-1}]\in C([0,\tau];C^{2,1}(\overline{D},\mathbb{R}^2))$;\\
%$\bullet$ & $[t\mapsto J_t]\in C^1([0,\tau];C^{1,1}(\overline{D}))$; &$\bullet$& $M_t,M_t^\top\in C^{1,1}(\overline{D},\mathbb{R}^{2\times2})$;\\
%$\bullet$ & $\frac{d}{dt}T_t\big|_{t=0} = \bth$; &$\bullet$& $\frac{d}{dt}T_t^{-1}\big|_{t=0} = -\bth$;\\
%$\bullet$ & $\frac{d}{dt}\nabla T_t\big|_{t=0} = \nabla\bth$; &$\bullet$& $\frac{d}{dt}M_t\big|_{t=0} = -\nabla\bth$;\\
%$\bullet$ & $\frac{d}{dt}J_t\big|_{t=0} = \dive\bth$; &$\bullet$& $\frac{d}{dt}[J_tM_t(M_t^\top)]\big|_{t=0} = \dive\bth-\nabla\bth-(\nabla\bth)^\top$;\\
%\end{tabular}
%\end{table}
\begin{itemize}
	\item[$\bullet$] $[t\mapsto T_t]\in C^1([0,t_0];C^{2,1}(\overline{D},\mathbb{R}^2));\ \quad\hspace{-.15in} \bullet\ [t\mapsto T_t^{-1}]\in C([0,t_0];C^{2,1}(\overline{D},\mathbb{R}^2));$
	\item[$\bullet$] $[t\mapsto J_t]\in C^1([0,t_0];C^{1,1}(\overline{D}));\quad\qquad\hspace{-.15in}\! \bullet\ M_t,M_t^\top\in C^{1,1}(\overline{D},\mathbb{R}^{2\times2});$
	\item[$\bullet$] $\frac{d}{dt}J_t\big|_{t=0} = \dive\bth;\hspace{-.15in}\quad\qquad\qquad\qquad\quad\ \bullet\ \frac{d}{dt}M_t\big|_{t=0} = -\nabla\bth.$
\end{itemize}
\label{Tprops}
\end{lmm}

Before we move on to the next part, let us first recall Hadamard's identity which will be integral for solving the necessary conditions.
\begin{lmm}
Let $f\in C([0,t_0];W^{1,1}(D))$ and suppose that $\frac{\partial}{\partial t}f(0)\in L^1(D)$, then 
\begin{align*}
	\frac{d}{dt}\int_{\Omega_t}f(t,x)\du x\Big|_{t = 0} = \int_\Omega \frac{\partial}{\partial t}f(0,x)\du x + \int_{\Gamma_{\rm f}} f(0,x)\bth\cdot\bn\du s.
\end{align*}
\label{hadamard}
\end{lmm}
\begin{proof}
See Theorem 5.2.2 and Proposition 5.4.4 of \cite{henrot2018}.
\end{proof}
\subsection{Rearrangement method} 
To investigate the sensitivity of the objective functional with respect to shape variations generated by the transformation $T_t$, we shall resort to a variational approach formalized by K. Ito, K. Kunisch, and G. Peichl \cite{ito2008}, which is known by many as the rearrangement method. This approach gets rid of the tedious process of solving first the sensitivity of the state solutions , then solving the {\it shape derivative} of the objective functional. Aside from the convenience the rearrangement method poses, we also mention that using the usual methods -- such as the chain rule, min-max formulation, etc. -- will not take into account the linearization of the state on the fixed boundary $\Gamma_{\rm out}$. This in turn will render the linearized state and the adjoint equation ill-posed. This problem, thankfully, is resolved by the rearrangment method which is focused on the Frech{\'e}t derivative of the state operator.

To start with,  we consider a Hilbert space $Y(\Omega)$ and an operator
\[E:Y(\Omega)\times \Oad\to Y(\Omega)^*,\]
where the equation $\langle E(y,\Omega), \phi \rangle_{Y(\Omega)^*\times Y(\Omega)} = 0$ corresponds to a variational problem in $\Omega.$

Suppose that the free-boundary is denoted by $\Gamma_{\rm f}\subset \partial\Omega$, and $g:Y(\Omega)\to\mathbb{R}$, the said method deals with the shape optimization 
\[ \min_{\Omega\in \Oad}\mathcal{J}(y,\Omega):=\int_\Omega g(y)dx + \alpha\int_{\Gamma_{\rm f}}\du s\]
subject to 
\begin{align}
 E(y,\Omega)=0\text{ in }Y(\Omega)^*.\label{weakgen}
\end{align}
We define the Eulerian derivative of $J$ at $\Omega$ in the direction $\bth\in\bTh$ by
\[ d\mathcal{J}(y,\Omega)\bth = \lim_{t\searrow0}\frac{\mathcal{J}(y_t,\Omega_t) - \mathcal{J}(y,\Omega)}{t},\]
where $y_t$ solves the equation $E(y_t,\Omega_t)=0$ in $Y(\Omega_t)^*$. If $d\mathcal{J}(y,\Omega)\bth$ exists for all $\bth\in\bTh$ and that $d\mathcal{J}(y,\Omega)$ defines a bounded linear functional on $\bTh$ then we say that $\mathcal{J}$ is {\it shape differentiable} at $\Omega$.

The so-called rearrangement method is given as below:
\begin{lmm}[cf \cite{ito2008}]
Suppose that the followring assumptions hold
\begin{pethau}
	\item[(A1)] There exists an operator $\tilde{E}:Y(\Omega)\times[0,t_0]\to Y(\Omega)^*$ such that $E(y_t,\Omega_t)=0$ in $Y(\Omega_t)^*$ is equivalent to 
	\begin{align}
		\tilde{E}(y^t,t) = 0\text{ in }Y(\Omega)^*,\label{weakbacktrack}
	\end{align}
	 with $\tilde{E}(y,0) = E(y,\Omega)$ for all $y\in Y(\Omega).$
	\item[(A2)]  Let $y,v\in Y(\Omega)$. Then $E_y(y,\Omega)\in \mathcal{L}(Y(\Omega),Y(\Omega)^*)$ satisfies
		\[ \langle E(v,\Omega)-E(y,\Omega)-E_y(y,\Omega)(v-y),z\rangle_{Y(\Omega)^*\times Y(\Omega)} = \mathcal{O}(\|v-y \|_{Y(\Omega)}^2),\]
		for all $z\in Y(\Omega). $
	\item[(A3)] Let  $y\in Y(\Omega)$ be the unique solution of \eqref{weakgen}. Then for any $f\in Y(\Omega)^*$  the solution of the following linearized equation exists:
	\begin{align*}
		\langle E_y(y,\Omega)\delta y,  z\rangle_{Y(\Omega)^*\times Y(\Omega)} = \langle f,  z\rangle_{Y(\Omega)^*\times Y(\Omega)} \text{ for all }z\in Y(\Omega).
	\end{align*}
	\item[(A4)] Let $y^t,y\in Y(\Omega)$ be the solutions of \eqref{weakbacktrack} and \eqref{weakgen}, respectively. Then $\tilde E$ and $E$ satisfy
	\[\lim_{t\searrow 0}\frac{1}{t}\langle \tilde E(y^t,t) - \tilde{E}(y,t) -F(y^t,\Omega) + E(y,\Omega) ,z\rangle_{Y(\Omega)^*\times Y(\Omega)} = 0\]
	for all $z\in Y(\Omega). $
	\item[(A5)] $g\in C^{1,1}(\mathbb{R}^2,\mathbb{R}).$
\end{pethau}
Let $y\in Y(\Omega)$ be the solution of \eqref{weakgen}, and suppose that the adjoint equation,  for all $z\in Y(\Omega)$
\begin{align}
\langle E_y(y,\Omega)z,p\rangle_{Y(\Omega)^*\times Y(\Omega)} = (g'(y),z)_{L^2(\Omega)} 
\label{weakadjoint}
\end{align}
has a unique solution $p\in Y(\Omega)$.  Then the Eulerian derivative of $J$ at $\Omega$ in the direction $\bth\in\bTh$ exists and is given by
\begin{align}
\begin{aligned}
d\mathcal{J}(y,\Omega)\bth =&\, -\frac{d}{dt}\langle \tilde{E}(y,t), p\rangle_{Y(\Omega)^*\times Y(\Omega)}\Big|_{t=0}+ \int_\Omega g(y)\dive\bth \du x + \alpha\int_{\Gamma_{\rm f}}\kappa\bth\cdot\bn \du s,
\end{aligned}
\label{necessaryconditiongen}
\end{align}
where $\kappa$ is the mean curvature of the surface $\Gamma_{\rm f}$.
\end{lmm}

It is noteworthy to mention that originally \cite{ito2008}, assumption {\it (A3)} is instead written as the H{\"o}lder continuity of solutions $y^t\in X(\Omega)$ to \eqref{weakbacktrack} with respect to the time parameter $t\in[0,\tau]$.  Fortunately, from the same paper, Ito, K., et.al.  have shown that assumption {\it (A3)} implies the aforementioned continuity. We cite the said result in the following lemma.
\begin{lmm}
Suppose that $y\in Y(\Omega)$ solves \eqref{weakgen} and $y^t\in Y(\Omega)$ is the solution to \eqref{weakbacktrack}. Assume furthermore that {\it (A3)} holds. Then, $\|y^t-y \|_{Y(\Omega)} = o(t^\frac{1}{2})$ as $t\searrow 0$.
\label{holder}
\end{lmm}
\begin{proof}
See \cite[Proposition 2.1]{ito2008}.
\end{proof}

We begin applying this method by introducing the velocity-pressure operator $\mathbb E(\cdot)_\Omega:X(\Omega)\to X(\Omega)^*  $ defined by
\begin{align*}
\langle\mathbb  E(\bu,p)_\Omega,(\bphi, \psi) \rangle_{X(\Omega)^*\times X(\Omega)} =&  \mathbb{A}(\bu;\bu,\bphi)_\Omega - \frac{1}{2}\mathbb{C}(\bu;\bu,\bphi) + d(\bphi,p)_\Omega\\& + d(\bu,\psi)_\Omega-  \langle\Phi,{\bphi}	\rangle_{{\bH}^{-1}(\Omega)\times\bH^1_{\Gamma_0}(\Omega)}
\end{align*} 
where $X(\Omega) := \bH^1_{\Gamma_0}(\Omega)\times L^2(\Omega)$, ${\bH}^{-1}(\Omega)$ is the dual of $ \bH^1_{\Gamma_0}(\Omega)$.
It can be easily shown that $d(\cdot,\cdot)_\Omega$ satisfies the inf-sup condition, hence there exists $(\tilde\bu, p)\in X(\Omega)$ such that for any $(\bphi, \psi)\in X(\Omega)$ 
\begin{align}
 \langle\mathbb E(\tilde\bu,p)_\Omega,(\bphi, \psi) \rangle_{X(\Omega)^*\times X(\Omega)} = 0.\label{operator}
 \end{align}
The element $\tilde\bu\in  \bH^1_{\Gamma_0}(\Omega)$, in particular, solves the variational equation \eqref{weak}.\\

\noindent{\bf Notation:} {\it Moving forward we shall use the following notations $X = X(\Omega)$,  $X_t = X(\Omega_t)$, and ${\bV}={\bV}(\Omega)$.}\\

Our goal is to characterize, and of course show the existence of the Eulerian derivative of the objective functional 
\begin{align}
\mathcal{G}(\bu,\Omega) =  \alpha\int_{\Gamma_{\rm f}} \du s-\frac{ \gamma}{2}\int_{\Omega}|\nabla\times \bu|^2 \du x.
\end{align}

Now,  from the deformation field $T_t$, we let $(\tilde\bu_t,p_t)\in X_t$ be the solution of the equation
\begin{align} 
\langle\mathbb E(\tilde\bu_t,p_t)_{\Omega_t},(\bphi_t, \psi_t) \rangle_{X_t^*\times X_t} = 0 \text{ for all }(\bphi_t, \psi_t)\in X_t.
\label{operatoront}
\end{align}
We perturb equation \eqref{operatoront} back to the reference domain $\Omega$ which gives us the operator $\tilde{\mathbb E}:X\times[0,\tau]\to X^*$ defined by
\begin{align}
\begin{aligned}
\langle\tilde{\mathbb E}((\bu,p),t),(\bphi,\psi) \rangle_{X^*\times X}  =&\, \mathbb{A}_t(\bu;\bu,\bphi)-  \frac{1}{2}\mathbb{C}(\bu;\bu,\bphi) + d_t(\bphi,p)\\ &+ d_t(\bu,\psi) -  \langle\Phi^t,{\bphi}	\rangle_{{\bH}^{-1}(\Omega)\times\bH^1_{\Gamma_0}(\Omega)}
\end{aligned}
\end{align}
where -- by denoting $(M_t^\top)_k$ the k$^{th}$ row of $M_t^\top$, and $\bv_k$ the k$^{th}$ component of a vector $\bv$ -- the components are defined as follows
\begin{align*}
\mathbb{A}_t(\bu;\bu,\bphi) =&\,  (J_tM_t\nabla{\bu},M_t\nabla{\bphi})_\Omega + (J_t{\bu}\cdot M_t\nabla{\bg},{\bphi})_\Omega \\ &+ (J_t{\bu}\cdot M_t\nabla{\bu},{\bphi} )_\Omega + (J_t{\bg}\cdot M_t\nabla{\bu},{\bphi} )_\Omega,\\					
d_t(\bv,\psi) = &\, -\sum_{k = 1}^2 (J_t \psi ,(M_t^\top)_k\nabla\bv_k)_\Omega,
\end{align*}
and the element $\Phi^t\in{\bH}^{-1}(\Omega) $ is defined by
\begin{align*} 
 \langle\Phi^t,{\bphi}	\rangle_{{\bH}^{-1}(\Omega)\times\bH^1_{\Gamma_0}(\Omega)} =&\, ({\blf}\circ T_t - J_t{\bg}\cdot M_t\nabla{\bg},{\bphi})_\Omega\\ &  - \nu(J_t M_t\nabla{\bg},M_t\nabla{\bphi})_{\Omega} + \frac{1}{2}({\bg}\cdot{\bn},{\bg}\cdot{\bphi})_{\Gamma_{\rm out}}.
 \end{align*}
By construction, this operator satisfies {\it(A1)}, in particular if $(\tilde\bu_t,p_t)\in X_t$ solves \eqref{operatoront}, then the translated element $(\tilde{\bu}^t,p^t)=(\tilde\bu_t\circ T_t,p_t\circ T_t)\in X$ solves 
\begin{align} 
\langle\tilde{\mathbb E}((\tilde\bu^t,p^t),t),(\bphi,\psi) \rangle_{X^*\times X} = 0 \text{ for all }(\bphi, \psi)\in X.
\label{backtrackop}
\end{align}
Indeed, for sufficiently small values of $t>0$ the coercivity of $\mathbb{A}_t(\cdot)-\frac{1}{2}\mathbb{C}(\cdot)$ can be easily verified, as well as the inf-sup condition for the bilinear form $d_t$.

For property {\it(A2)}, we introduce the linearization of the Stokes operator $A_\Omega:\bV\cap H^2(\Omega)^2 \to \bV^*$,  and of the bilinear operators $B_\Omega:\bV\times\bV\to \bV^*$ and $C:\bV\times\bV\to H^{-\frac{1}{2}}(\Gamma_{\rm out})^2$ which were briefly discussed at the end of Section \ref{sect2}. 
\begin{prpstn}
The Fr{\' e}chet derivative of the operators $A_\Omega,$ $B_\Omega$ and $C$ at the point $\bu\in \bV$ in the direction $\delta\!\bu\in\bV$ are given as follows:
\begin{itemize}
	\item[1.] $\langle D_uA_\Omega(\bu)\delta\!\bu, \bphi\rangle_{\bV^*\times\bV} = a(\delta\!\bu,\bphi)_{\Omega}$;
	\item[2.] $\langle D_uB_\Omega\bu\delta\!\bu, \bphi\rangle_{\bV^*\times\bV} = b(\bu;\delta\!\bu,\bphi)_{\Omega} +b(\delta\!\bu;\bu,\bphi)_{\Omega}$;
	\item[3.] $\langle D_uB_\Omega(\bu,\bv)\delta\!\bu, \bphi\rangle_{\bV^*\times\bV} = b(\delta\!\bu;\bv,\bphi)_{\Omega}$;
	\item[4.] $\langle D_uB_\Omega(\bv,\bu)\delta\!\bu, \bphi\rangle_{\bV^*\times\bV} = b(\bv;\delta\!\bu,\bphi)_{\Omega}$;
	\item[5.] $\langle D_u C\!\bu\delta\!\bu,\bphi \rangle_{\bH^{-\frac{1}{2}}\times \bH^{\frac{1}{2}}} = \langle C(\bu,\delta\!\bu) + C(\delta\!\bu,\bu),\bphi \rangle_{\bH^{-\frac{1}{2}}\times \bH^{\frac{1}{2}}};$
	\item[6.] $\langle D_u C(\bu,\bv)\delta\!\bu,\bphi \rangle_{\bH^{-\frac{1}{2}}\times \bH^{\frac{1}{2}}} =\langle C(\delta\!\bu,\bv),\bphi \rangle_{\bH^{-\frac{1}{2}}\times \bH^{\frac{1}{2}}};$
	\item[7.] $\langle D_u C(\bv,\bu)\delta\!\bu,\bphi \rangle_{\bH^{-\frac{1}{2}}\times \bH^{\frac{1}{2}}} = \langle C(\bv,\delta\!\bu),\bphi \rangle_{\bH^{-\frac{1}{2}}\times \bH^{\frac{1}{2}}},$
\end{itemize}
where we used the notations $C\!\bu = C(\bu,\bu)$, $\bH^{\frac{1}{2}} = H^{\frac{1}{2}}(\Gamma_{\rm out})^2$ and $\bH^{-\frac{1}{2}} = H^{-\frac{1}{2}}(\Gamma_{\rm out})^2$.
\label{frechet}
\end{prpstn}
\begin{proof} We only expose the parts where the nonlinearity occurs, i.e., we show that 
	\[ \langle D_uB_\Omega\bu\delta\!\bu, \bphi\rangle_{\bV^*\times\bV} = b(\bu;\delta\!\bu,\bphi)_{\Omega} +b(\delta\!\bu;\bu,\bphi)_{\Omega}\]
	and 
	\[\langle D_u C\!\bu\delta\!\bu,\bphi \rangle_{\bH^{-\frac{1}{2}}\times \bH^{\frac{1}{2}}}\! = \langle C(\bu,\delta\!\bu),\bphi \rangle_{\bH^{-\frac{1}{2}}\times \bH^{\frac{1}{2}}}+\langle C(\delta\!\bu,\bu),\bphi \rangle_{\bH^{-\frac{1}{2}}\times \bH^{\frac{1}{2}}}. \]
	Indeed, we have the following computations
	\begin{align*}
		\langle B_\Omega(\delta\!\bu + \bu) - B_\Omega\bu-d_uB_\Omega\bu\delta\!\bu, \bphi\rangle_{\bV^*\times\bV}= &\, b(\delta\!\bu;\delta\!\bu,\bphi)_{\Omega} \le c\|\delta\!\bu\|_{\bV(\Omega)}^2\|\bphi\|_{\bV(\Omega)}
	\end{align*}
%	\begin{align*}
%		\langle B_\Omega(\delta\!\bu + \bu) - B_\Omega\bu-d_uB_\Omega\bu\delta\!\bu, \bphi\rangle_{\bV'\times\bV}=&\, b(\delta\!\bu+\bu;\delta\!\bu+\bu,\bphi)_{\Omega} - b(\bu;\bu,\bphi)_{\Omega} - b(\bu;\delta\!\bu,\bphi)_{\Omega} -b(\delta\!\bu;\bu,\bphi)_{\Omega}\\
%		= &\, b(\delta\!\bu;\delta\!\bu,\bphi)_{\Omega} \le c\|\delta\!\bu\|_{\bV(\Omega)}^2\|\bphi\|_{\bV(\Omega)}
%	\end{align*}
and 
\begin{align*}
		\langle C(\delta\!\bu + \bu) - C\!\bu-d_uC\!\bu\delta\!\bu, \bphi\rangle_{\bH^{-\frac{1}{2}}\times \bH^{\frac{1}{2}}} & = \frac{1}{2}(\delta\!\bu\cdot\bn,\delta{\bu}\cdot\bphi)_{\Gamma_{\rm out}}\\ & \le c\|\delta\!\bu\|_{\bV(\Omega)}^2\|\bphi\|_{\bH^{\frac{1}{2}}}.
	\end{align*}
%	\begin{align*}
%		\langle C(\delta\!\bu + \bu) - C\!\bu-d_uC\!\bu\delta\!\bu, \bphi\rangle_{\bH^{-\frac{1}{2}}\times \bH^{\frac{1}{2}}}=&\,  \frac{1}{2}((\delta\!\bu+ \bu)\cdot\bn,(\delta\!\bu  + \bu)\cdot\bphi)_{\Gamma_{\rm out}} -\frac{1}{2} ( \bu\cdot\bn,\bu\cdot\bphi)_{\Gamma_{\rm out}} \\
%		& - \frac{1}{2}\big[(\delta{\bu}\cdot\bn,\bu\cdot\bphi)_{\Gamma_{\rm out}}+(\bu\cdot\bn,\delta\!\bu\cdot\bphi)_{\Gamma_{\rm out}} \big]\\
%		= &\, \frac{1}{2}(\delta\!\bu\cdot\bn,\delta{\bu}\cdot\bphi)_{\Gamma_{\rm out}} \le c\|\delta\!\bu\|_{\bV(\Omega)}^2\|\bphi\|_{\bH^{\frac{1}{2}}}.
%	\end{align*}
	On the last inequality, we used Rellich--Kondrachov embedding $H^1(\Omega)^2 \hookrightarrow L^q(\partial\Omega)^2$ for $q\ge2$.
	\end{proof}

%%chang X(Omega) to X
With these derivatives, we can determine the Fr{\'e}chet derivative of $\mathbb E(\cdot)_\Omega$ at an element $(\bu,p)\in X$ in the direction $(\delta\!\bu,\delta p)\in X$, which we shall denote as $\mathbb E'(\bu,p) \in\mathcal{L}(X,X^*)$, given by
\begin{align*}
	\langle\mathbb E'(\bu,p)(\delta\!\bu,\delta p),&(\bphi,\psi)\rangle_{X^*\times X}  = \langle [D_uA_\Omega\!\bu+ D_uB_\Omega\!\bu + D_uB_\Omega(\bu,\bg)]\delta\!\bu, \bphi\rangle_{\bV^*\times\bV}\\ 
	&\, + 	\langle B_\Omega(\bg,\bu)\delta\!\bu, \bphi\rangle_{\bV^*\times\bV}  - \langle [D_uC\!\bu + D_uC(\bg,\bu)]\delta\!\bu,\bphi \rangle_{\bH^{-\frac{1}{2}}\times \bH^{\frac{1}{2}}}\\
	 & + d(\bphi,\delta p)_\Omega + d(\delta\!\bu,\psi)_\Omega.
\end{align*}
%With these derivatives, we can express the linearization of the operator $E(\cdot)_\Omega$ about an element $(\bu,p)\in X,$ which we shall denote as $E'(\bu,p) \in\mathcal{L}(X,X^*).$
%First, we write $E(\cdot)_\Omega$ as the sum of the operators we just mentioned above, that is
%\begin{align*}
%	\langle E({\bu},p)_\Omega, (\bphi,\psi)\rangle_{X^*\times X} =& \langle A_\Omega{\bu} + B_\Omega{\bu} + B_\Omega({\bu},{\bg}) + B_\Omega({\bg},{\bu}), \bphi\rangle_{\bV'\times\bV}- \frac{1}{2}\langle C{\bu} + C({\bg},{\bu}),\bphi \rangle_{\bH^{-\frac{1}{2}}\times \bH^{\frac{1}{2}}}\\ & -\frac{1}{2}\langle  C({\bu},{\bg}),\bphi \rangle_{\bH^{-\frac{1}{2}}\times \bH^{\frac{1}{2}}} + d(\bphi,p)_\Omega+ d(\bu,\psi)_\Omega-  \langle\Phi,{\bphi}\rangle_{{\bH}^{-1}(\Omega)\times\bH^1_{\Gamma_0}(\Omega)}
%\end{align*}
%Hence,  the Fr{\'e}chet derivative of $E(\cdot)_\Omega$ at an element $(\bu,p)\in X$ in the direction $(\delta\!\bu,\delta p)\in X$, is given by
%\begin{align*}
%	\langle E'(\bu,p)(\delta\!\bu,\delta p),(\bphi,\psi)\rangle_{X^*\times X}  =&\,\langle [d_uA_\Omega\!\bu+ d_uB_\Omega\!\bu + d_uB_\Omega(\bu,\bg)+ B_\Omega(\bg,\bu)]\delta\!\bu, \bphi\rangle_{\bV'\times\bV}\\ & - \langle [d_uC\!\bu + d_uC(\bg,\bu)]\delta\!\bu,\bphi \rangle_{\bH^{-\frac{1}{2}}\times \bH^{\frac{1}{2}}} + d(\bphi,\delta p)_\Omega + d(\delta\!\bu,\psi)_\Omega
%\end{align*}
This gives us {\it (A2)}. Indeed, by following the proof Proposition \ref{frechet}
\begin{align*}
	\langle\mathbb E(\delta\!\bu+\bu,&\delta p+p) -\mathbb E(\bu,p) -\mathbb E'(\bu,p)(\delta\!\bu,\delta p),(\bphi,\psi)\rangle_{X^*\times X}\\
	= &\, \langle B_\Omega\delta\!\bu, \bphi\rangle_{\bV^*\times\bV} - \langle C\delta\!\bu, \bphi\rangle_{\bH^{-\frac{1}{2}}\times \bH^{\frac{1}{2}}}\le \tilde c\|\delta\!\bu\|_{\bV}^2\|\delta\!\bphi\|_{\bV}
\end{align*}
Furthermore, by the same arguments as done for the proof of Theorem \ref{th:wp} and since $d(\cdot,\cdot)$ satisfies the inf-sup condition,   for any $\mathcal{F}\in X^*$, there exists a unique solution $(\delta\!\bu,\delta p)\in X$ to 
\[ \langle\mathbb E'(\bu,p)(\delta\!\bu,\delta p),(\bphi,\psi)\rangle_{X^*\times X} = \langle \mathcal{F},(\bphi,\psi)\rangle_{X^*\times X},\]
for all $(\bphi,\psi)\in X.$ This implies {\it (A3)}.

To verify {\it (A4)},  we note that 
\begin{align*}
	\langle \tilde{\mathbb E}&((\tilde\bu^t,p^t),t)-\tilde{\mathbb E}((\tilde\bu,p),t) - \mathbb E(\tilde\bu^t,p^t)_\Omega +\mathbb E(\tilde\bu,p)_\Omega,(\bphi,\psi)\rangle_{X^*\times X} \\
	= &\, \nu\left[\big((J_tM_t - I)\nabla(\tilde\bu^t-\tilde\bu)\!:\!\nabla\bphi\big)_\Omega + \big(\nabla(\tilde\bu^t-\tilde\bu)\!:\!(M_t - I)\nabla\bphi\big)_{\Omega}\right]\\
	&+\big( (\tilde\bu^t-\tilde\bu)\cdot(J_tM_t - I)\nabla\tilde\bu^t,\bphi\big)_\Omega +\big( (\tilde\bu^t-\tilde\bu)\cdot(J_tM_t - I)\nabla\bg,\bphi\big)_\Omega  \\
	& +\big( \tilde\bu\cdot(J_tM_t - I)\nabla(\tilde\bu^t-\tilde\bu),\bphi\big)_\Omega +\big( \bg\cdot(J_tM_t - I)\nabla(\tilde\bu^t-\tilde\bu),\bphi\big)_\Omega \\
	& - \sum_{k=1}^2\left[\big( (p^t-p)(J_tM_t-I)e_k,\nabla\bphi_k\big)_\Omega + \big( \psi(J_tM_t-I)e_k,\nabla(\tilde\bu_k^t - \tilde\bu_k)\big)_\Omega  \right]\\
	\le &\, \nu( \|J_tM_t-I\|_\infty+\|M_t-I\|_\infty) \|\tilde\bu^t - \tilde\bu\|_{H^1(\Omega)^2}\|\bphi\|_{H^1(\Omega)^2}\\
	& +  \|J_tM_t-I\|_\infty(\|\tilde\bu^t\|_{H^1(\Omega)^2}+\|\bg\|_{H^1(\Omega)^2})\|\tilde\bu^t - \tilde\bu\|_{H^1(\Omega)^2}\|\bphi\|_{H^1(\Omega)^2}\\
	& +  \|J_tM_t-I\|_\infty(\|\tilde\bu\|_{H^1(\Omega)^2}+\|\bg\|_{H^1(\Omega)^2})\|\tilde\bu^t - \tilde\bu\|_{H^1(\Omega)^2}\|\bphi\|_{H^1(\Omega)^2}\\
	& + \|J_tM_t-I\|_\infty(\|p^t-p\|_{L^2(\Omega)^2} \|\bphi\|_{H^1(\Omega)^2}+ \|\psi\|_{L^2(\Omega)^2} \|\tilde\bu^t - \tilde\bu\|_{H^1(\Omega)^2})
\end{align*}
Thus, by dividing the previous computation by $t$ and by utilizing Lemmata \ref{Tprops} and \ref{holder},  we infer that $\mathbb E(\cdot)_\Omega$ and $\tilde{\mathbb E}$ satisfy {\it (A4)}.

Having been able to show that the operator $\mathbb E(\cdot)_\Omega$ and $\tilde{\mathbb E}$ satisfy {\it (A1)-(A4)}, and from the fact that the objective functional mostly consists of squared-norms, i.e.,  $P,J\in C^\infty(X,\mathbb{R}),$ the last remaining task to assure the existence of the Eulerian derivative of $\mathcal{G}$ is the unique existence of the solution $(\bv,\pi)\in X$ to the adjoint problem
\begin{align} 
\langle\mathbb E'(\tilde \bu,p)(\bphi,\psi),(\bv,\pi)\rangle_{X^*\times X} =  (G'(\tilde\bu), \bphi)_{L^2(\Omega)^2}\label{adjoint}
\end{align}
where $G'$ corresponds to the derivative of the integrand of $\mathcal{G}$ with respect to $\bu$ whose action is given as 
\[ (G'(\tilde\bu), \bphi)_{L^2(\Omega)^2} = -\gamma(\vec{\nabla}\times(\nabla\times(\tilde\bu+\bg)), \bphi)_{L^2(\Omega)^2}. \]
The unique existence of the adjoint variables $(\bv,\pi)\in X$ can quite easily be established following the arguments of Theorem \ref{th:wp}.  Furthermore,  from the regularity of the domain, the adjoint solution $\bv$ satisfies ${\bv}\in{\bV}\cap H^2(\Omega)^2$, and that the solution can be looked at as the solution to the system
\begin{align}
	\left\{
		\begin{aligned}
			-\Delta{\bv} + (\nabla{\bu})^{\top}{\bv} -({\bu}\cdot\nabla){\bv} + \nabla\pi & = -\gamma\vec{\nabla}\times(\nabla\times{\bu}) &&\text{ in }\Omega\\
			\dive{\bv} & = 0 &&\text{ in }\Omega\\
			{\bv} & = 0 &&\text{ on }\partial\Omega\backslash\Gamma_{\rm out}\\
			-\pi{\bn}+\nu\partial_{\bn}{\bv} & = \frac{1}{2}\big[({\bu}\cdot{\bv}){\bn} - ({\bu}\cdot{\bn}){\bv} \big] &&\text{ on }\Gamma_{\rm out}.
 		\end{aligned}
	\right.
	\label{strong_adjoint}
\end{align}

To finally characterize the shape derivative of $\mathcal{G}$, we start by evaluating
\[\mathcal{D}_t\tilde{\mathbb E}:=\frac{d}{dt}\langle\tilde{\mathbb E}((\tilde{\bu},p),t),({\bv},\pi) \rangle_{X^*\times X}\big|_{t=0},\]
where $(\tilde{\bu},p),({\bv},\pi) \in X$ solve \eqref{operator} and \eqref{adjoint}, respectively. To do this, we write the operator $\tilde{\mathbb E}$ with the pushed-forward form on $\Omega_t$ and utilize Lemma \ref{hadamard}, that is, we determine the derivative of 
\begin{align*}
\langle\tilde{\mathbb E}((\tilde{\bu},p),t),({\bv},\pi) \rangle_{X^*\times X} =&\, \mathbb{A}(\tilde{\bu}\circ T_t^{-1};\tilde{\bu}\circ T_t^{-1},{\bv}\circ T_t^{-1})_{\Omega_t} - \frac{1}{2}\mathbb{C}(\tilde{\bu};\tilde{\bu},{\bv})- \frac{1}{2}({\bg}\cdot{\bn},{\bg}\cdot{\bv})_{\Gamma_{\rm out}}\\
	&+ d({\bv}\circ T_t^{-1},p\circ T_t^{-1})_{\Omega_t} + d(\tilde{\bu}\circ T_t^{-1},\pi\circ T_t^{-1})_{\Omega_t} + d({\bg}\circ T_t^{-1},\pi\circ T_t^{-1})_{\Omega_t}\\
	&- ({\blf} - ({\bg}\circ T_t^{-1})\cdot\nabla({\bg}\circ T_t^{-1}),{\bv}\circ T_t^{-1} )_{\Omega}+ \nu (\nabla({\bg}\circ T_t^{-1}),\nabla({\bv}\circ T_t^{-1}))_\Omega .
\end{align*}
We note that in the expression of $\tilde{\mathbb E}$ above, we added the term $d({\bg}\circ T_t^{-1},\pi\circ T_t^{-1})_{\Omega_t}$ since the divergence of $\bg$ is zero in $\Omega$. Denoting $\boldsymbol{\psi}_{\tilde{u}} = -(\nabla\tilde\bu)^\top{\bth}$, $\boldsymbol{\psi}_{g} = -(\nabla\bg)^\top{\bth}$, $\boldsymbol{\psi}_v = -(\nabla\bv)^\top{\bth}$, $\psi_p = -\nabla p\cdot{\bth}$, and $\psi_{\pi} = -\nabla\pi\cdot{\bth}$ yields
\begin{align*}
	\mathcal{D}_t\tilde{\mathbb E} = &\,  \nu (\nabla\tilde{\bu}:\nabla{\bv},{\bth}\cdot{\bn})_{\Gamma_{\rm f}}+ \nu a(\boldsymbol{\psi}_{\tilde{u}},{\bv})_{\Omega}+ \nu a(\tilde{\bu},\boldsymbol{\psi}_v)_{\Omega}  \\
										&+ ((\tilde{\bu}\cdot\nabla{\bg}) \cdot {\bv},{\bth}\cdot{\bn})_{\Gamma_{\rm f}} + b(\boldsymbol{\psi}_{\tilde{u}};{\bg},{\bv})_{\Omega} + b(\tilde{\bu};\boldsymbol{\psi}_{g},{\bv})_{\Omega}  + b(\tilde{\bu};{\bg},\boldsymbol{\psi}_v)_{\Omega}\\
										& + (({\bg}\cdot\nabla\tilde{\bu}) \cdot {\bv} ,{\bth}\cdot{\bn})_{\Gamma_{\rm f}} + b(\boldsymbol{\psi}_{g};\tilde{\bu},{\bv})_{\Omega}+ b({\bg};\boldsymbol{\psi}_{\tilde{u}},{\bv})_{\Omega} + b({\bg};\tilde{\bu},\boldsymbol{\psi}_v)_{\Omega}\\
										&+ ((\tilde{\bu}\cdot\nabla\tilde{\bu}) \cdot {\bv},{\bth}\cdot{\bn})_{\Gamma_{\rm f}} + b(\boldsymbol{\psi}_{\tilde{u}};\tilde{\bu},{\bv})_{\Omega} +  b(\tilde{\bu};\boldsymbol{\psi}_{\tilde{u}},{\bv})_{\Omega} + b(\tilde\bu;\tilde\bu,\boldsymbol{\psi}_v)_{\Omega}\\ 
										&+ (({\bg}\cdot\nabla{\bg}) \cdot {\bv},{\bth}\cdot{\bn})_{\Gamma_{\rm f}} + b(\boldsymbol{\psi}_{g};{\bg},{\bv})_{\Omega} +  b({\bg};\boldsymbol{\psi}_{g},{\bv})_{\Omega} + b({\bg};{\bg},\boldsymbol{\psi}_v)_{\Omega}\\ 
										& -(p\dive{\bv},{\bth}\cdot{\bn})_{\Gamma_{\rm f}}  + d({\bv},\psi_p)_{\Omega} + d(\boldsymbol{\psi}_v,p)_{\Omega}- (\pi\dive\tilde{\bu},{\bth}\cdot{\bn})_{\Gamma_{\rm f}}\\
										&   + d(\boldsymbol{\psi}_{\tilde{u}},\pi)_{\Omega} + d(\tilde{\bu},\psi_{\pi})_{\Omega} +d(\boldsymbol{\psi}_{g},\pi)_{\Omega} + d({\bg},\psi_{\pi})_{\Omega}  - ({\blf}\cdot{\bv},{\bth}\cdot{\bn})_{\Gamma_{\rm f}} \\
										& - ({\blf},\boldsymbol{\psi}_v)_{\Omega} +\nu(\nabla{\bg}:\nabla{\bv},{\bth}\cdot{\bn})_{\Gamma_{\rm f}}  + \nu a(\boldsymbol{\psi}_{g},{\bv})_{\Omega} + \nu a({\bg},\boldsymbol{\psi}_v) .
\end{align*}
Since $\tilde\bu = \bv = 0$ on $\Gamma_{\rm f}$,  and $\dive \tilde\bu = \dive\bv = 0$,  the integrals on the boundary $\Gamma_{\rm f}$ vanish except for $(\nabla\tilde{\bu}:\nabla{\bv},{\bth}\cdot{\bn})_{\Gamma_{\rm f}}$ and $(\nabla{\bg}:\nabla{\bv},{\bth}\cdot{\bn})_{\Gamma_{\rm f}}$.  Furthermore, since $({\bu},p)=(\tilde{\bu}+{\bg},p)$, we get
\begin{align*}
	\mathcal{D}_t\tilde{\mathbb E} = &\,  \nu (\nabla{\bu}:\nabla{\bv},{\bth}\cdot{\bn})_{\Gamma_{\rm f}}+ \nu a(\boldsymbol{\psi}_{u},{\bv})_{\Omega}+ \nu a({\bu},\boldsymbol{\psi}_v)_{\Omega}  \\
										&+ b(\boldsymbol{\psi}_{u};{\bu},{\bv})_{\Omega} + b({\bu};\boldsymbol{\psi}_{u},{\bv})_{\Omega}  + b({\bu};{\bu},\boldsymbol{\psi}_v)_{\Omega}\\
										& + d({\bv},\psi_p)_{\Omega} + d(\boldsymbol{\psi}_v,p)_{\Omega} + d(\boldsymbol{\psi}_{u},\pi)_{\Omega} + d({\bu},\psi_{\pi})_{\Omega}\\
										&    - ({\blf},\boldsymbol{\psi}_v)_{\Omega}.
\end{align*}

 To further simplify the expression above, we take into account the facts that $({\bu},p)$ satisfies \eqref{strong_nontranslated}, and $({\bv},\pi)$ on the other hand solves \eqref{strong_adjoint}. We begin the simplification on the terms with $\boldsymbol{\psi}_v$ and $\psi_{\pi}$, which we shall denote by $I_1$.
\begin{align*}
	I_1 =&\, \nu a({\bu},\boldsymbol{\psi}_v)_{\Omega} + b({\bu};{\bu},{\bu})_{\Omega} + d(\boldsymbol{\psi}_v,p)_{\Omega} + d({\bu},\psi_{\pi})_{\Omega} - ({\blf},\boldsymbol{\psi}_v)_{\Omega} \\
		= &\, (-\nu\Delta{\bu} +{\bu}\cdot\nabla{\bu} + \nabla p - {\blf},\boldsymbol{\psi}_v)_{\Omega}+(\nu\partial_{\bn}{\bu} - p{\bn},\boldsymbol{\psi}_v)_{\Gamma_{\rm f}}\\
		= &\, -((\nu\partial_{\bn}{\bu} - p{\bn})\cdot\partial_{\bn}{\bv}, {\bth}\cdot{\bn})_{\Gamma_{\rm f}}.
\end{align*}
The computation above utilized the divergence-free property of $\tilde{\bu}$ and the identity $(\nabla\bv)^\top{\bth} = \partial_{\bn}{\bv}({\bth}\cdot{\bn})$ on $\Gamma_{\rm f}$. Furthermore, the boundary integral on $\partial\Omega$ is simplified into just the integral on $\Gamma_{\rm f}$ since $\bth=0$ on $\partial\Omega\backslash\Gamma_{\rm f}.$ Similarly, if we denote by $I_2$ the terms containing $\psi_u$ and $\psi_p$, we have the following simplification,
\begin{align*}
	I_2 = &\, \nu a(\boldsymbol{\psi}_u,{\bv})_{\Omega} + b(\boldsymbol{\psi}_u;{\bu},{\bv})_{\Omega} +b({\bu};\boldsymbol{\psi}_u,{\bv})_{\Omega} + d({\bv},\psi_p)_{\Omega} + d(\boldsymbol{\psi}_u,\pi)_{\Omega} \\
	= &\, (-\nu\Delta{\bv} + (\nabla{\bu})^\top{\bv} + \nabla\pi, \boldsymbol{\psi}_u)_\Omega - b({\bu};{\bv},\boldsymbol{\psi}_u)_{\Omega} + ({\bu}\cdot{\bn},{\bv}\cdot\boldsymbol{\psi}_u)_{\partial\Omega}\\
	& + (\nu\partial_{\bn}{\bv}-\pi{\bn},\boldsymbol{\psi}_u)_{\partial\Omega}\\
	= &\, -(-\nu\Delta{\bv} + (\nabla{\bu})^\top{\bv} - {\bu}\cdot\nabla{\bv} + \nabla\pi, (\nabla{\bu})^\top{\bth})_\Omega - ((\nu\partial_{\bn}{\bv}-\pi{\bn})\cdot\partial_{\bn}{\bu},{\bth}\cdot{\bn} )_{\Gamma_{\rm f}}\\
	= &\, \gamma(\vec{\nabla}\times(\nabla\times{\bu}), (\nabla{\bu})^\top{\bth})_\Omega - ((\nu\partial_{\bn}{\bv}-\pi{\bn})\cdot\partial_{\bn}{\bu},{\bth}\cdot{\bn} )_{\Gamma_{\rm f}}\\
	= &\, \gamma(\nabla\times{\bu}, \nabla\times((\nabla{\bu})^\top{\bth}))_\Omega -\gamma((\nabla\times{\bu}){\btau}\cdot\partial_{\bn}{\bu},{\bth}\cdot{\bn})_{\Gamma_{\rm f}}\\& - ((\nu\partial_{\bn}{\bv}-\pi{\bn})\cdot\partial_{\bn}{\bu},{\bth}\cdot{\bn} )_{\Gamma_{\rm f}}.
\end{align*}
We note that the last equality is achieved using Green's curl identity \cite[Theorem 3.29]{monk2003}.  Combining the expressions obtained from simplifying $I_1$ and $I_2$, and since ${\bu}={\bv} = 0$ on $\Gamma_{\rm f}$, then we can further simplify $\mathcal{D}_t\tilde{\mathbb E}$ as follows:
\begin{align*}
	\mathcal{D}_t\tilde{\mathbb E} = &\, \nu (\nabla{\bu}:\nabla{\bv},{\bth}\cdot{\bn})_{\Gamma_{\rm f}} + I_1 + I_2\\ 
	= &\, \nu(\partial_{\bn}{\bu}\cdot\partial_{\bn}{\bv},{\bth}\cdot{\bn})_{\Gamma_{\rm f}} -((\nu\partial_{\bn}{\bu} - p{\bn})\cdot\partial_{\bn}{\bv}, {\bth}\cdot{\bn})_{\Gamma_{\rm f}}\\
	& +\gamma(\nabla\times{\bu}, \nabla\times((\nabla{\bu})^\top{\bth}))_\Omega -\gamma((\nabla\times{\bu}){\btau}\cdot\partial_{\bn}{\bu},{\bth}\cdot{\bn})_{\Gamma_{\rm f}}\\& - ((\nu\partial_{\bn}{\bv}-\pi{\bn})\cdot\partial_{\bn}{\bu},{\bth}\cdot{\bn} )_{\Gamma_{\rm f}}\\
	= &\, -(\partial_{\bn}{\bu}\cdot(\nu\partial_{\bn}{\bv}  + (\nabla\times{\bu}){\btau}) -\pi{\bn}\cdot\partial_{\bn}{\bu} - p{\bn}\cdot\partial_{\bn}{\bv},{\bth}\cdot{\bn})_{\Gamma_{\rm f}}\\
	& +\gamma(\nabla\times{\bu}, \nabla\times((\nabla{\bu})^\top{\bth}))_\Omega . 
\end{align*}
Since $\dive{\bu}=\dive{\bv}=0$ in $\Omega$, and ${\bu}={\bv} =0$ on $\Gamma_{\rm f}$, by utilizing the definition of tangential divergence \cite[p. 82]{sokolowski1992} the integrals $(\pi{\bn}\cdot\partial_{\bn}{\bu},{\bth}\cdot{\bn})_{\Gamma_{\rm f}}$ and $(p{\bn}\cdot\partial_{\bn}{\bv},{\bth}\cdot{\bn})_{\Gamma_{\rm f}}$ both equate to zero. Hence, we get the following derivative,
\begin{align}
\begin{aligned}
\frac{d}{dt}\langle\tilde{\mathbb E}((\tilde{\bu},p),t),({\bv},\pi) \rangle_{X^*\times X}\big|_{t=0} =&\, -(\partial_{\bn}{\bu}\cdot(\nu\partial_{\bn}{\bv}  + (\nabla\times{\bu}){\btau}),{\bth}\cdot{\bn})_{\Gamma_{\rm f}}\\
& +\gamma(\nabla\times{\bu}, \nabla\times((\nabla{\bu})^\top{\bth}))_\Omega.
\end{aligned}
\label{derE}
\end{align}

With the computation of the derivative above, we finally characterize the Eulerian derivative of $\mathcal{G}$ which is shown in the following theorem.

\begin{thrm}
	Let $\Omega\subset D$ satisfy \eqref{domainassumption}, ${\blf}\in L^2(\Omega)^2$, and ${\bg}\in H^2(\Omega)$ satisfying \eqref{gprop}. Suppose furthermore that \eqref{est:uniqueness} hold, and that $(\tilde{\bu},p),({\bv},\pi)\in (\bH_{\Gamma_0}^1(\Omega)\cap H^2(\Omega)^2)\times L^2(\Omega)$ are the unique solutions of \eqref{operator} and \eqref{adjoint}, respectively. Then for any ${\bth}\in{\bTh}$ the Eulerian derivative of $\mathcal{G}$ exists and is characterized as 
	\begin{align*}
		d\mathcal{G}(({\bu},p),\Omega){\bth} = \int_{\Gamma_{\rm f}}\Big[ \alpha\kappa - \frac{\gamma}{2}|\nabla\times{\bu}|^2 + \frac{\partial{\bu}}{\partial{\bn}}\Big(\nu\frac{\partial{\bv}}{\partial{\bn}} + \gamma(\nabla\times{\bu})\btau \Big)\Big]{\bth}\cdot{\bn} \du s,
	\end{align*}
	where ${\bu} = \tilde{\bu}+{\bg}\in H^2(\Omega)^2$, and $\btau$ is the unit tangential vector on $\Gamma_{\rm f}$.
\end{thrm}
\begin{proof}
	The existence of the Eulerian derivative is implied due to the satisfaction of properties {\it (A1)-(A5)}.  Substituting \eqref{derE} into \eqref{necessaryconditiongen} with $\tilde{E}=\tilde{\mathbb E}$, and $y = (\tilde{u},p)$, we get
	\begin{align}
	\begin{aligned}
		d\mathcal{G}(({\bu},p),&\Omega){\bth} =(\alpha\kappa+\partial_{\bn}{\bu}\cdot(\nu\partial_{\bn}{\bv}  + (\nabla\times{\bu}){\btau}),{\bth}\cdot{\bn})_{\Gamma_{\rm f}}\\
& -\gamma(\nabla\times{\bu}, \nabla\times((\nabla{\bu})^\top{\bth}))_\Omega + (j({\bu}),\dive{\bth})_{\Omega},
\end{aligned}
\label{dG1}
	\end{align}
where $j({\bu})$ is such that
\[ J({\bu},\Omega) = \int_{\Omega} j({\bu}) \du x = - \frac{\gamma}{2}\int_{\Omega} |\nabla\times{\bu}|^2\du x.\]
Since $J$ is quadratic in nature, its Fr{\'e}chet derivative at ${\bu}\in H^1(\Omega)^2$ in the direction $\delta{\bu}\in H^1(\Omega)^2$ can be computed as
\[ \langle J'({\bu},\Omega),  \delta{\bu}\rangle_{(H^1(\Omega)^2)'\times H^1(\Omega)^2} = (j'({\bu}), \delta{\bu})_{\Omega} = -\gamma(\nabla\times{\bu}, \nabla\times\delta{\bu})_{\Omega}.
 \]
Furthermore, we note that 
\[\int_{\Omega} \dive(j(\bu){\bth})\du x = (j'({\bu}),\nabla{\bu}^{\top}{\bth})_{\Omega} + (j({\bu}), \dive{\bth})_{\Omega}. \]

From the two previous identities, we rewrite the second line in \eqref{dG1} as follows:
\begin{align*}
-\gamma(\nabla\times{\bu}, \nabla\times((\nabla{\bu})^\top{\bth}))_\Omega + (j({\bu}),\dive{\bth})_{\Omega}&= (j'({\bu}), \nabla{\bu}^{\top}{\bth})_{\Omega} + (j({\bu}),\dive{\bth})_{\Omega} \\
& = \int_{\Omega} \dive(j({\bu}){\bth})\du x.
\end{align*}
Therefore,  from the assumed regularity of the domain, and by employing divergence theorem, we obtain
	\begin{align*}
		d\mathcal{G}(({\bu},p),\Omega){\bth}=\int_{\Gamma_{\rm f}}\Big[ \alpha\kappa - \frac{\gamma}{2}|\nabla\times{\bu}|^2 + \frac{\partial{\bu}}{\partial{\bn}}\Big(\nu\frac{\partial{\bv}}{\partial{\bn}} + \gamma(\nabla\times{\bu})\btau \Big)\Big]{\bth}\cdot{\bn} \du s.
	\end{align*}

\end{proof}

\begin{rmrk}
(i) The shape derivative of $\mathcal{G}$ as we have formulated it agrees with the Zolesio-Hadamard {\it Structure Theorem} \cite[Corollary 9.3.1]{zolesio2011}, in it we were able to write the derivative in the form
\[ d\mathcal{G}((\bu,p),\Omega){\bth} = \int_{\Gamma_{\rm f}} \nabla{G}({\bth}\cdot{\bn}) \du s.\]
In this case, we shall call $\nabla G$ the shape gradient of the objective functional. Furthermore, this form gives us an intuitive gradient descent direction given by ${\bth} = -\nabla G$.  This fact -- together with the challenges with this chosen direction -- will further be explored in the subsequent parts of the paper.

{(ii) If one observes the adjoint equation \eqref{strong_adjoint}, we can easily see why \eqref{dirdonothing} will not be easily handled. In fact, if such condition is imposed instead of \eqref{conboundcond}, it would be impossible to write the adjoint equation in its strong form. Furthermore, the weak form would include the term $({\bphi}\cdot{\bn})_{-}$, where ${\bphi}$ is a test function. This expression is quite hard to treat numerically due to its discontinuous nature.}
\label{remark:endof3}
\end{rmrk}

%%%%%%%%%%%%%%%%%%%%%%%%
%                             Numerics							 %
%%%%%%%%%%%%%%%%%%%%%%%%
\section{Numerical Realization}\label{sect5}
	
	We shall discuss the numerical implementation of the shape optimization problem in this section.  We start by discussing the resolution on solving the nonlinearity on the state equations,  then we proceed by introducing gradient descent methods based on the Eulerian derivative of the objective functional.  The said gradient descent methods will include the rectification of the volume preservation issue, as required in the formulation of the problem.  Lastly, we shall show the convergence -- with respect to domain discretization --  of the final shapes to a manufactured solution in terms of the final deformation fields and in terms of the Hausdorff distance.

\subsection{Newton implementation of the Navier--Stokes equations}
	The nonlinearity is one of the challenges not only in the analysis but also in the numerical implementation of the Navier--Stokes equations. In this subsection, we shall discuss how this is resolved by means of a Newton's method.
	
	We begin by reiterating the fact that if $(\tilde{\bu},p)\in X$ is the unique solution of \eqref{strong_nontranslated}, then for any $\mathcal{F}\in X^*$, there exists a unique solution $(\delta\!\bu,\delta p)\in X$ to 
\[ \langle \mathbb E'(\tilde{\bu},p)(\delta{\bu},\delta p),({\bphi},\psi)\rangle_{X^*\times X} = \langle \mathcal{F},({\bphi},\psi)\rangle_{X^*\times X},\quad\forall({\bphi},\psi)\in X.\]
This implies that $\mathbb E'(\tilde{\bu},p)\in\mathcal{L}(X,X^*)$ is an isomorphism.
Furthermore, by the inverse function theorem there exists a closed ball $\mathcal{X}((\tilde{\bu},p);\varepsilon)$ centered at $(\tilde{\bu},p)\in X$ with radius $\varepsilon>0$ such that $(\tilde{\bu},p)$ is an isolated nonsingular solution.

With all the facts presented above,  we propose the following Newton's algorithm: we start with an initial element $(\tilde{\bu}^0,p^0)$, and we generate the following sequence $\{(\tilde{\bu}^k,p^k) \}_k$ using the difference equation
\[(\tilde{\bu}^{k+1},p^{k+1}) = (\tilde{\bu}^k,p^k) - [\mathbb E'(\tilde{\bu}^k,p^k)]^{-1}\cdot\mathbb E(\tilde{\bu}^k,p^k) \]
or equivalently, denoting $(\delta\tilde{\bu}^{k+1},\delta p^{k+1}) = (\tilde{\bu}^{k+1}-\tilde{\bu}^{k},p^{k+1}-p^{k})$,
\begin{align}
	\langle\mathbb E'(\tilde{\bu}^k,p^k)(\delta\tilde{\bu}^{k+1},\delta p^{k+1}),({\bphi},\psi)\rangle_{X^*\times X} = - \langle\mathbb E(\tilde{\bu}^k,p^k),({\bphi},\psi)\rangle_{X^*\times X},
	\label{weak:newton}
\end{align}
	for all $({\bphi},\psi)\in X$.  In strong form, by letting ${\bu}^k = \tilde{\bu}^k + {\bg}$,  we can write \eqref{weak:newton} as 
	\begin{align}
		\left\{
		\begin{aligned}
			-\nu\Delta{\bu}^{k+1} + ({\bu}^{k+1}\cdot\nabla){\bu}^k + ({\bu}^k\cdot\nabla){\bu}^{k+1} + \nabla p^{k+1}  & = {\bu}^k\cdot\nabla{\bu}^k + {\blf} && \text{in }\Omega,\\
			\dive{\bu}^{k+1} &= 0&& \text{in }\Omega,\\ 
			{\bu}^{k+1} &= {\bg} && \text{on }\Gamma_{\rm in},\\
			{\bu}^{k+1} &= 0&& \text{on }\Gamma_{\rm f}\cup\Gamma_{\rm w},\\
			-p^{k+1}{\bn} +\nu\partial_{\bn}{\bu}^{k+1} - \frac{1}{2}[({\bu}^{k+1}\cdot{\bn}){\bu}^k + ({\bu}^{k}\cdot{\bn}){\bu}^{k+1}]   &= - \frac{1}{2}({\bu}^{k}\cdot{\bn}){\bu}^k && \text{on }\Gamma_{\rm out}.
			\end{aligned}
		\right.
		\label{newtonaprrox}
	\end{align}
	
	For a given $\blf\in L^2(\Omega)^2$, and ${\bg}\in H^2(\Omega)^2$, one can show -- by following the arguments as in \cite{girault1986} -- that if ${\bu}\in{\bV}\cap H^2(\Omega)^2 $ is the solution to \eqref{strong_nontranslated} there exists $c<1$ such that the following convergence estimate holds
	\begin{align} 
	\|{\bu}^{k+1} - {\bu} \|_{\bV(\Omega)} \le c\|{\bu}^{k} - {\bu} \|_{\bV(\Omega)}.
	\label{approx:newton}	
	\end{align}
	
	Solving numerically, we shall approximate the solution to \eqref{strong_nontranslated} using \eqref{newtonaprrox}, with the stopping criterion $\|{\bu}^{k+1} - {\bu}^k \|_{\bV(\Omega)} < \varepsilon$ for sufficiently small $\varepsilon>0$.
	
	It is noteworthy to mention that Newton's method for the Navier--Stokes equations with \eqref{conboundcond} is naturally constructed thanks to the absence of $(\cdot)_{-}$, on the other hand the linearization will not be easily handled if \eqref{dirdonothing} is used.
	
\subsection{Gradient descent methods for deformation fields}
As mentioned before, in this part we discuss gradient methods based on the Eulerian derivative of the objective functional. Furthermore, due to the volume constraint we shall utilize two methods, namely the use of an augmented Lagrangian method based on \cite{nocedal2006}, and of divergence-free deformation fields.
\subsubsection{Augmented Lagrangian method}

Note that the optimization problem can be written as the equality constrained optimization by
\[ \min_{\Omega\in\Oad} \mathcal{G}(\Omega)\text{ subject to }\mathcal{F}(\Omega) := |\Omega| -m = 0. \]
With this reason, we formulate the augmented Lagrangian given as
\[ \mathcal{L}(\Omega,\ell,b) = \mathcal{G}(\Omega) - \ell\mathcal{F}(\Omega) + \frac{b}{2}\mathcal{F}(\Omega)^2, \]
	where $\ell>0$ is a Lagrangian multiplier, and $b>0$ is a regularizing parameter.  We note that the quadratic term -- aside from its regularizing effect -- acts as a more strict penalizing term as compared to the usual Lagrangian methods. This method was formalized in the context of shape optimization by Dapogny, C. et. al.\cite{dapogny2018}.  So to minimize the objective functional while not neglecting the constraint $\mathcal{F} = 0$, we instead minimize the augmented Lagrangian $\mathcal{L}$.
	
	By following the same arguments as for solving the Eulerian derivative of $\mathcal{G}$, we can formulate the derivative of $\mathcal{L}$ and is solved as
	\[ \mathcal{L}(\Omega,\ell,b){\bth} = \int_{\Gamma_{\rm f}} \nabla L({\bth}\cdot{\bn})\du s, \]
	where $\nabla L = \nabla G - \ell + b(|\Omega| - m).$
	
\subsubsection{Smooth extensions of the deformation fields and the unit normal vector}

As mentioned in Remark 5, an intuitive gradient descent direction on the boundary $\Gamma_{\rm f}$ is either ${\bth} = -\nabla G$, or ${\bth}=-\nabla L$ if one chooses to minimize the augmented Lagrangian. However, this choice of ${\bth}$ may cause irregularities to the deformed boundary $\Gamma_{\rm f}$. For this reason, methods of approximating the deformation field ${\bth}$ and extending it to the domain $\Omega$ have been proposed, for example in \cite{azegami2006} the authors developed a seminal approach on such smoothing method. In our current problem, we shall adapt such smoothing extension for minimizing the augmented Lagrangian. In particular, we shall be tasked to solve for ${\bth}\in H^1_{\bTh}(\Omega)^2 := \{{\bphi}\in H^1(\Omega)^2; {\bphi}=0\text{ on }\partial\Omega\backslash\Gamma_{\rm f} \}$ that solves the Robin equation
\begin{align}
	\varepsilon(\nabla{\bth},\nabla{\bphi})_{\Omega} + ({\bth},{\bphi})_{\Gamma_{\rm f}} = -(\nabla L{\bn},{\bphi})_{\Gamma_{\rm f}},\quad \forall{\bphi}\in H^1_{\bTh}(\Omega)^2.
	\label{defo:lagrange}
\end{align}
Meanwhile, we shall also utilize what was proposed in \cite{simon2021} by minimizing the objective functional itself but uses divergence-free deformation fields, i.e., we shall solve for $({\bth},\vartheta)\in H^1_{\bTh}(\Omega)^2\times L^2(\Omega)^2$ that solves the following Stokes equation
\begin{align}
	\left\{
		\begin{aligned}
			\varepsilon_1(\nabla{\bth},\nabla{\bphi})_{\Omega} + ({\bth},{\bphi})_{\Gamma_{\rm f}} - \varepsilon_2(\vartheta,\dive{\bphi})_\Omega & = -(\nabla G{\bn},{\bphi})_{\Gamma_{\rm f}}\\
			(\psi,\dive{\bth})_\Omega & = 0
		\end{aligned}
	\right.
	\label{defo:divfree}
\end{align}
for all $({\bphi},\psi)\in H^1_{\bTh}(\Omega)^2\times L^2(\Omega)^2$. We note that on both methods $\varepsilon,\varepsilon_1,\varepsilon_2>0$ are chosen sufficiently small, so that $-\varepsilon\partial_{\bn}{\bth}+{\bth} = -\nabla L$ and $-\varepsilon_1\partial_{\bn}{\bth}+{\bth} + \varepsilon_2\vartheta{\bn} = -\nabla G$ on $\Gamma_{\rm f}$, and hence are respective approximations for $\bth = -\nabla L$ and $\bth = -\nabla G$ on $\Gamma_{\rm f}$.

Using the same narrative as above, we shall determine a smooth extension of the unit normal vector $\bn$ on $\Gamma_{\rm f}$, by finding the vector $\bN$ that solves 
\begin{align}
	\varepsilon_3(\nabla{\bN},\nabla{\bphi})_\Omega + ({\bN},{\bphi})_{\Gamma_{\rm f}}= ({\bn},{\bphi})_{\Gamma_{\rm f}}, \quad \forall{\bphi}\in H_{\bTh}^1(\Omega)^2.
	\label{normal:ext}
\end{align}
	The impetus for solving for such extension is the appearance of the mean curvature $\kappa$ in the shape gradients.  According to \cite[Proposition 5.4.8]{henrot2018}, the mean curvature of the surface $\Gamma_{\rm f}$ may be determined using the identity $\kappa = \dive_{\Gamma_{\rm f}}{\bn}= \dive{\bN}$ for any extension ${\bN}$, given that the surface is sufficiently smooth. Since our domain is $C^{1,1}$-regular, we have the freedom to apply such identity.
	
\subsubsection{Gradient descent iterations}

Having been able to lay out the necessary ingredients, we finally write here the iterative scheme of approximating the shape solution.  
Recall that we introduced the domain variations that utilizes the identity perturbation operator. One of the reason why we take advantage of such perturbation is its usability and simplicity for iterative numerical methods.  In particular,  for a given initial domain $\Omega_0$ we shall generate a sequence $\{\Omega_k\}_k$ by means of the difference equation $\Omega_{k+1} = \Omega_k + t^k{\bth}^k.$ In this iteration, we note that the only changing boundary is the free-boundary $\Gamma_{\rm f}$, so by denoting the k$^{th}$ iteration of the mentioned boundary by $\Gamma_{\rm f}^k$ and the k$^{th}$ iteration of $\nabla G$ (or $\nabla L$) by $\nabla G^k$ (resp. $\nabla L^k$), i.e., the states and the mean curvature are all evaluated in $\Omega_k$ instead of $\Omega$, and the deformation fields ${\bth}^k$ are determined by solving either \eqref{defo:lagrange} or \eqref{defo:divfree}, but with $\Omega_k$, $\Gamma_{\rm f}^k$ and $\nabla G^k$ (resp. $\nabla L^k$) instead of $\Omega$, $\Gamma_{\rm f}$, and $\nabla G$ (resp. $\nabla L$), respectively. 

The time step on the other hand, just like most gradient descent methods, is determined using a line search method, i.e., for fixed parameters $0<\sigma_1,\sigma_2 < 1$, $j\in\{0,1,2,\ldots \}$, and by denoting 
\begin{align}
\hat{t}_j^k = (\sigma_2)^j\frac{\sigma_1\mathcal{G}(\Omega_k)}{\|{\bth}^k\|^2_{L^2(\Gamma_{\rm f})^2}},\label{timestep}
\end{align}
(or $\mathcal{L}(\Omega_k)$ for the case of the augmented Lagrangian ), we take  $t^k = \hat{t}_j^k$, where $j$ is the smallest integer such that $\mathcal{G}(\Omega_k + \hat{t}_j^k{\bth}^k) < \mathcal{G}(\Omega_k )$ ( resp. $\mathcal{L}(\Omega_k + \hat{t}_j^k{\bth}^k) < \mathcal{L}(\Omega_k )$ ), and such that the perturbed domain $\Omega_k + \hat{t}_j^k{\bth}^k$ does not exhibit mesh degeneracy.
	 
	Lastly, the parameters $\ell,b>0$ in the augmented Lagrangian will also be updated iteratively based on \cite[Framework 17.3]{nocedal2006}. Given initial parameters $\ell_0,b_0>0$, we generate the iterates $\ell_k,b_k$ using the following rules
\begin{align}
		\ell_{k+1} = \ell_k - b_k\mathcal{F}(\Omega_k),\quad b_{k+1} = \begin{cases}\tau b_k & \text{if }b_k<\overline{b}\\ b_k &\text{otherwise }\end{cases},
		\label{auglangparams}
\end{align}	 
where $\tau>1$ and $\overline{b}>0$ are given parameters.

Summarizing these methods, the algorithm\footnote{When we refer to {\it steps} we mean them to be inside a \texttt{for} loop.} below provides the steps for solving an approximate shape solution:
\begin{pethau}
	\item[{\bf Initialization:}] Choose the parameters $\nu,\alpha, \gamma,\varepsilon,\varepsilon_1,\varepsilon_2,\varepsilon_3,\sigma_1,\sigma_2,\tau$ and $\overline{b}$; initialize the domain $\Omega_0$, the solution ${\bu}_0$ (using Newton's method \eqref{newtonaprrox}), and the parameters $\ell_0$ and $b_0$ (only if the augmented Lagrangian is being implemented).
	\item[{\bf Step 1:}] Evaluate $\mathcal{G}(\Omega_k)$ (or $\mathcal{L}(\Omega_k)$); solve for the adjoint variable ${\bv}_k$ from \eqref{strong_adjoint}, the mean curvature $\kappa_k$ from \eqref{normal:ext}, and the deformation field ${\bth}^k$ using \eqref{defo:divfree}, or \eqref{defo:lagrange} in the case of the augmented Lagrangian method.
	\item[{\bf Step 2:}] Set $\Omega_{k+1} = \Omega_k + t^k{\bth}^k$, where $t^k$ is chosen appropriately as discussed above.
	\item[{\bf Step 3:}] Solve the new solution ${\bu}_{k+1}$ and evaluate $\mathcal{G}(\Omega_{k+1})$ (or $\mathcal{L}(\Omega_{k+1})$); if $|\mathcal{G}(\Omega_{k+1})-\mathcal{G}(\Omega_{k})| < \texttt{tol}$ (resp. $|\mathcal{L}(\Omega_{k+1})-\mathcal{L}(\Omega_{k})| < \texttt{tol}$) then $\Omega_{k+1}$ is accepted as the approximate solution.
	\item[{\bf Step 4:}] (Only for augmented Lagrangian method) Update the parameters $\ell_{k+1},b_{k+1}$ using \eqref{auglangparams}. 
\end{pethau}

	\subsection{Numerical Examples}
	This part shows examples of implementations of the algorithm previously discussed. We start with implementing the augmented Lagrangian method, and show its volume preserving limitations by simulating different values of the Lagrangian multiplier $\ell_0$. We then simulate examples using the divergence-free deformation fields, and compare the solutions with the augmented Lagrangian method based on the convergence rate (number of iterations) and the volume preserving property. Lastly, we end with showing convergence of shape solutions to a manufactured solution based on the Hausdorff measure on the boundaries and on the H$^1$-convergence of the deformation fields.
	
	The simulations were all ran using FreeFem++ \cite{hecht2012} on a  Intel Core i7 CPU $@$ 3.80 GHz with 64GB RAM.  The state, adjoint, and the resolution of the deformation fields (for the divergence-free deformation fields) are solved using triangular Taylor-Hood (P2-P1) finite elements, while the resolution of the mean curvature and of the deformation fields using the augmented Lagrangian method are respectively solved using P1 and P2 finite elements, all of which are solved with the UMFPACK solver. The input function ${\bg}$ is a Poiseuille-like function given by $\bg =(1.2(0.25 - x_2^2),0)^\top$ on $\Gamma_{\rm in}$,  for simplicity the external force is taken as ${\blf} = 0$, the viscosity parameter is chosen as $\nu = 1/100$, and the Tikhonov parameter is chosen as $\alpha = 7$ to ensure that $\mathcal{G}(\Omega^*)\approx 0$ at the solution domain $\Omega^*$.
	As for the domain, we consider a rectangular outer boundary and an circular  initial free(inner)-boundary, as shown in Figure \ref{initialshape}.
	\begin{figure}[h!]
 \centering
  \includegraphics[width=.8\textwidth]{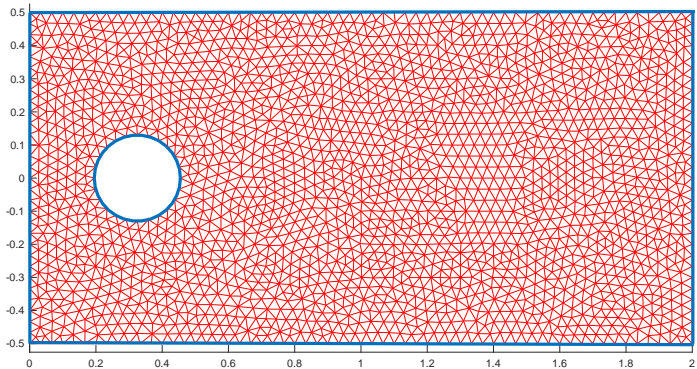}\vspace{-.1in}
 \caption{{The figure shows the initial set-up of the domain. The rectangular outer boundary has vertices located at $(0,1/2)$, $(0,-1/2)$, $(2,1/2)$, and $(2,-1/2)$, while the inner boundary is a circle centered at $(0.325,0)$ with radius $r = 0.13$.} }
 \label{initialshape}
 \end{figure}
 
 	Meanwhile, the domain variation based on the deformation fields will be dealt with by using the \texttt{movemesh} command in FreeFem++, and the possible mesh degeneracy will be circumvented (aside from the choice of the step size) by utilizing the combination of the \texttt{checkmovemesh} and \texttt{adaptmeshmesh} commands.
 	
 	Lastly, we mention that for the first two subparts of this subsection the mesh size will be taken uniformly and has size $h = 1/60$, i.e., the diameter of all the triangles in the domain triangulation will be taken as $1/60$. Furthermore, the tolerance for the stopping criterion of shape approximation is decided to be $\texttt{tol}=10^{-4}$.

	\subsubsection{Augmented Lagrangian method}
	
 We start the implementation of the augmented Lagrangian with $\ell_0 = 54$ since this value is relatively effective in our goal of preserving the volume of the domain as compared to other values of $\ell_0$ that we simulated. We note also that in some figures we shall utilize scaling of the computed values, as such scaling is done in a way that the maximum value is at $y = 1$, and the minimimum value is at $y =0$, i.e.,for an iteration $k$ the value in the trend is computed as $\frac{F(\Omega_k) - min_{j}(F(\Omega_j))}{max_{j}(F(\Omega_j))-min_{j}(F(\Omega_j))}$, where $F$ is either the objective function or the volume.
 
\begin{figure}[h!]
 \centering
  \includegraphics[width=.8\textwidth]{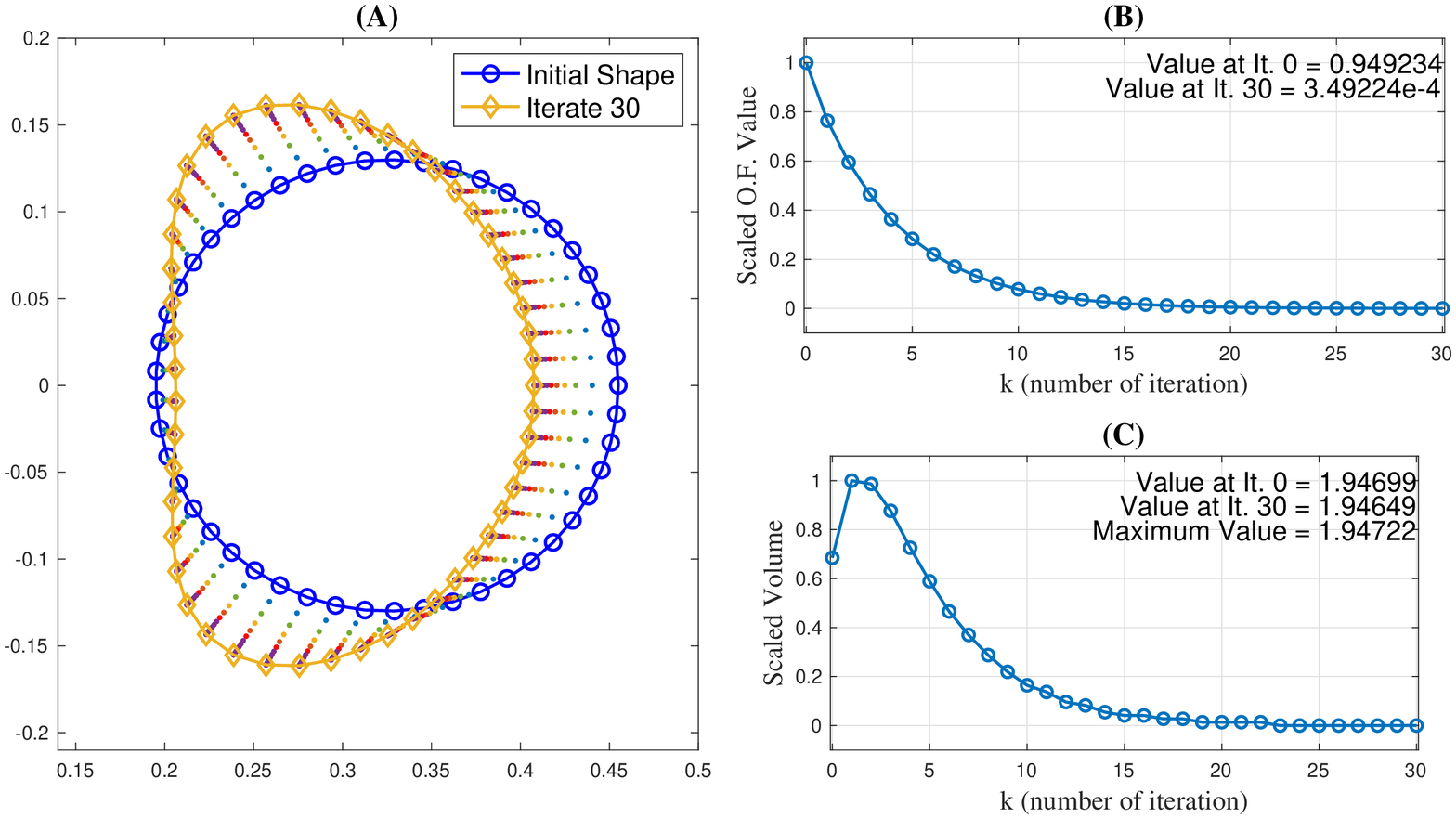}\vspace{-.1in}
 \caption{Evolution of the free-boundary using the augmented Lagrangian method with $\ell_0=54$(A); and the scaled trends of the objective function (B) and of the volume of the domain (C), where the $x$-axis corresponds to the number of iteration.}
 \label{evo:auglang}
 \end{figure}
	
	As shown in Figure \ref{evo:auglang}(A), the circular initial shape evolves into a  bean-shaped surface but exhibits more convexity than the usual bean shape. The method converges after thirty iterations, where -- as seen in Figure \ref{evo:auglang}(B) -- the value of the objective functional starts with the value $\mathcal{G}(\Omega_0)= 0.949$ and converges to the approximate solution with value $\mathcal{G}(\Omega_{30})= 3.49\times10^{-4}$. With regards to the volume preservation, we observe a sudden increase on the first iteration. Nevertheless, the sudden increase on the first iteration is corrected starting from the second iteration, this is caused by the manner by which we update the Lagrange multiplier $\ell$ and the regularizing parameter $b$. However, the decrease continues until the volume, which started with the value $|\Omega_0| = 1.94699$, is decreased up to the volume of $|\Omega_{30}|= 1.94649$. This decrease on the volume would still be considered a fair volume preservation, with the relative percentage difference\footnote{The relative percentage difference is computed with respect to the initial volume, i.e., $\frac{||\Omega_0|- |\Omega_{\rm final}||}{|\Omega_0|}\times 100\%$.} equal to $2.568\times10^{-2}\%.$ 
		
		\begin{figure}[h!]
 \centering
  \includegraphics[width=.8\textwidth]{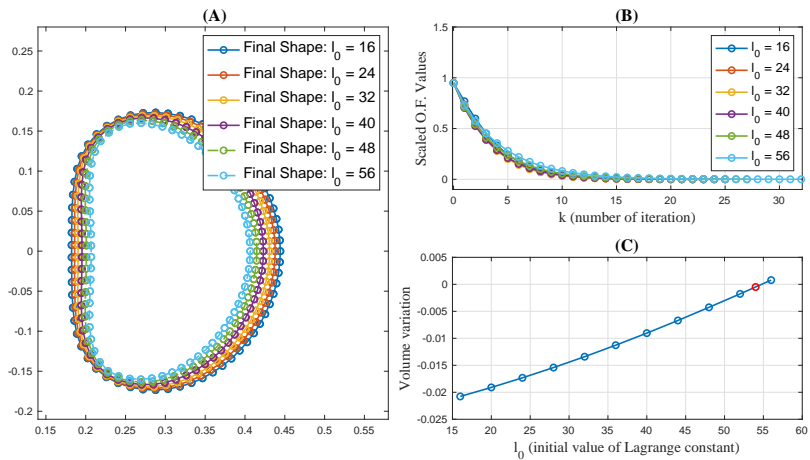}\vspace{-.1in}
 \caption{Final shapes of the boundary $\Gamma_{\rm f}$ for different values of $\ell_0$(A); trends of the objective functional values on each iteration for different values of $\ell_0$(B); variation of the volume from the initial domain to the final shape plotted against the value of $\ell_0$ (C).}
 \label{varl0}
 \end{figure}

	For the next illustration, we show the effects of varying the initial value of the Lagrange multiplier $\ell_0$.  From Figure \ref{varl0}(A), the domain bounded by the surface $\Gamma_{\rm f}$ gets larger as the value of $\ell_0$ decreases. This observation implies that the stringency of the equality constraint $\mathcal{F}(\Omega) = 0$ is stronger as the value of the Lagrange multiplier $\ell_0$ increases, but the satisfaction of the said constraint is only up to a maximum value. This can be observed in Figure \ref{varl0}(C) where, starting from $\ell_0=16$ with an increment of four, the variation\footnote{The variation is computed as $|\Omega_0|-|\Omega_{\rm final}|$.} increases from around -0.02 to the value zero. However, at $\ell_0=56$ the variation becomes positive and thus the constraint becomes too constricting which -- when higher values of $\ell_0$  are simulated -- causes the domain inside the surface $\Gamma_{\rm f}$ to become smaller, and fails the volume constraint. We also plotted with a differently colored hollow point the variation from the initial volume to the volume of the final shape for $\ell_0=54$ which was the parameter value chosen in our first illustration. From the trend of the variational line we can make a conjecture that, given all other parameters are taken constant, the best value of the Lagrange multiplier that will satisfy $\mathcal{F} = 0$ is around $\ell_0=55$.  
	We also observe in Figure \ref{varl0}(B) that even though all objective function trends converge to zero, the decrease for $\ell_0=56$ is the least steep. 
	This is because when the value of $\ell_0$ is smaller, the shape is more allowed to change and thus the algorithm is more relaxed to immediately decrease the value of the objective functional.

	\subsubsection{Divergence-free method}
	
	For the divergence-free method,  we get the freedom from tediously choosing an appropriate parameter $\ell_0$ that will give a volume preserving deformation fields.  
	\begin{figure}[h!]
 \centering
  \includegraphics[width=.8\textwidth]{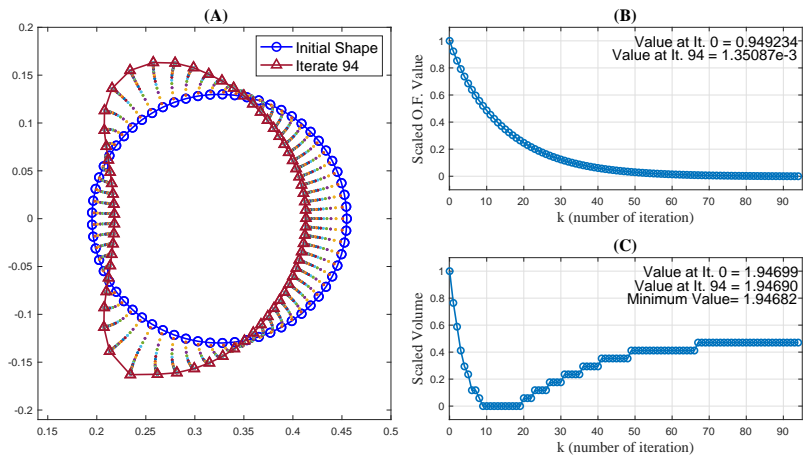}\vspace{-.1in}
 \caption{Evolution of the free-boundary using the divergence-free method (A); and the scaled trends of the objective function (B) and of the volume of the domain (C), where the $x$-axis corresponds to the number of iteration.}
 \label{evo:divfree}
 \end{figure}
 
 We can also observe a bean-shaped surface develop as we go further in each iteration (see Figure \ref{evo:divfree}(A)).  Contrary to the shape generated by the augmented Lagrangian method, the surface obtained in this current method bounds a domain whose convexity on the left side is lost. Nevertheless, as we can see from Figure \ref{evo:divfree}(B), the objective functional tends to zero as the method reaches its tolerance. Although, a drawback in this method is the apparent {\it slow} convergence, as the method converges only after 94 iterations. Fortunately, the volume preservation in this current method is better -- in most cases -- as compared to the augmented Lagrangian, where the relative percentage change in the volume is computed as $4.623\times 10^{-3}\%$. The only case in the augmented Lagrangian method that can match this volume preservation is when $\ell_0=55$, with relative percentage change $7.19\times10^{-3}$ (see Figure \ref{evo:auglang55}).
	\begin{figure}[h!]
 \centering
  \includegraphics[width=.8\textwidth]{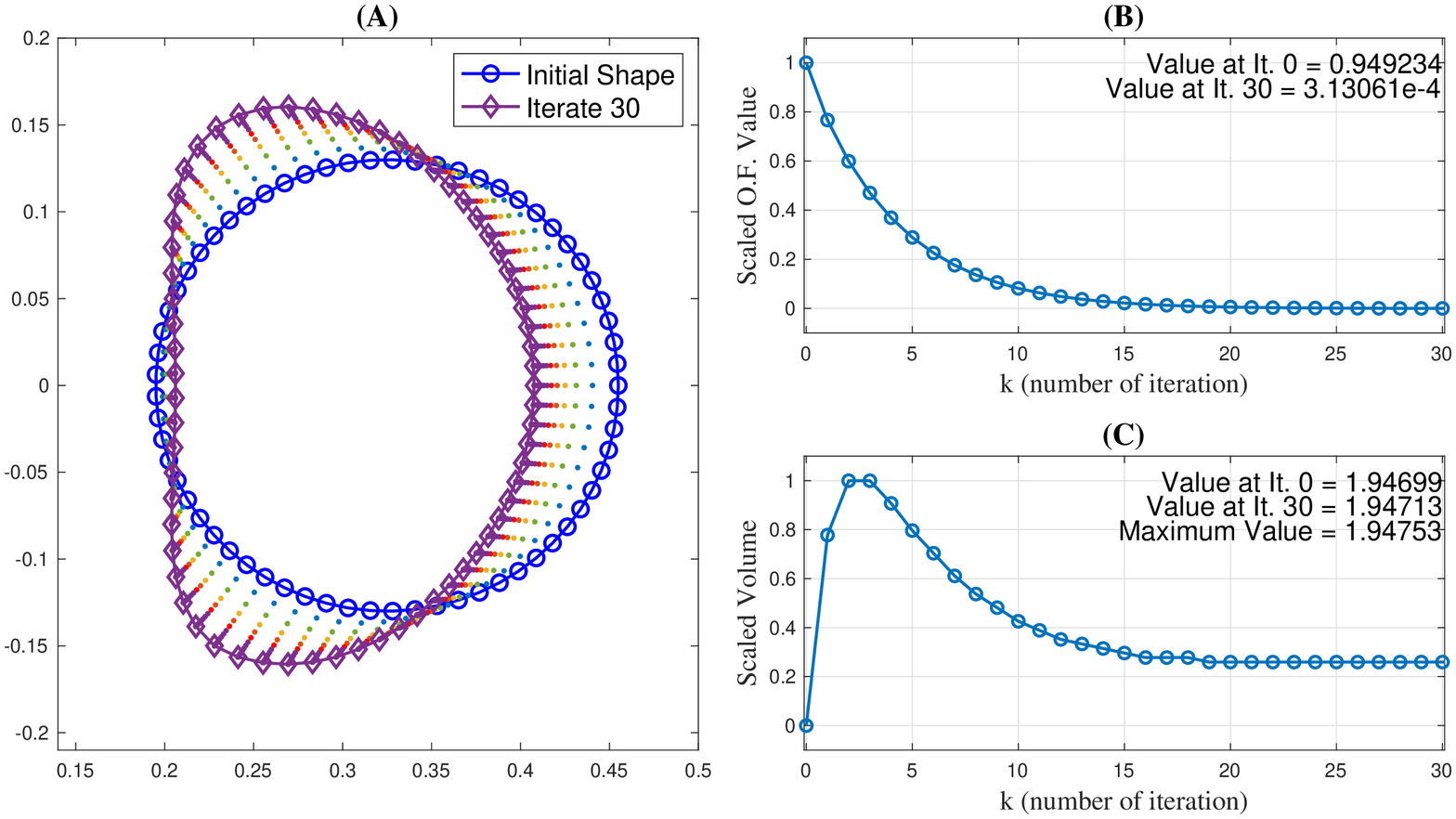}\vspace{-.1in}
 \caption{Evolution of the free-boundary using the aigmented Lagrangian method with $\ell_0=55$ (A); and the scaled trends of the objective function (B) and of the volume of the domain (C), where the $x$-axis corresponds to the number of iteration.}
 \label{evo:auglang55}
 \end{figure}
	 
	 To see even further the difference between the augmented Lagrangian and the divergence-free method with regards to the final shapes, objective value and volume trends, we refer to Figure \ref{comp:augvsdf}.
\begin{figure}[h!]
 \centering
  \includegraphics[width=.8\textwidth]{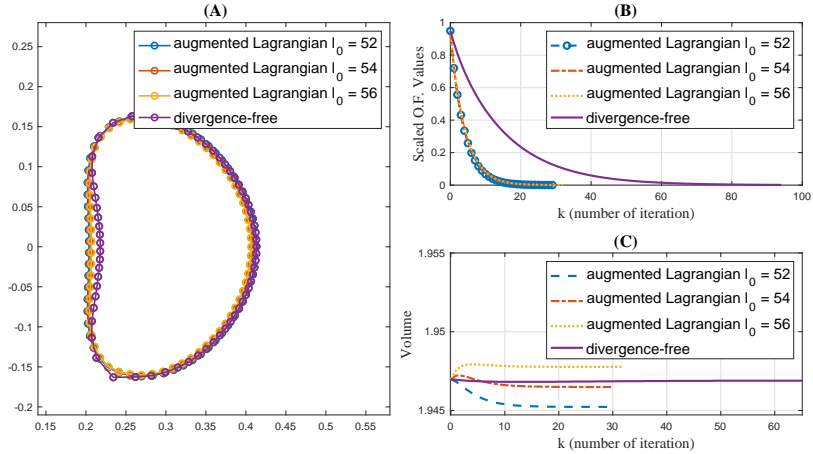}\vspace{-.1in}
 \caption{Comparison of the final shapes (A), the trends of the objective function (B) and of the volume of the domain (C),  between the divergence-free method and the augmented Lagrangian method.}
 \label{comp:augvsdf}
\end{figure}

	Before we move on, let us first verify if we are fulfilling one of the main goals of the problem which is maximizing the vorticity of the fluid. Even though we are observing a decrease in the objective functional in Figures \ref{evo:auglang}(B) and \ref{evo:divfree}(B), it is important that we confirm that the descent does not only affect the perimeter of the free-boundary but also minimizes the negative of the $L^2$-norm of the curl of the velocity field. 
\begin{figure}[h!]
 \centering
  \includegraphics[width=.5\textwidth]{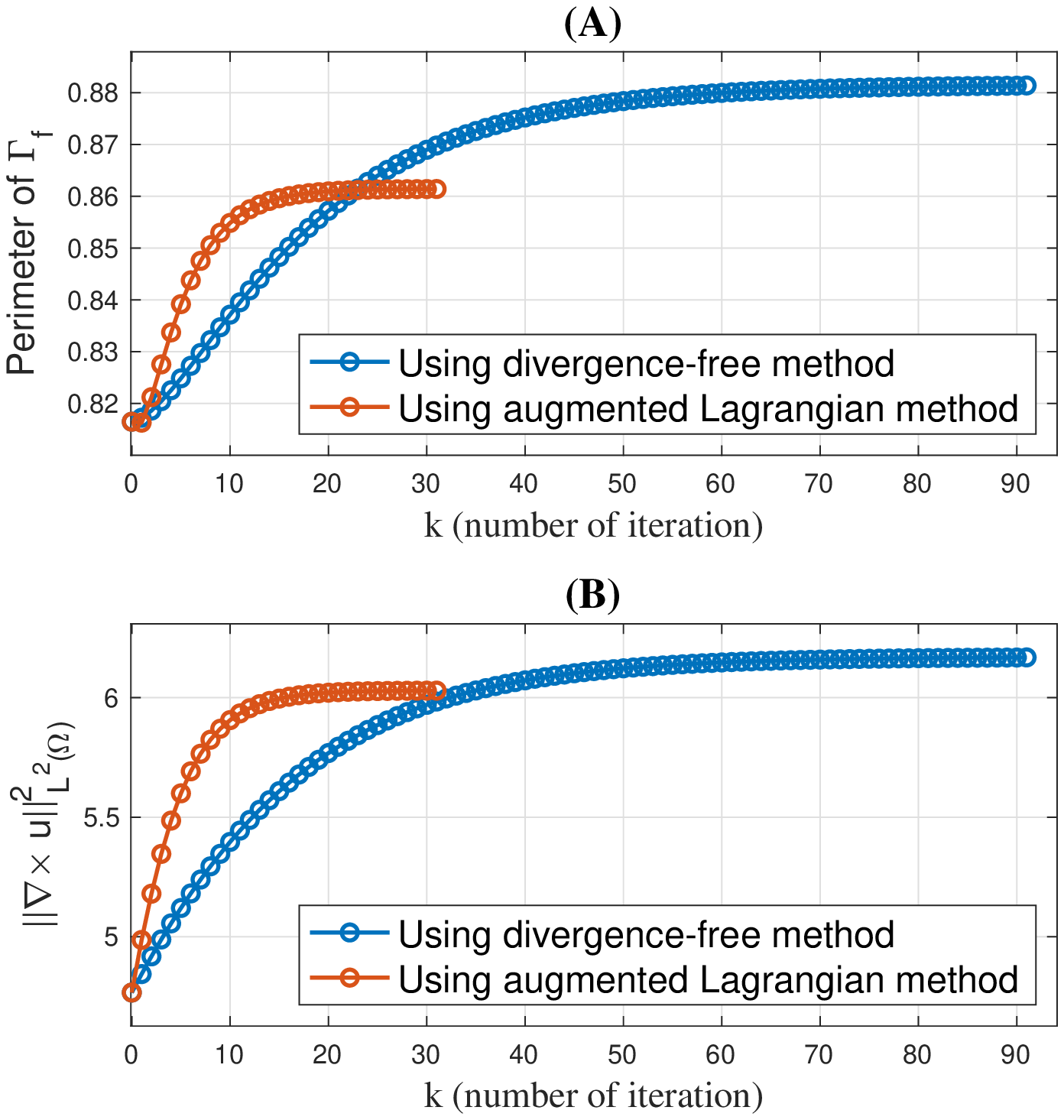}\vspace{-.1in}
 \caption{Trends of the perimeter of the surface $\Gamma_{\rm f}$ (A), and of the vorticity functional $\|\nabla\times{\bu}\|_{L^2(\Omega)}^2$ (B).}
 \label{perimeterundcurl}
 \end{figure}
	
	{As can be observed in Figure \ref{perimeterundcurl}, the perimeter functional is increasing in each iteration, while the quantity $\|\nabla\times{\bu}\|_{L^2(\Omega)}^2$ is also increasing.  This verifies and satisfies our goal of maximizing the vorticity of the fluid. Nevertheless, it would be foolish to assume that it is okay to neglect this part of the objective functional, or even remove it in the definition of the objective functional. As a matter of fact, we recall that this portion of the functional helped us in regularizing the domain variation, and the negligence on this part of the objective functional may cause topological changes in the domain, i.e., additional holes may emerge due to lack of the regularization in the perimeter. Furthermore, the choice of $\alpha>0$ in the regularization helps to make sure that $\mathcal{G}(\Omega_k) \ge 0$ which implies that our choice of step size given in \eqref{timestep} is always nonnegative.}
	
	{To visually observe the effect of the shape solutions on the fluid vortex, we simulate a dynamic version\footnote{Here we used $ \partial_t{\bu}- \nu\Delta{\bu} +({\bu}\cdot\nabla){\bu} +  \nabla p = {\blf}$ instead of the first equation in \eqref{strong_nontranslated} with same boundary conditions and divergence-free assumption, and with the initial data ${\bu}(0) = (0,0)$ in $\Omega$.}{}$^{,}${}\footnote{We note that generation of vortices using the stationary Navier--Stokes equations is almost negligible, hence we use a time-dependent system for observation which was solved using a Lagrange--Galerkin scheme (see \cite{suli1988} for such scheme, and \cite{notsu2016} for a stabilized version).} of the Navier--Stokes equations and observe the length and width of the twin-vortex right before Karman vortex shedding. }
	
	\begin{figure}[h!]
 \centering
  \includegraphics[width=.8\textwidth]{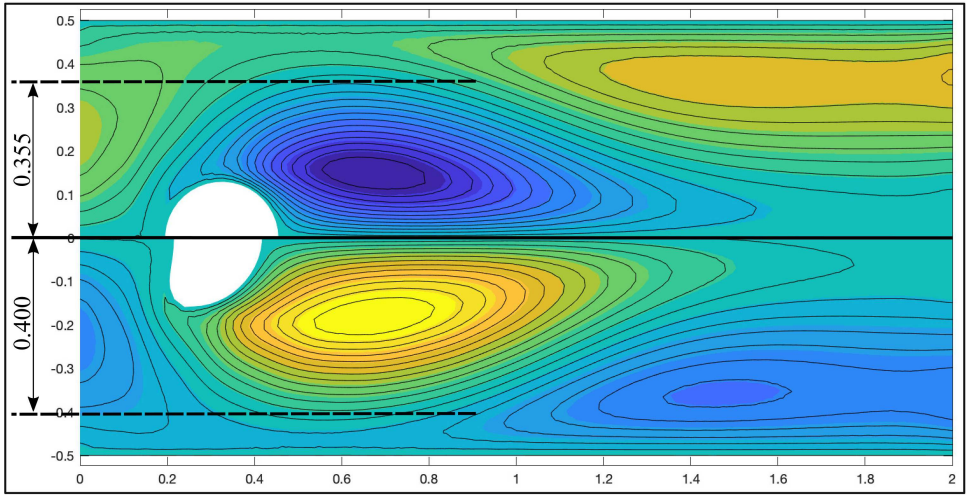}\vspace{-.1in}
 \caption{The figure shows the generation of twin-vortex before the shedding of Karman vortex. The upper part of the figure corresponds to the upper half of the twin-vortex using the initial shape, while the lower part is done using the final shape from the divergence-free method.}
 \label{twinvortex}
 \end{figure}

	{Figure \ref{twinvortex} shows the comparison of the twin-vortex induced by the initial geometry and the final shape using the divergence-free method. Note that the simulations were done using similar contour levels for the comparison to be reliable. We observe that the length and the width of the vortex is improved when the final shape from the optimization problem is used. The improved length and width of the vortex using the final shape from the augmented Lagrangian method behaves similarly as when the divergence-free method is used, hence we skip such illustration.}

\subsubsection{Convergence with respect to domain triangulations}	

For our last set of illustrations, we shall show the convergence of the approximate shape solutions to a manufactured {\it exact} solution according to the domain triangulation. The manufactured exact solution is obtained by using higher order polynomial basis functions and finer domain triangulation. Here, we chose (P4-P3) Taylor-Hood finite elements for the state, adjoint, and the divergence-free deformation fields, while we utilized P3 finite elements for the mean curvature and the augmented Lagrangian deformation fields, we also considered the mesh size $h = 1/160$.

By denoting $\Gamma_{{\rm f},e}$ and $\Gamma_{{\rm f},h}$ the respective free-boundaries of the exact and the approximate domain solutions, we shall first show the convergence in terms of the Hausdorff measure\footnote{Given two sets $A,B\subset \mathbb{R}^2$, the Hausdorff measure $d_H(A,B)$ is defined as $$\displaystyle d_H(A,B) = \max\left\{\sup_{a\in A}{ \inf_{b\in B}|a-b|}, \sup_{b\in B} {\inf_{a\in A}|a-b|} \right\}$$} between $\Gamma_{{\rm f},e}$ and $\Gamma_{{\rm f},h}$. We consider the mesh sizes $h=1/10,1/20,\ldots,1/100$ and utilize the same finite elements considered in the previous illustrations.

Figure \ref{conv:divfree} shows the results using the divergence-free method, which we implemented with the different values of mesh size as discussed above. In Figure \ref{conv:divfree} (A) we see the comparison of the final shapes $\Gamma_{{\rm f},h}$, and the manufactured final shape of the boundary $\Gamma_{{\rm f},e}$.  It can be observed that indeed the shape generated with mesh size $h=1/10$ is much coarse as compared to the other cases, and becomes smoother as $h$ becomes smaller. Furthermore, the approximated boundaries get closer to the exact boundary as $h\to0$.  It can be observed as well that the approximations exhibit a sharp edge on the bottom of the boundary while the exact solution has a smooth boundary all through out. 

\begin{figure}[h!]
 \centering
  \includegraphics[width=.8\textwidth]{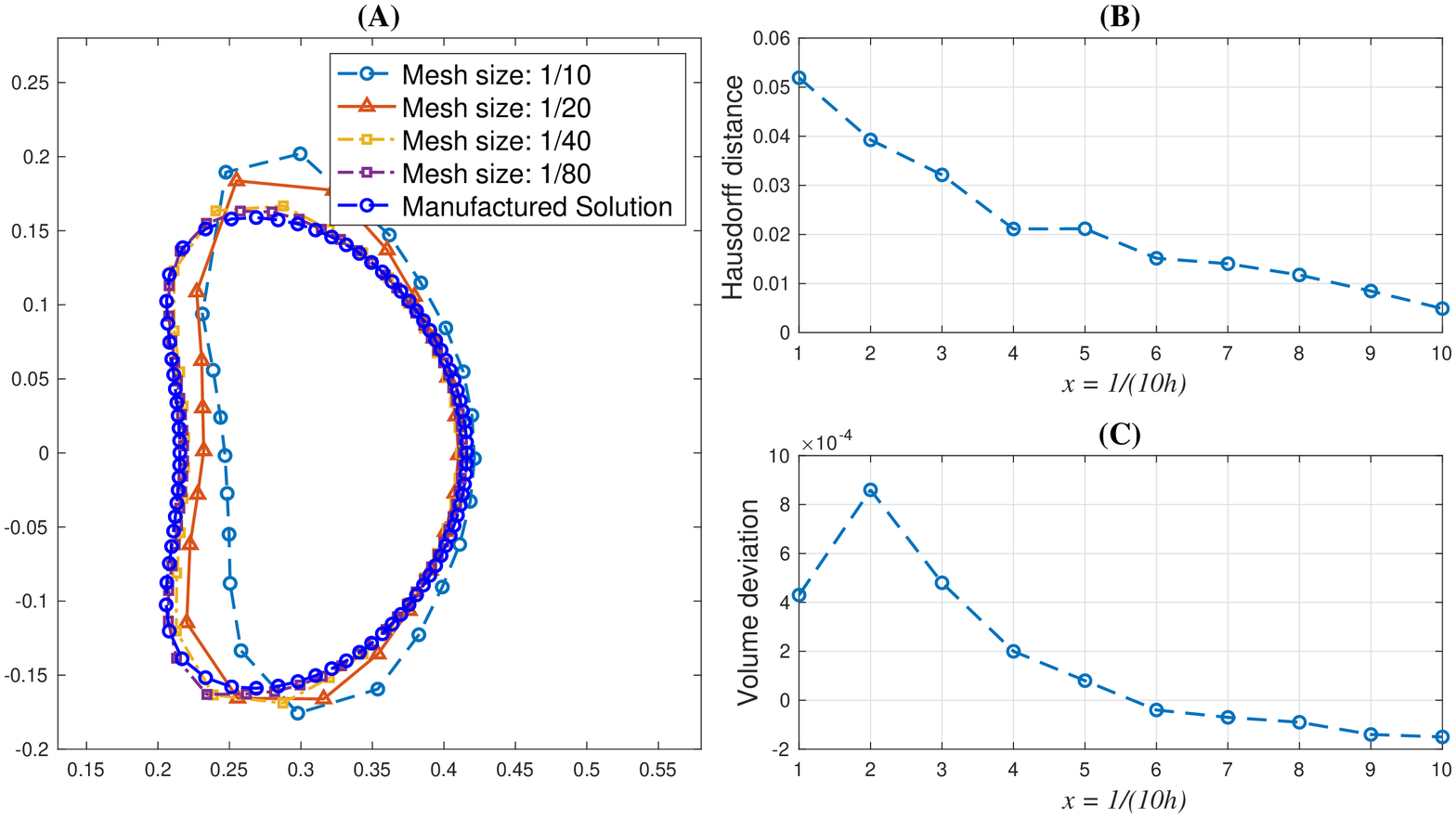}\vspace{-.1in}
 \caption{Using divergence-free method, the figure shows plot the final shapes of the boundary $\Gamma_{\rm f}$ with different mesh sizes(A); plots of $d_H(\Gamma_{{\rm f},e},\Gamma_{{\rm f},h})$ (B) and volume variation (C) versus $x=1/(10h)$, where $h$ is the mesh size}
 \label{conv:divfree}
\end{figure}

Aside from the qualitative observation of the domain convergence, Figure \ref{conv:divfree} (B) shows how the Hausdorff measure between the approximate and exact boundaries becomes smaller as we decrease the mesh size.  As for the preservation of the volume,  we can see that the divergence-free method preserves the domain volume efficiently. In fact, we can see from Figure \ref{conv:divfree} (C) that the domain variations do not exceed $9\times 10^{-4}$. 

As for the augmented Lagrangian method,  Figure \ref{conv:auglang} shows the implementation with different values of $h$ as discussed before, but with a fixed Lagrange multiplier $\ell_0=54$. 

\begin{figure}[h!]
 \centering
  \includegraphics[width=.8\textwidth]{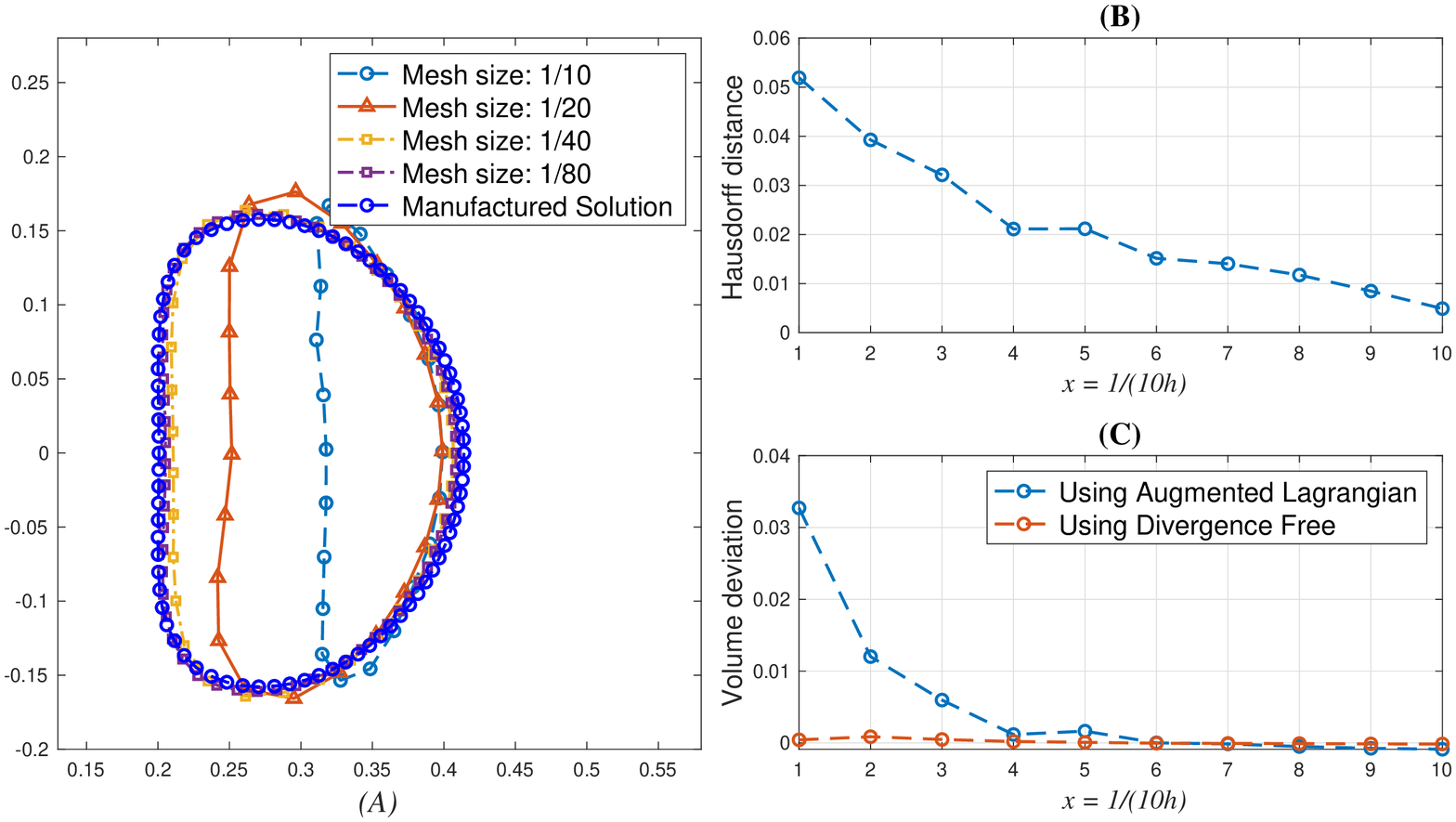}\vspace{-.1in}
 \caption{Using augmented Lagrangian method, the figure shows plot the final shapes of the boundary $\Gamma_{\rm f}$ with different mesh sizes(A); plots of $d_H(\Gamma_{{\rm f},e},\Gamma_{{\rm f},h})$ (B) and volume variation (C) versus $x=1/(10h)$, where $h$ is the mesh size}
 \label{conv:auglang}
\end{figure}

As we can observe, the approximate solutions converge to the manufactured exact solution. Starting with a coarse boundary, the solutions smoothen while converging to the said exact shape (see Figure \ref{conv:auglang}(A)).  

If we compare it with the approximate solutions generated by the divergence-free method, we can see that the solutions to the divergence-free method can be imagined as results of applying pressure to an inflated balloon. Meanwhile, the approximations using the augmented Lagrangian method are obtained by slicing a {\it bread} into portions.  This situation is also reflected on the volume preserving capacity of the two methods,  as slicing does not conserve the volume of the {\it bread} while squeezing/applying pressure retains the volume of the fluid/air in a balloon.  
This observation can also be witnessed in Figure \ref{conv:auglang}(C), where the volume preservation for the augmented Lagrangian method is much less efficient for higher values of $h$, while the volume for the shapes generated by the divergence-free method seems to be consistently close to the initial volume regardless of the mesh size.

As for the quantitative convergence of the shapes -- by the virtue of Hausdorff measure  -- we can see in Figure \ref{conv:auglang} (B) that the distance between the exact and approximate solutions tends to go smaller as the mesh size approaches zero, hence the convergence witnessed in Figure \ref{conv:auglang} (A) is quantitatively verified.

{For our last set of illustrations, we shall show the experimental order of convergence of the solutions of the state and adjoint equations, and of the deformation fields, all of which are evaluated at the final shapes.}

{For $k=0,1,2,3$ we shall consider the following mesh sizes $h_k = 10^{-1}\times 2^{-k}$. Furthermore, we shall respectively denote by ${\bu}_k$, ${\bv}_k$, and ${\bth}_k$ the solutions to the state and adjoint equations, and the deformation fields at the approximate shape solution $\Omega_k$ when the triangulation has a mesh size $h_k$. Meanwhile, we shall denote by ${\bu}^*$, ${\bv}^*$, and ${\bth}^*$ the counterparts of the quantities previously mentioned but solved on the {\it exact} solution $\Omega^*$. }

{For an approximate state ${\bphi}_k$ and an exact state ${\bphi}^*$, the experimental order of convergence is solved by computing the following quantity
\[ eoc_k = \log_2\left(\frac{\|{\bphi}_{k-1} - {\bphi}^* \|_{H^m(D_{k-1})}}{\|{\bphi}_k - {\bphi}^* \|_{H^m(D_k)}} \right),\] 
where $D_k$ is the discretization of the hold-all domain $D$ with mesh size $h_k$ and contains the nodes of the approximate solution $\Omega_k$. We note that it is imperative to measure the difference ${\bphi}_{k} - {\bphi}^*$ in $D_k$ to make sure that the approximation ${\bphi}_{k}$ is well-defined, while the exact solution is interpolated according to the nodes of the discretized hold-all domain.  We also mention that the impetus for solving the {\it eoc} for the deformation fields ${\bth}_k$, is that these fields can be looked at as the controls for an optimal control problem. In this way, we verify that the error estimates for our type of controls does not agree with the usual error estimates for optimal control problems\footnote{For example, in \cite{troltzsch2010} (which can be made as reference for most optimal control problems with second order elliptic states), the error estimates for the control are proportional to the error estimates for the state and adjoint variables.}, which is due to the evolving domain. }

{\begin{table}[h!]
\caption{Experimental order of convergence using divergence-free method.}
\centering
\begin{tabular}{|c| c  c c c|}
\hline
&& $eoc_k: \|{\bu}_k - {\bu}^* \|_{H^m(D_k)}$ &&		\\[0.5ex]
\hline
 & $k=0$& $k=1$& $k=2$& $k=3$\\
 \hline
 $m = 0$ & -- & 0.9513 &	1.552 &	1.112\\
 $m = 1$& --  & 0.2880	&	0.5129	&	0.6507\\
\hline
&& $eoc_k: \|{\bv}_k - {\bv}^* \|_{H^m(D_k)}$ &&		\\[0.5ex]
\hline
 & $k=0$& $k=1$& $k=2$& $k=3$\\
 \hline
 $m = 0$ & -- & 1.102	&	1.531	&	1.006\\
 $m = 1$& --  & 0.2640	&	0.9270	&	0.2554\\
 \hline
&& $eoc_k: \|{\bth}_k - {\bth}^* \|_{H^m(D_k)}$ &&		\\[0.5ex]
\hline
 & $k=0$& $k=1$& $k=2$& $k=3$\\
 \hline
 $m = 0$ & -- & 0.4747	&	1.392	&	0.8346\\
 $m = 1$& --  & 0.2296	&	0.8490	&	0.4484\\
 \hline
\end{tabular}
\label{eoc:divfree}
\end{table}}

{From Table \ref{eoc:divfree}, we see that using the divergence free method, the error reduction for the state, adjoint and the deformation fields using the $L^2$-norm is at most of order $3/2$, while for the $H^1$-norm is at most linear.  Unlike the error estimates for elliptic problems (or even the Stokes/Navier--Stokes equations) with fixed domain where the convergence rate can be easily obtained as quadratic and linear for the $L^2$ and $H^1$ norms, respectively.  }

{\begin{table}[h!]
\caption{Experimental order of convergence using the augmented Lagrangian method.}
\centering
\begin{tabular}{|c| c  c c c|}
\hline
&& $eoc_k: \|{\bu}_k - {\bu}^* \|_{H^m(D_k)}$ &&		\\[0.5ex]
\hline
 & $k=0$& $k=1$& $k=2$& $k=3$\\
 \hline
 $m = 0$ & -- & 1.148	&	1.917	&	 1.350\\
 $m = 1$& --  & 0.2036	&	0.7188	&	1.385\\
\hline
&& $eoc_k: \|{\bv}_k - {\bv}^* \|_{H^m(D_k)}$ &&		\\[0.5ex]
\hline
 & $k=0$& $k=1$& $k=2$& $k=3$\\
 \hline
 $m = 0$ & -- & 1.100	&	1.889	&	1.290\\
 $m = 1$& --  & 0.4196	&	0.5684	&	1.433\\
 \hline
&& $eoc_k: \|{\bth}_k - {\bth}^* \|_{H^m(D_k)}$ &&		\\[0.5ex]
\hline
 & $k=0$& $k=1$& $k=2$& $k=3$\\
 \hline
 $m = 0$ & -- & 1.022	&	1.542	&	1.928\\
 $m = 1$& --  & 0.3090	&	0.7255	&	1.116\\
 \hline
\end{tabular}
\label{eoc:auglang}
\end{table}}

{Using the augmented Lagrangian method on the other hand yields an almost quadratic and linear error reduction for the deformation fields on the $L^2$ and $H^1$ norms, respectively.  This convergence may be attributed to the {\it slicing} phenomena for approximating the free-boundary that we talked about earlier. However, the irregularity and slow convergence that was apparent in the divergence-free method may still be witnessed in the convergence rates for the state and adjoint solutions.}

{With these observations, we conclude this section by giving some possible reason for the {\it anomaly} with the above convergence rates. We note that the approximate shape solution $\Omega_k$ is not necessarily -- or even close to -- the discretization of the exact shape $\Omega^*$\footnote{Meaning, $\Omega_k$ is not the union of a triangulation of $\Omega^*$.}.  If we look at Figure \ref{conv:divfree}(A) for instance, we can see that the approximate boundary $\Gamma_{{\rm f},h}$ when $h=1/10$ is very much distinct from the free-boundary $\Gamma_{{\rm f},e}$ of the exact shape solution.  For this reason, the usual error estimates with respect to domain discretizations may not hold.  In fact, $\Omega_k$ is the solution to the fully discretized problem
\begin{align}
	\min_{\Omega_k}\mathcal{G}_k(\Omega_k) \text{ subject to system }\eqref{newtonaprrox}
\end{align}
where all the integrals including that of the objective functional are evaluated using Gaussian quadrature.  Now, if the exact shape $\Omega^*$ is discretized with same mesh size $h = 10^{-1}\times2^{-k}$, which we denote by $\Omega^*_k$, and solve the Navier--Stokes equations numerically and denote its solution as ${\bu}_k^*$, then we can achieve $\|{\bu}^* - {\bu}_k^\star \|_{H^m(\Omega^*)} = \mathcal{O}(h^{2-m})$\footnote{If we denote by ${\bth}^\star_k$ the deformation field solved in $\Omega_k^*$ one can show that $\|{\bth}^* - {\bth}_k^\star \|_{H^m(\Omega_k)}$ enjoys the same convergence rate (see for example \cite{murai2013}).} as an order of convergence. Furthermore,  we infer from using triangle inequality that
\begin{align} 
\|{\bu}^* - {\bu}_k \|_{H^m(D)} \le \|{\bu}^* - {\bu}^*_k \|_{H^m(D)} + \|{\bu}_k^* - {\bu}_k \|_{H^m(D)}.  
\label{femdiss}
\end{align}
Since the first term on the right hand side of \eqref{femdiss} can be estimated by $\mathcal{O}(h^{2-m})$, we conjecture that the second term is what impedes us to fully realize the rate of convergence that is enjoyed by solutions for usual optimal control problems. }

%%%%%%%%%%%%%%%%%%%%%%%%
%                       Conclusion							 %
%%%%%%%%%%%%%%%%%%%%%%%%

\section{Conclusion}\label{sect6}

To summarize, we started the exposition by introducing an artificial boundary condition to make sure of the existence of the solutions to the Navier--Stokes equations. We later on showed the existence of states and shape solutions. The improved regularity of the state solution based on an additional regularity assumption on the domain was also briefly  discussed.  {During the course of the discussions,  we also unraveled some advantages of the outflow boundary condition we considered as opposed to the artificial conditions considered in the literature. Such properties are summarized in the Table \ref{table:bdyconds}.}
\begin{table}[h!]
\caption{Comparison of the artificial outflow conditions.}
\centering
\begin{tabular}{|c | c | c | c | c |}
\hline 
B.C. & Existence\footnote{Existence of Solutions to the Navier--Stokes equations} & Uniqueness\footnote{Uniqueness of Solution to the Navier--Stokes equations} &  Adjoint problem (AP) & Numerical treatment of the (AP)\\
 \hline
 \eqref{donothing} & No assurance & Vacuously Holds & Possible\footnote{See \cite{kasumba2012} for such formulation. Note also that theoretical existence of solution is not guaranteed.} & Possible\\
 \hline
 \eqref{conboundcond} & Theorem \ref{th:wp} & Under condition \eqref{est:uniqueness} & See \eqref{adjoint} and \eqref{strong_adjoint} & Easily Managed\\
  \hline
 \eqref{dirdonothing} & See \cite{braack2014} and \cite{bruneau1996}& See \cite{braack2014} and \cite{bruneau1996} & Quite challenging\footnote{see Remark \ref{remark:endof3}} & Quite challenging\footnote{see Remark \ref{remark:endof3}}\\
 \hline
\end{tabular}
\label{table:bdyconds}
\end{table}

We proceeded by formulating the Eulerian derivative of the objective functional.  In the shape derivative, the volume constraint seems to have been neglected which compelled us to adapt an augmented Lagrangian method and to use a class of deformation fields that satisfy the divergence-free property. The final shapes yielded by the two methods were then compared. We observed that the volume constraint is more satisfied when the latter method is utilized. However, we also saw that this method needed much more iterations as compared to the former. 

We ended by showing the convergence of numerical solutions to a manufactured solution based on the domain discretization. The convergence was measured in terms of the Hausdorff distance, and of the $H^1$ and $L^2$ norms of the state and adjoint solutions, and of the deformation fields. We found out that the experimental convergence rates do not reflect the same convergence rates of typical finite element error estimates for elliptic problems defined in fixed domains. With this in mind, we close this exposition by recommending to future authors the formulation of the theoretical error estimates.  In particular, and if possible, we hope to see the development of estimating the gap between the states ${\bu}_k^\star$ and ${\bu}^*$.

%%%%%%%%%%%%%%%%%%%%%%%%
%                             Appendix							 %
%%%%%%%%%%%%%%%%%%%%%%%%
	\appendix
	\section{}
	This section is dedicated to expose the proof some of the properties that were briefly mentioned that we deemed to be too important to just ignore.
	
	We start with the proof of Lemma \ref{midg} which will be used -- aside from the existence of solutions --  in the subsequent propositions.
	
	\noindent{\it Proof of Lemma \ref{midg}.}
	The proof for the existence of the element ${\bw}_0\in H^2(\Omega)^2$ can be established using the same arguments as in \cite[Lemma IV.2.3]{girault1986}.  
	In particular, we note that the assumption that $({\bg},{\bn})_{\partial\Omega} = 0$, implies that there exists ${\bu}_0 \in H^1(\Omega)^2$ such that $\dive{\bu}_0 = 0$ in $\Omega$, and ${\bu}_0|_{\partial\Omega} = {\bg}$. From \cite[Theorem I.3.1]{girault1986}, we see that ${\bu}_0 = \nabla\times\phi$ for some stream function $\phi\in H^2(\Omega)^2$.  Furthermore,  we can choose a particular stream function $\phi$ that is zero on the boundary $\Gamma_{\rm w}$. We shall then define ${\bw}_{0,\varepsilon}\in H^1(\Omega)^2$ as ${\bw}_{0,\varepsilon} = \nabla\times(\theta_\varepsilon \phi)$, where $\theta_\varepsilon\in C^2(\overline{\Omega})$ -- thanks to \cite[Lemma IV.2.4]{girault1986} -- is the function that satisfies 
	\begin{align*}
		\left\{
			\begin{aligned}
				&\theta_\varepsilon = 1 &&\text{in a neighborhood of }\Gamma_0,\\
				&\theta_\varepsilon(x) = 0 &&\text{for }d(x,\Gamma_0) \ge 2e^{-1/\varepsilon},\\
				&|\partial\theta_\varepsilon/\partial x_i| \le \varepsilon/d(x,\Gamma_0) &&\text{for }d(x,\Gamma_0) \le 2e^{-1/\varepsilon}.
			\end{aligned}
		\right.
	\end{align*}
	
	Futhermore, using using the same arguments as in \cite[Lemma IV.2.3]{girault1986} one can show that 
	\begin{align}
	 \|v_iw_j\|_{L^2(\Omega)}\le c_\varepsilon|v_i|_{H^1(\Omega)}
	 \label{l2est}
	 \end{align}
	where $v_i$ and $w_j$, for $i,j=1,2$, are such that ${\bv}=(v_1,v_2)\in H_{\Gamma_0}^1(\Omega)^2$ and ${\bw}_{0,\varepsilon}=(w_1,w_2)$, and the constant $c_\varepsilon>0$ is dependent on $\varepsilon>0$ in such a way that $c_\varepsilon \to 0$ as $\varepsilon\to0$.
	Now,  from Lemma \ref{propb}(2) we get 
	\begin{align*}
		|({\bv}\cdot\nabla{\bw}_{0,\varepsilon},{\bv})_{\Omega} -  \frac{1}{2}({\bv}\cdot{\bn}, &{\bv}\cdot{\bw}_{0,\varepsilon})_{\Gamma_{\rm out}}|  = |({\bv}\cdot\nabla{\bv},{\bw}_{0,\varepsilon})_{\Omega}- \frac{1}{2}({\bv}\cdot{\bn}, {\bv}\cdot{\bw}_{0,\varepsilon})_{\Gamma_{\rm out}}|\\
		&\le \left| \sum_{i,j=1}^2 \int_{\Omega} v_iw_j\frac{\partial v_j}{\partial x_i} \du x\right| +  \frac{1}{2}\left| \int_{\Gamma_{\rm out}} ({\bv}\cdot{\bn})({\bv}\cdot{\bw}_{0,\varepsilon}) \du s \right|.
	\end{align*}
	
	Using \eqref{l2est}, the first expression on the last line of the previous computation can be estimated as follows:
	\begin{align}
		\left| \sum_{i,j=1}^2 \int_{\Omega} v_iw_j\frac{\partial v_j}{\partial x_i} \du x\right| & \le c_{1,\varepsilon}\|\bv\|_{\bV}^2.\label{est:trili}
	\end{align}
	where $ c_{1,\varepsilon} > 0$ is such that $c_{1,\varepsilon} \to 0$ as $\varepsilon\to0$.
	
	As for the boundary integral,  our goal is to show that 
	\begin{align}
	 \|{\bw}_{0,\varepsilon}\|_{L^2(\Gamma_{\rm out})} \le c_{2,\varepsilon}.\label{ineq:w0}
	 \end{align}
	 where $c_{2,\varepsilon}\to 0$ as $\varepsilon\to0$.
	By doing so, we can show -- with the help of the Rellich-Kondrachov embedding theorem -- that
	\begin{align}
		\begin{aligned}
		\left| \int_{\Gamma_{\rm out}} ({\bv}\cdot{\bn})({\bv}\cdot{\bw}_{0,\varepsilon}) \du s \right| \le&\, c\|({\bv}\cdot{\bn}){\bv}\|_{L^2(\Gamma_{\rm out})}\|{\bw}_{0,\varepsilon} \|_{L^2(\Gamma_{\rm out})}\\ 
		\le&\,  c\|{\bv}\|_{L^4(\Gamma_{\rm out})^2}^2\|{\bw}_{0,\varepsilon}\|_{L^2(\Gamma_{\rm out})^2}\\
		\le&\,  c_{3,\varepsilon}\|{\bv}\|_{{\bV}}^2.
		\end{aligned}\label{est:tribdy}
	\end{align}
	where $c_{3,\varepsilon} \to 0$ as $\varepsilon\to 0$.
	
	To establish \eqref{ineq:w0}, we shall refer to the following Hardy inequality \cite[Theorem 330]{hardy1934}:
	\begin{lmm}
		Let $p> 1$, and $\delta \in (0,\infty]$. For any $u\in W^{1,p}(0,\delta)$ such that $u(\delta) = 0$, the following inequality holds:
		\begin{align*}
			\int_0^\delta \frac{|u|^p}{t^p}\du t \le \left| \frac{p}{p-1} \right|^p\int_0^\delta {|u'|^p}\du t.
		\end{align*}
		\label{lemma:hardy}
	\end{lmm}
	
	Indeed, from the properties of $\theta_\varepsilon$, by denoting $\Gamma_{\rm out}^{\varepsilon} := \{s\in\Gamma_{\rm out}: d(s,\Gamma_0) \le 2e^{-1/\varepsilon} \}$, we get
	\begin{align}
		\begin{aligned}
		\|{\bw}_{0,\varepsilon}\|_{L^2(\Gamma_{\rm out})^2}  & = \left( \int_{\Gamma_{\rm out}^\varepsilon} |{\bw}_{0,\varepsilon}|^2 \du s \right)^{1/2}\\
			& \le c \left( \int_{\Gamma_{\rm out}^\varepsilon}  \varepsilon^2\frac{|\phi|^2}{d(s,\Gamma_0)^2} + \left|\nabla\phi\right|^2 \du s \right)^{1/2}\\
			& \le c_1\varepsilon\left( \int_{\Gamma_{\rm out}^\varepsilon}  \frac{|\phi|^2}{d(s,\Gamma_0)^2} \du s \right)^{1/2} + c_2\|\nabla\phi\|_{L^2(\Gamma_{\rm out}^\varepsilon)^2}
		\end{aligned}	\label{est:w0}
	\end{align}
	
	Due to the regularity assumption on the domain $\Omega$ and since $\phi = 0$ on $\Gamma_{\rm w}$ which is adjacent to the boundary $\Gamma_{\rm out}^{\varepsilon}$,  we can write the boundary integral as the one-dimensional integral
	\[ \int_0^{\delta(\varepsilon)} \left|\frac{\phi(t)}{t}\right|^2 \du t, \]
	where $\delta(\varepsilon)$ corresponds to the arc length of the boundary $\Gamma_{\rm out}^{\varepsilon}$. From Lemma \ref{lemma:hardy}, we get 
	\[ \int_0^{\delta(\varepsilon)} \left|\frac{\phi(t)}{t}\right|^2 \du t \le 4  \int_0^{\delta(\varepsilon)} \left|\phi'(t)\right|^2 \du t. \]
	
	This implies that the estimate \eqref{est:w0} can further be estimated as 
	\begin{align}
		\begin{aligned}
		\|{\bw}_{0,\varepsilon}\|_{L^2(\Gamma_{\rm out})^2}  & \le (c_1\varepsilon + c_2)\|\nabla\phi\|_{L^2(\Gamma_{\rm out}^\varepsilon)^2}.
		\end{aligned}	
	\end{align}
	Lastly, since $\|\nabla\phi\|_{L^2(\Gamma_{\rm out}^\varepsilon)^2} \to 0$ as $\varepsilon\to 0$, we can choose $ c_{2,\varepsilon} := (c_1\varepsilon + c_2)\|\nabla\phi\|_{L^2(\Gamma_{\rm out}^\varepsilon)^2}$.  

	From \eqref{est:trili} and \eqref{est:tribdy},  we get 
	\begin{align*}
		|({\bv}\cdot\nabla{\bw}_{0,\varepsilon},{\bv})_{\Omega} -  \frac{1}{2}({\bv}\cdot{\bn}, &{\bv}\cdot{\bw}_{0,\varepsilon})_{\Gamma_{\rm out}}|  \le \left(c_{1,\varepsilon} + \frac{c_{3,\varepsilon}}{2}\right)\|{\bv}\|_{{\bV}}^2.
	\end{align*}
	Since $\left(c_{1,\varepsilon} + \frac{c_{3,\varepsilon}}{2}\right)\to 0$ as $\varepsilon\to 0$, we can choose $\varepsilon>0$ small enough so that $\left(c_{1,\varepsilon} + \frac{c_{3,\varepsilon}}{2}\right)\le \gamma$, and with this choice of $\varepsilon$ we take ${\bw}_{0} = {\bw}_{0,\varepsilon}$. \qed

%	This implies that for sufficiently small $\varepsilon$, ${\bw}_0 = 0$ on $\Gamma_{\rm out}$. This in turn makes the boundary integral be equal to zero. Therefore,
%	\[ |b({\bv};{\bw}_0,{\bv})_{\Omega} - \frac{1}{2}({\bv}\cdot{\bn}, {\bv}\cdot{\bw}_0)_{\Gamma_{\rm out}}|\le \gamma\|{\bv}\|^2_{\bV(\Omega)}\quad \forall\bv\in\bV(\Omega).\]\qed
	
%	
%	
%	Similarly, using the Rellich-Kondrachov embedding, Poincar{\'e} inequality and \eqref{l2est}, we get
%	\begin{align*}
%		\left| \int_{\Gamma_{\rm out}} ({\bv}\cdot{\bn})({\bv}\cdot{\bw}_0) \du s \right| \le c\|{\bv}\cdot{\bn}\|_{L^2(\Gamma_{\rm out})}\|{\bv}\cdot{\bw}_0 \|_{L^2(\Gamma_{\rm out})}\le c_{2,\gamma}\|\bv\|_{\bV}^2
%	\end{align*}
%%	\begin{align*}
%%		\left| \int_{\Gamma_2} ({\bv}\cdot{\bn})({\bv}\cdot{\bw}) \du s \right| & \le c\|{\bv}\cdot{\bn}\|_{L^2(\Gamma_2)}\|{\bv}\cdot{\bw} \|_{L^2(\Gamma_2)}\\
%%		& \le c\|{\bv}\|_{H^1(\Omega)^2}\|{\bv}\cdot{\bw} \|_{H^1(\Omega)}\\
%%		& \le c\|{\bv}\|_{H^1(\Omega)^2}\|{\bv}\cdot{\bw} \|_{L^2(\Omega)}
%%		\le c_{2,\gamma}\|\bv\|_{\bV}^2
%%	\end{align*}
%	Here, the constants $c_{1,\gamma},c_{2,\gamma}>0$ can be chosen such that $\max\{c_{1,\gamma},c_{2,\gamma}/2\} \le \gamma.$\qed

	\begin{prpstn}[Uniqueness of the Navier--Stokes solution]
	Let the assumptions in Theorem \ref{th:wp} hold, and assume that \eqref{est:uniqueness} holds.
	Then the solution $\tilde{\bu}\in{\bV}$ of \eqref{weak} is unique.
	\label{prop:unique}
\end{prpstn}
\begin{proof}
Suppose that we have two solutions $\tilde{\bu}_1,\tilde{\bu}_2\in{\bV}$ for \eqref{weak}. By definition,  for any ${\bphi}\in{\bV}$ and $i=1,2$, we have
\begin{align}
	\mathbb{A}(\tilde{\bu}_i;\tilde{\bu}_i,{\bphi}) - \frac{1}{2}\mathbb{C}(\tilde{\bu}_i;\tilde{\bu}_i,{\bphi}) = \langle \Phi, {\bphi}\rangle_{\bV'\times\bV}.
	\label{wNSi}
\end{align}
Furthermore,  by following the proof of Theorem \ref{th:wp} we have, for any $i=1,2$, the estimate
\begin{align}
	\|\tilde{\bu}_i\| \le \frac{2}{\nu}\|\Phi\|_{V'}.
	\label{solnorm}
\end{align}
By taking the difference of \eqref{wNSi}, and by utilizing Lemma \ref{propb}(iii), the element ${\bw}=\tilde{\bu}_1-\tilde{\bu}_2\in{\bV}$ satisfies 
\begin{align}
\begin{aligned}
	\nu\|{\bw}\|_{\bV}^2 + ({\bw}\cdot\nabla{\bg},{\bw})_\Omega - \frac{1}{2}({\bw}\cdot{\bn},{\bg}\cdot{\bw})_{\Gamma_2} =\frac{1}{2}({\bw}\cdot{\bn},\tilde{\bu}_1\cdot{\bw})_{\Gamma_2}-({\bw}\cdot\nabla\tilde{\bu}_1,{\bw})_\Omega.
	\end{aligned}
\label{eqnforw}
\end{align}
Furthermore, by introducing the notation 
\[ \mathcal{N} = \sup_{\bu,\bv,\bw}\frac{|b(\bu;\bv,\bw)_{\Omega}-\frac{1}{2}({\bu}\cdot{\bn},{\bv}\cdot{\bw})_{\Gamma_{\rm out}} |}{\|\bu\|_{\bV(\Omega)}\|\bv\|_{\bV(\Omega)}\|\bw\|_{\bV(\Omega)}}, \]
the right-hand side of \eqref{eqnforw} can be further be bounded using the estimate \eqref{solnorm}. Indeed,  we have 
\begin{align}
\begin{aligned}
\frac{1}{2}({\bw}\cdot{\bn},\tilde{\bu}_1\cdot{\bw})_{\Gamma_2}-({\bw}\cdot\nabla\tilde{\bu}_1,{\bw})_\Omega  \le \mathcal{N}\|\tilde{\bu}_1\|_{\bV}\|{\bw}\|_{\bV}^2 \le \frac{2}{\nu}\mathcal{N}\|\Phi\|_{\bV'}\|{\bw}\|_{\bV}^2.
\end{aligned}
\label{RHS}
\end{align}
Meanwhile, the left-hand side may be estimated from below by utilizing Lemma \ref{gprop} with $\gamma=\nu/2$, and this yields - together with \eqref{RHS} -
\begin{align*}
	(\nu^2 - 4\mathcal{N}\|\Phi\|_{\bV'})\|{\bw}\|_{\bV}^2 \le 0.
\end{align*}
Therefore, from the assumption \eqref{est:uniqueness}, ${\bw} = 0$. 
\end{proof}

	Another important property that have been referred to several times is the isomorphism of $\mathbb E'({\bu},p)\in \mathcal{L}(X,X^*)$ where $(\tilde{\bu},p)\in X$ is the unique solution of \eqref{strong_nontranslated}.  By utilizing De Rham's Lemma, we know that this property is equivalent to showing that given $\mathcal{F}\in \bH^{-1}(\Omega)$, the following equation is well-posed
\begin{align}
	\mathbb{A}'(\delta{\bu};\delta{\bu},{\bphi}) - \frac{1}{2}\mathbb{C}'(\delta{\bu};\delta{\bu},{\bphi}) = \langle\mathcal{F},{\bphi} \rangle_{\bV'\times\bV}, \quad\forall{\bphi}\in\bV
	\label{weak:frechet}
\end{align}
	where 
\begin{align*}
	\mathbb{A}'(\delta{\bu};\delta{\bu},{\bphi}) & = \nu a(\delta{\bu},{\bphi})_\Omega + b(\delta{\bu};{\bu},{\bphi})_\Omega +  b({\bu};\delta{\bu},{\bphi})_\Omega, \\
	\mathbb{C}'(\delta{\bu};\delta{\bu},{\bphi}) & = (\delta{\bu}\cdot{\bn},{\bu}\cdot{\bphi})_{\Gamma_{\rm out}} + ({\bu}\cdot{\bn},\delta{\bu}\cdot{\bphi})_{\Gamma_{\rm out}},
\end{align*}
	and $\bu = \tilde{\bu}+{\bg}.$
	
	\begin{prpstn}
		Suppose that the assumptions in Proposition \ref{prop:unique} hold. Then the unique solution $\delta{\bu}\in \bV$  to \eqref{weak:frechet} exists.
	\end{prpstn}
	\begin{proof}
		The proof of the proposition follows the same arguments as that of Theorem \ref{th:wp}, except for showing the coercivity of the left hand side. For that reason, we shall only show here the aforementioned coercivity. 
		
		For any $\delta{\bu}\in\bV$,  Lemma \ref{propb}(3) implies that 
		\begin{align}
		\begin{aligned}
			\mathbb{A}'(\delta{\bu};\delta{\bu},\delta{\bu}) - \frac{1}{2}\mathbb{C}'(\delta{\bu};\delta{\bu},\delta{\bu})= \nu\|\delta{\bu} \|_{\bV}^2	 + b(\delta{\bu};{\bu},\delta{\bu})_\Omega - \frac{1}{2}(\delta{\bu}\cdot{\bn},{\bu}\cdot\delta{\bu})_{\Gamma_{\rm out}}.
		\end{aligned}
		\label{estdummy1}
		\end{align}
		From Lemma \ref{midg},  by using the quantity $\mathcal{N}$,  and since $\|\tilde{\bu}\|_{\bV}\le \frac{2}{\nu}\|\Phi\|_{\bV'}$ we get the following absolute estimate
		\begin{align*}
			|b(\delta{\bu};{\bu},\delta{\bu})_\Omega - \frac{1}{2}(\delta{\bu}\cdot{\bn},{\bu}\cdot\delta{\bu})_{\Gamma_{\rm out}} | 	\le &\, \mathcal{N}\|\tilde{\bu}\|_{\bV}\|\delta{\bu}\|_{\bV}^2 + \frac{\nu}{2}\|\delta{\bu}\|_{\bV}^2\\
			\le&\, \left(\frac{2\mathcal{N}\|\Phi\|_{\bV'}}{\nu} + \frac{\nu}{2}\right)\|\delta{\bu}\|_{\bV}^2
		\end{align*}
%		\begin{align*}
%			|b(\delta{\bu};{\bu},\delta{\bu})_\Omega - \frac{1}{2}(\delta{\bu}\cdot{\bn},{\bu}\cdot\delta{\bu})_{\Gamma_{\rm out}} | \le &\, |b(\delta{\bu};\tilde{\bu},\delta{\bu})_\Omega - \frac{1}{2}(\delta{\bu}\cdot{\bn},\tilde{\bu}\cdot\delta{\bu})_{\Gamma_{\rm out}}|+ |b(\delta{\bu};{\bg},\delta{\bu})_\Omega - \frac{1}{2}(\delta{\bu}\cdot{\bn},{\bg}\cdot\delta{\bu})_{\Gamma_{\rm out}}|\\ 
%			\le &\, \mathcal{N}\|\tilde{\bu}\|_{\bV}\|\delta{\bu}\|_{\bV}^2 + \frac{\nu}{2}\|\delta{\bu}\|_{\bV}^2\\
%			\le&\, \left(\frac{2\mathcal{N}\|\Phi\|_{\bV'}}{\nu} + \frac{\nu}{2}\right)\|\delta{\bu}\|_{\bV}^2
%		\end{align*}
	Using the above estimate to \eqref{estdummy1} yields
	\begin{align*}
		\mathbb{A}'(\delta{\bu};\delta{\bu},\delta{\bu}) - \frac{1}{2}\mathbb{C}'(\delta{\bu};\delta{\bu},\delta{\bu}) & \ge \nu\|\delta{\bu} \|_{\bV}^2 - \left(\frac{2\mathcal{N}\|\Phi\|_{\bV'}}{\nu} + \frac{\nu}{2}\right)\|\delta{\bu}\|_{\bV}^2\\ & = \frac{\nu^2 -4\mathcal{N}\|\Phi\|_{\bV'}}{2\nu}\|\delta{\bu}\|_{\bV}^2.
	\end{align*}
	Lastly, the uniqueness assumption \eqref{est:uniqueness} implies that $\nu^2 -4\mathcal{N}\|\Phi\|_{\bV'} >0$ and therefore yields the coercivity. 
	\end{proof}

%%-----------------------------
%\bibliographystyle{plain}      
%\bibliography{references.bib}   % name your BibTeX data base
%%-----------------------------

\end{document}